\documentclass[11.5pt,a4paper,pointlessnumbers]{scrartcl}

\usepackage[T1]{fontenc}

\usepackage{amssymb}
\usepackage{amsmath}
\usepackage{amsfonts}
\usepackage{bbm}	
\usepackage{amsthm}
\usepackage{lmodern}
\usepackage{mathrsfs}
\usepackage{xcolor}
 \usepackage[pdftex,hidelinks,
     bookmarks         = true,
     bookmarksnumbered = true,
     pdfstartview      = FitH,
     pdfpagelayout     = SinglePage,
     urlcolor          = teal,
     linkcolor          = teal,
     citecolor=teal,
     colorlinks,
 	pdftoolbar = true
     ]{hyperref}                

\usepackage{wrapfig}
\usepackage[margin=2.5cm]{geometry}
\usepackage[all,cmtip]{xy}
\usepackage[utf8]{inputenc}
\usepackage{graphicx}
\usepackage{varwidth}
\usepackage{overpic}

\makeatletter
\DeclareRobustCommand{\em}{%
	\@nomath\em \if b\expandafter\@car\f@series\@nil
	\normalfont \else \slshape \fi}
\makeatother
\usepackage{upgreek}	
\usepackage{rotating}
\usepackage{tikz}
\usetikzlibrary{matrix,arrows,decorations.pathmorphing,shapes.geometric}
\usepackage{tikz-cd}
\usepackage{needspace}
\newcommand{\spaceplease}{\needspace{5\baselineskip}}

\usetikzlibrary{decorations.markings}

\usetikzlibrary{backgrounds}

\usepackage{todonotes}

\usepackage{tikz-cd}

\usetikzlibrary{decorations.markings}
\usetikzlibrary{backgrounds}

\usepackage{etex}

\usepackage{mathtools}

\newtheoremstyle{mytheorem}
{\topsep}
{\topsep}
{\slshape}
{0pt}
{\bfseries}
{.}
{ }
{\thmname{#1}\thmnumber{ #2}\thmnote{ {\normalfont(#3)}}}

\newtheoremstyle{mydefinition}
{\topsep}
{\topsep}
{\normalfont}
{0pt}
{\bfseries}
{.}
{ }
{\thmname{#1}\thmnumber{ #2}\thmnote{ {\normalfont\slshape(#3)}}}

\theoremstyle{mytheorem}
\newtheorem{theorem}{Theorem}[section]

\makeatletter
\newtheorem*{rep@theorem}{\rep@title}
\newcommand{\newreptheorem}[2]{%
	\newenvironment{rep#1}[1]{%
		\def\rep@title{#2 \ref{##1}}%
		\begin{rep@theorem}}%
		{\end{rep@theorem}}}
\makeatother

\newreptheorem{theorem}{Theorem}
\newreptheorem{corollary}{Corollary}

\newtheorem{lemma}[theorem]{Lemma}
\newtheorem{proposition}[theorem]{Proposition}
\newtheorem{corollary}[theorem]{Corollary}
\theoremstyle{mydefinition}\makeatletter
\newtheorem*{repx@theorem}{\repx@title}
\newcommand{\newrepxtheorem}[2]{%
	\newenvironment{repx#1}[1]{%
		\def\repx@title{#2 \ref{##1}}%
		\begin{repx@theorem}}%
		{\end{repx@theorem}}}
\makeatother
\newtheorem{definition}[theorem]{Definition}\newrepxtheorem{definition}{Definition}

\newtheorem{remm}[theorem]{Remark}
\newenvironment{remark}
{\pushQED{\hfill$\lozenge$}\remm}
{\popQED\endremm}
\numberwithin{equation}{section}
\usepackage{enumitem}

\usepackage{dingbat}

\renewcommand{\today}{\ifcase \month \or January\or February\or March\or %
	April\or May \or June\or July\or August\or September\or October\or November\or %
	December\fi {} \number  \year} 

\makeatletter
\newcommand{\monthyeardate}{%
	\DTMenglishmonthname{\@dtm@month}, \@dtm@year
}
\makeatother

\setkomafont{caption}{\small}

\setkomafont{section}{\rmfamily\large}
\setkomafont{subsection}{\rmfamily}
\setkomafont{subsubsection}{\rmfamily}
\setkomafont{subparagraph}{\normalfont\scshape}

\newtheorem*{theorem*}{Theorem}
\newtheorem*{corollary*}{Corollary}

\makeatletter
\renewcommand\section{\@startsection {section}{1}{\z@}%
	{-3.5ex \@plus -1ex \@minus -.2ex}%
	{2.3ex \@plus.2ex}%
	{\normalfont\scshape\centering}}
\makeatother

\usepackage{titlesec}
\titleformat{\subsection}[runin]
{\normalfont\bfseries}
{\thesubsection}
{0.5em}
{}
[.]

\usepackage{float}

\usepackage{titletoc}
\dottedcontents{section}[1em]{\rmfamily}{1.5em}{0.2cm}
\dottedcontents{subsection}[5em]{\rmfamily}{1.9em}{0.2cm}

\usepackage{multicol}

\usepackage[export]{adjustbox}

\usepackage{xcolor}
\definecolor{mygreen}{rgb}{0, 0.471, 0}
\definecolor{myred}{rgb}{0.784, 0, 0}
\definecolor{myorange}{rgb}{0.922, 0.627, 0}
\definecolor{myviolet}{rgb}{0.686, 0.275, 0.882}

\newcommand{\C}{\mathbb{C}}
\newcommand{\R}{\mathbb{R}}
\newcommand{\bigslant}[2]{{\raisebox{.2em}{$#1$}\left/\raisebox{-.2em}{$#2$}\right.}}

\newcommand{\RepqG}{\text{Rep}_q\text{GL}_t}
\newcommand{\RepqH}{\text{Rep}_q\text{L}_{t}}
\newcommand{\RepqP}{\text{Rep}_q\text{P}_{t}}

\newcommand{\Hom}{\text{Hom}}
\newcommand{\End}{\text{End}}
\newcommand{\Rib}{\text{Rib}}
\newcommand{\Rep}{\text{Rep}}
\newcommand{\GL}{\text{GL}}
\renewcommand{\H}{\text{L}}
\renewcommand{\L}{\text{L}}

\renewcommand{\C}{\mathcal{C}}

\newcommand{\Sk}{\text{SkCat}}
\renewcommand{\P}{\text{P}}
\newcommand{\id}{\text{id}}
\newcommand{\Uqgln}{U_q(\mathfrak{gl}_n)}
\newcommand{\Uqglm}{U_q(\mathfrak{gl}_m)}
\newcommand{\Uq}{{U_q(\mathfrak{gl}_N)}}

\newcommand{\rot}{\text{rot}}
\newcommand{\ev}{\text{ev}}
\newcommand{\coev}{\text{coev}}
\newcommand{\Bimod}{\text{Bimod}}
\newcommand{\A}{\mathcal{A}}

\newcommand*{\QEDA}{\null\nobreak\hfill\ensuremath{\blacksquare}}

\newcommand{\orup}{\boldsymbol{\textcolor{myorange}{\uparrow}}}
\newcommand{\upp}{\text{\textcolor{blue}{$\boldsymbol{\uparrow}$}}}
\newcommand{\downn}{\text{\textcolor{blue}{$\boldsymbol{\downarrow}$}}}
\newcommand{\gup}{\text{\textcolor{mygreen}{$\boldsymbol{\uparrow}$}}}
\newcommand{\rup}{\text{\textcolor{myred}{$\boldsymbol{\uparrow}$}}}
\newcommand{\gdown}{\text{\textcolor{mygreen}{$\boldsymbol{\downarrow}$}}}
\newcommand{\rdown}{\text{\textcolor{myred}{$\boldsymbol{\downarrow}$}}}

\newcommand{\Vect}{\textsc{Vect}}
\newcommand{\Cat}{\textsc{Cat}}
\newcommand{\res}{\text{res}_{t}}
\newcommand{\pres}{\text{pRes}}
\newcommand{\Skh}{\text{SkCat}_{\text{L}_t}(S)}
\newcommand{\Skg}{\Sk_{\GL_t}(S)}
\newcommand{\skg}{\text{Sk}_{\GL_t}(S)}
\newcommand{\skh}{\text{Sk}_{\text{L}_t}(S)}

\newcommand{\gl}{\text{gl}}

\newcommand{\op}{\text{op}}

\newcommand{\gsquare}{{\textcolor{mygreen}{\blacksquare}}}
\newcommand{\redsquare}{{\textcolor{myred}{\blacksquare}}}
\newcommand{\orsquare}{{\textcolor{myorange}{\blacksquare}}}
\newcommand{\bsquare}{{\textcolor{blue}{\blacksquare}}}
\newcommand{\vsquare}{{\textcolor{myviolet}{\blacksquare}}}

\usepackage{scalerel}
\DeclareMathOperator*{\bigboxtimes}{\scalerel*{\boxtimes}{\sum}}

\let\emptyset\varnothing

\newcommand{\blueposx}{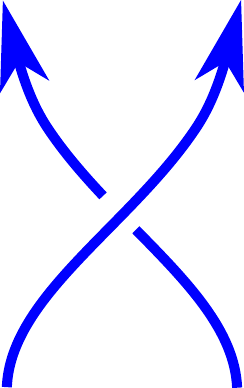}
\newcommand{\bluenegx}{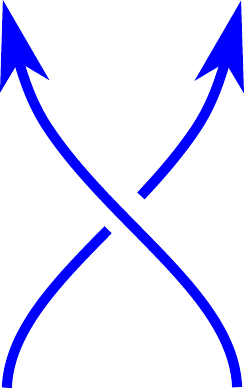}
\newcommand{\blueidid}{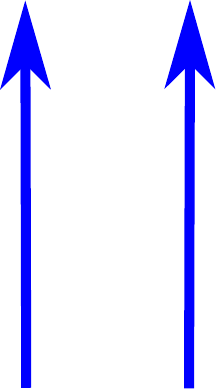}
\newcommand{\bluetwista}{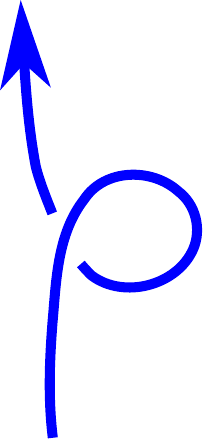}

\newcommand{\RIIIa}{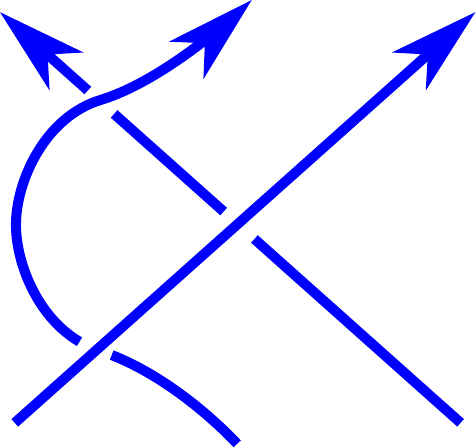}
\newcommand{\RIIIb}{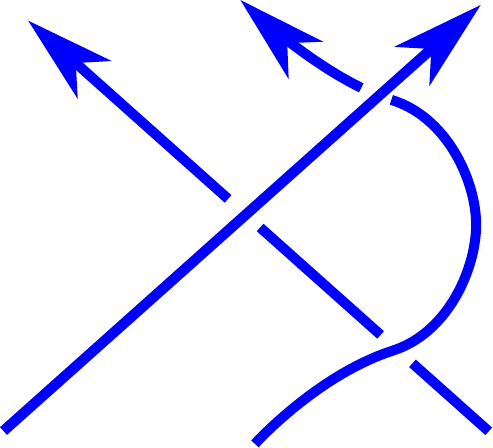}
\newcommand{\blueid}{drawings/blueid.pdf}
\newcommand{\blueiddual}{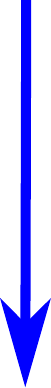}
\newcommand{\bknot}{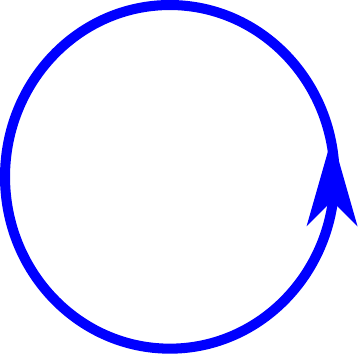}
\newcommand{\blueevv}{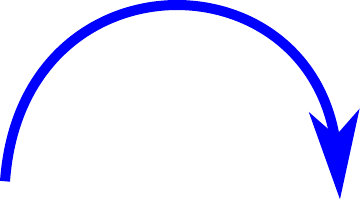}
\newcommand{\bluecoev}{drawings/bluecoev.pdf}
\newcommand{\blueposz}{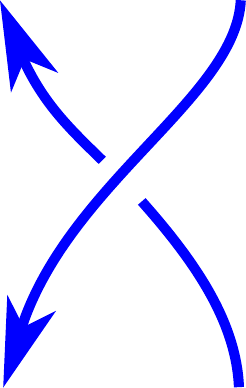}

\newcommand{\zigzaga}{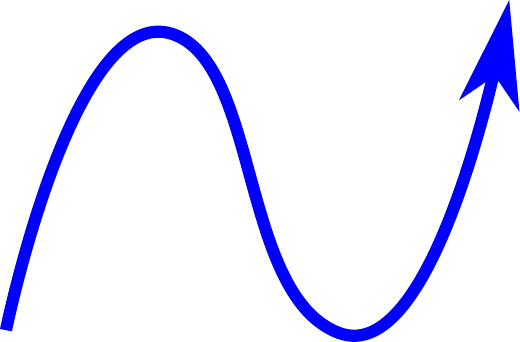}
\newcommand{\zigzagb}{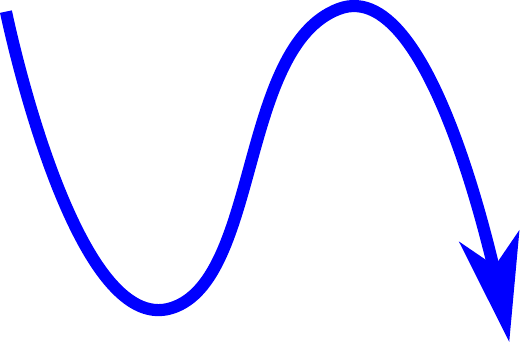}
\newcommand{\blueinvz}{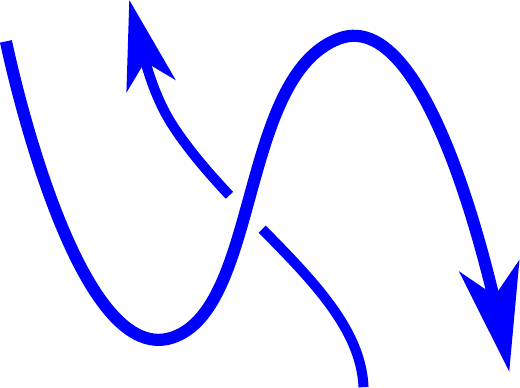}

\newcommand{\bluedimrelation}{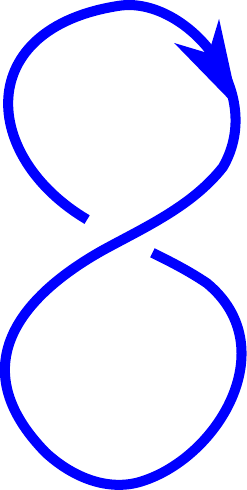}
\newcommand{\evfot}{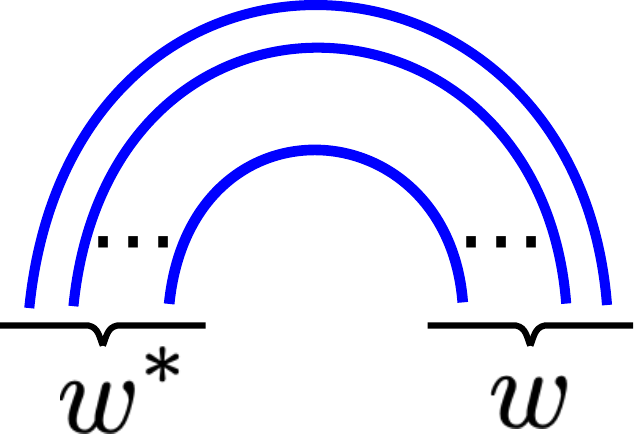}
\newcommand{\coevfot}{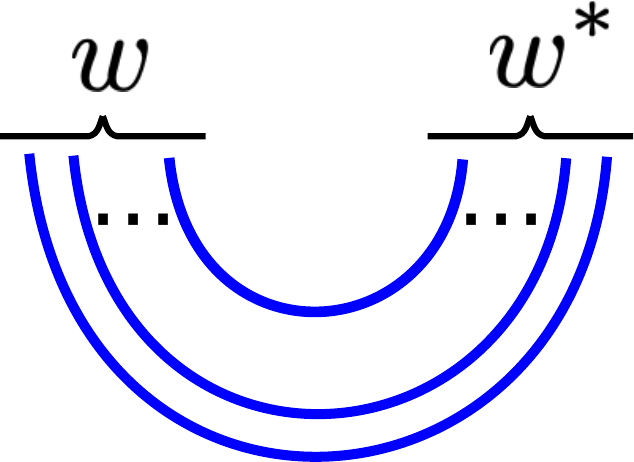}
\newcommand{\twistfot}{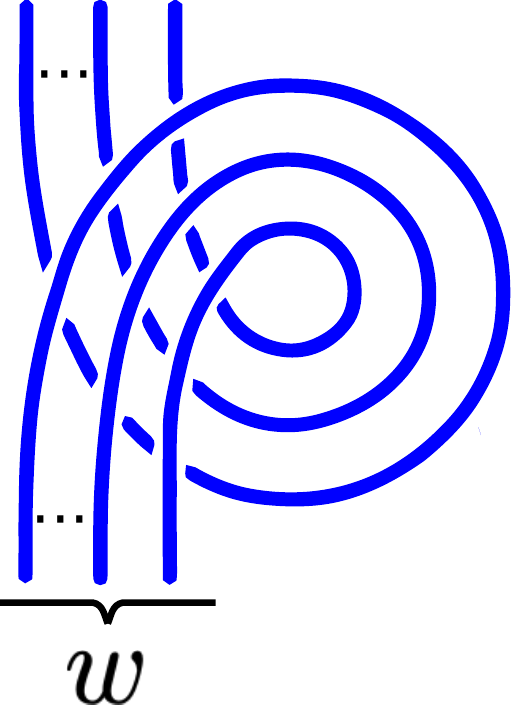}
\newcommand{\bluexx}{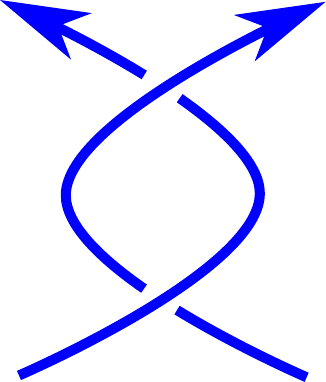}

\newcommand{\boxx}{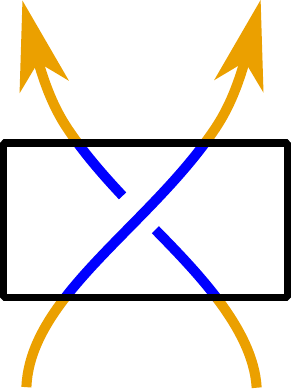}
\newcommand{\boxz}{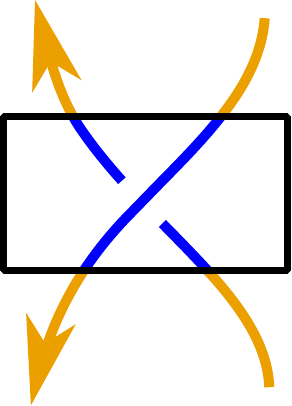}

\newcommand{\boxgooox}{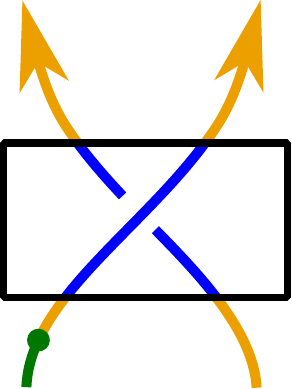}

\newcommand{\boxooorx}{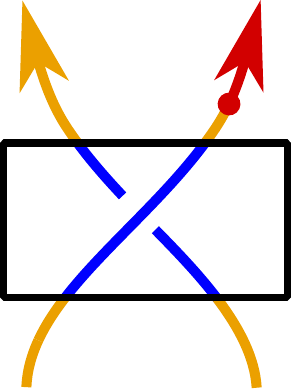}
\newcommand{\ooorx}{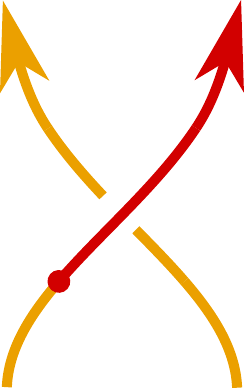}
\newcommand{\boxroooz}{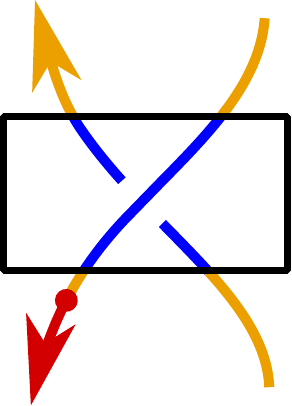}
\newcommand{\roooz}{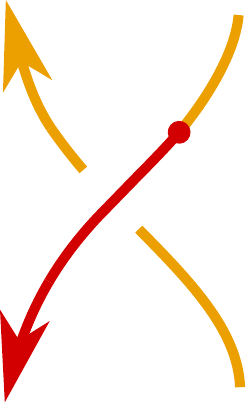}
\newcommand{\boxooogz}{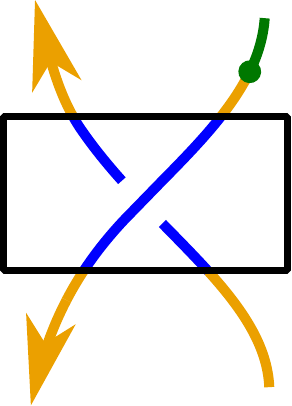}
\newcommand{\ooogz}{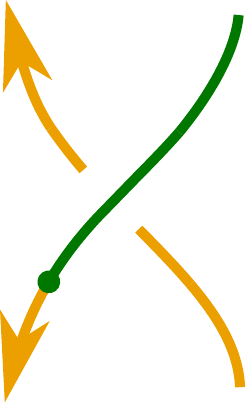}

\newcommand{\boxogoox}{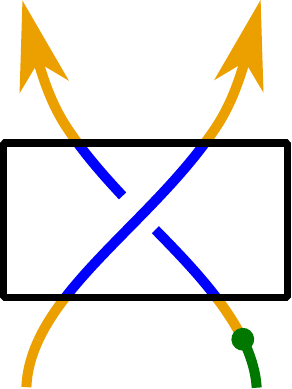}
\newcommand{\ogoox}{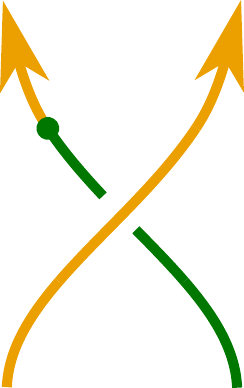}
\newcommand{\boxrooox}{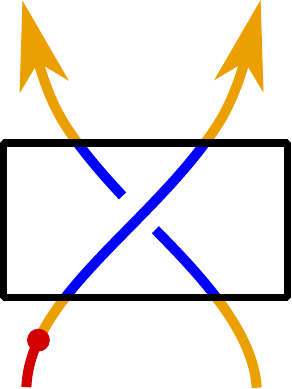}
\newcommand{\rooox}{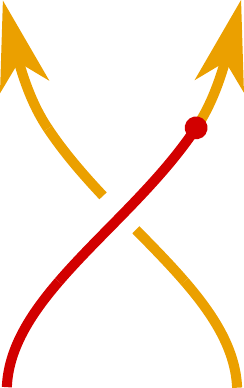}
\newcommand{\boxdim}{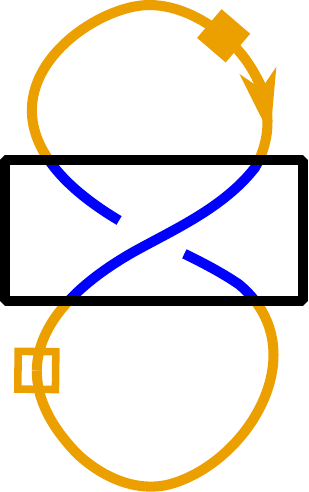}

\newcommand{\gooox}{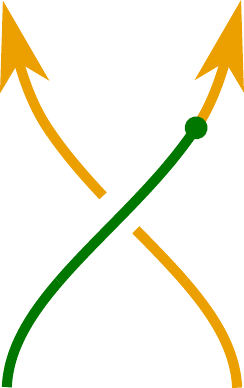}

\newcommand{\oposx}{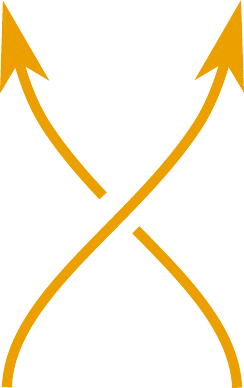}
\newcommand{\ornegx}{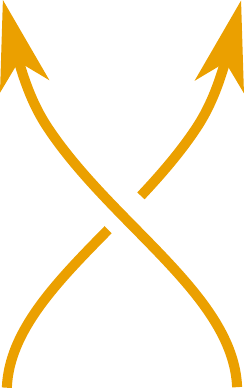}

\newcommand{\ortwist}{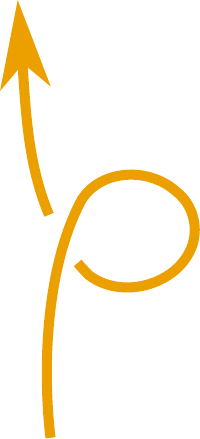}
\newcommand{\orid}{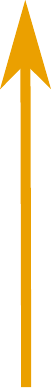}
\newcommand{\orknot}{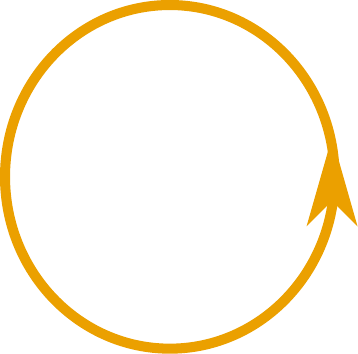}
\newcommand{\orip}{drawings/orip.pdf}
\newcommand{\ogip}{drawings/ogip.pdf}
\newcommand{\ogposx}{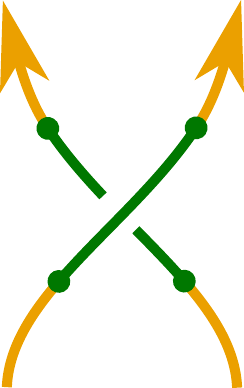}
\newcommand{\orposx}{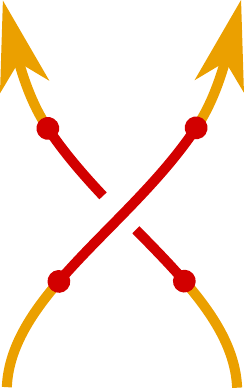}
\newcommand{\ogrx}{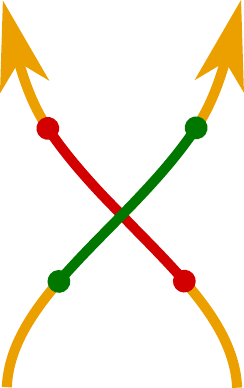}
\newcommand{\orgx}{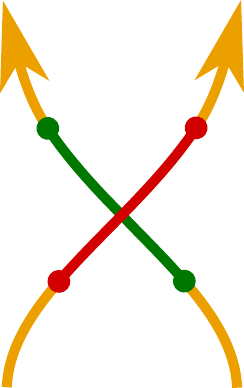}

\newcommand{\ogocoev}{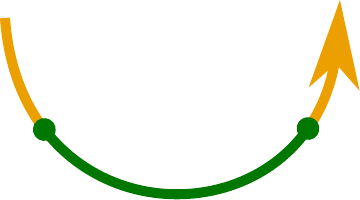}
\newcommand{\orocoev}{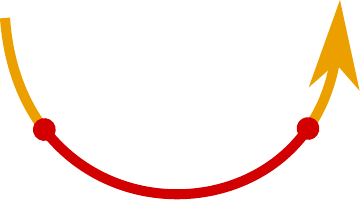}
\newcommand{\ogoevv}{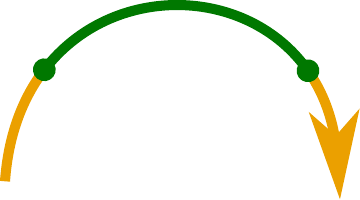}
\newcommand{\oroevv}{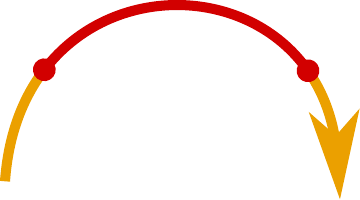}

\newcommand{\ogrevvcoevv}{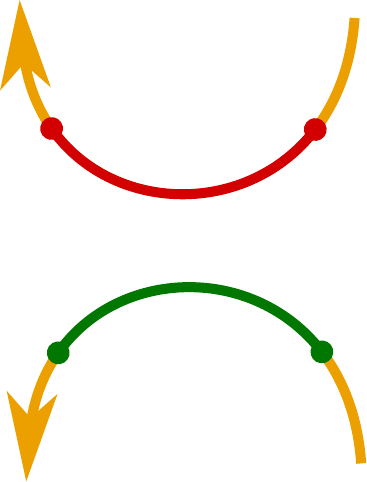}

\newcommand{\ogrevcoev}{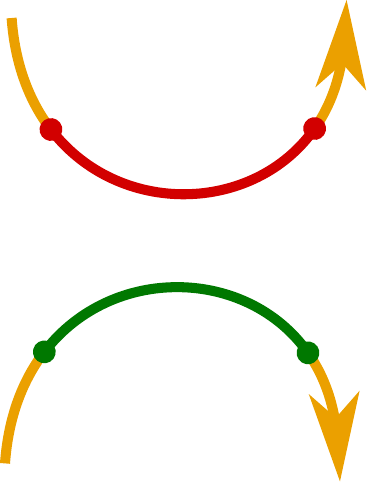}
\newcommand{\oposz}{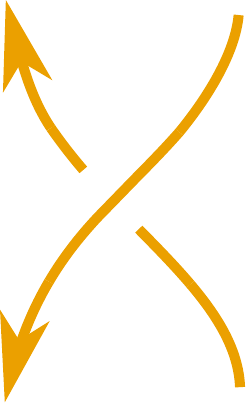}
\newcommand{\ozinv}{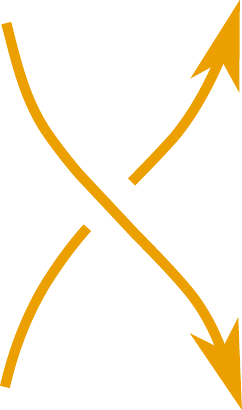}
\newcommand{\oRIII}{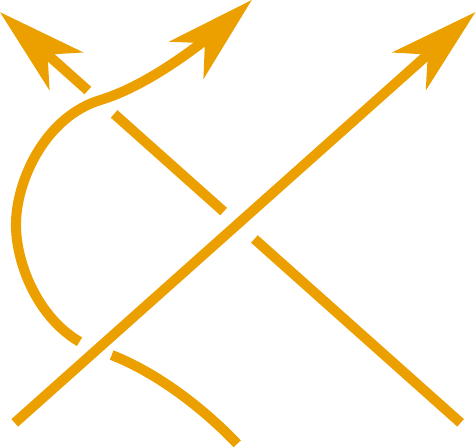}
\newcommand{\orgxx}{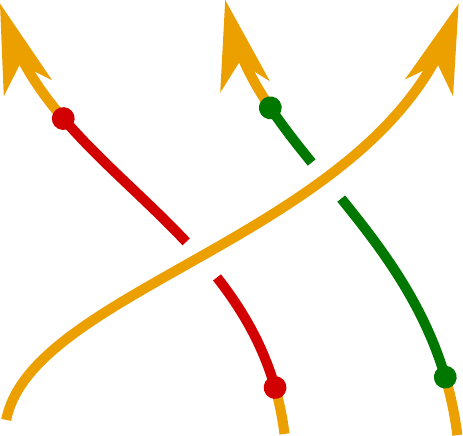}

\newcommand{\ogipip}{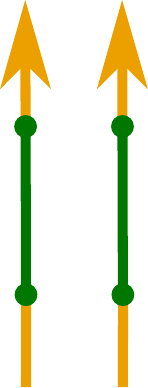}
\newcommand{\oripip}{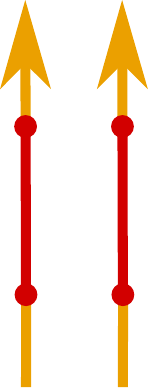}
\newcommand{\orgipip}{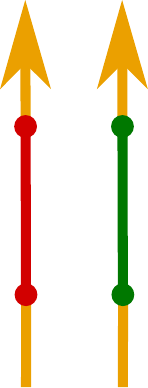}

\newcommand{\gid}{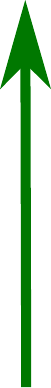}

\newcommand{\grposx}{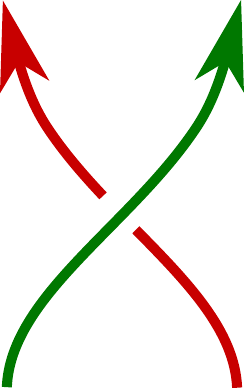}
\newcommand{\grnegx}{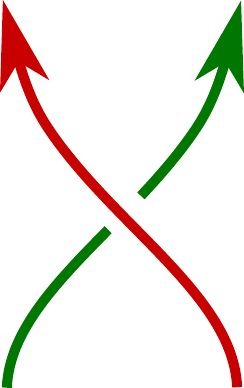}

\newcommand{\ginc}{drawings/ginc.pdf}

\newcommand{\gproj}{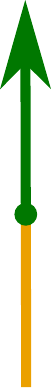}

\newcommand{\ggip}{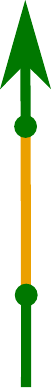}
\newcommand{\grip}{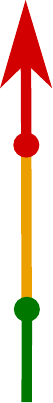}

\newcommand{\gdim}{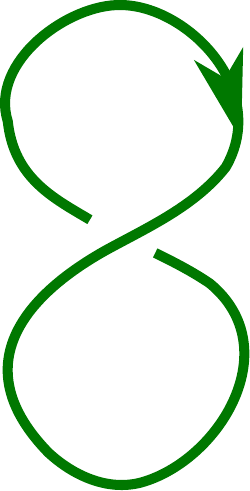}

\newcommand{\rid}{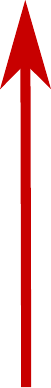}

\newcommand{\rgposx}{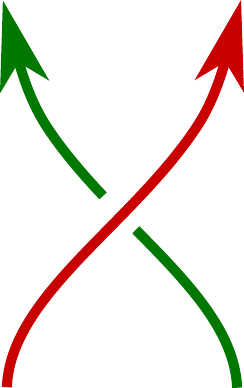}
\newcommand{\rgnegx}{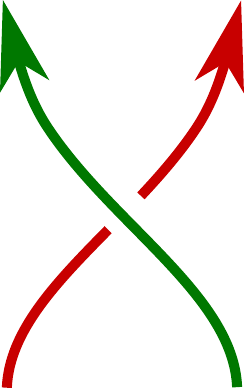}

\newcommand{\rinc}{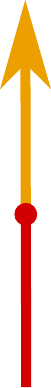}

\newcommand{\rproj}{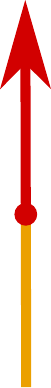}

\newcommand{\rgip}{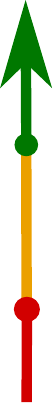}
\newcommand{\rrip}{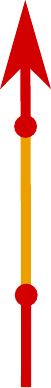}

\newcommand{\rdim}{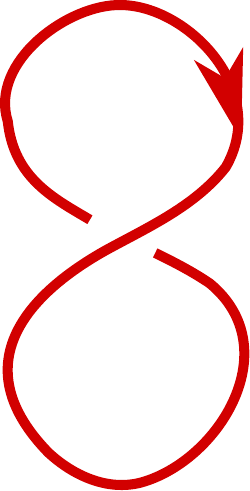}
\newcommand{\rgknot}{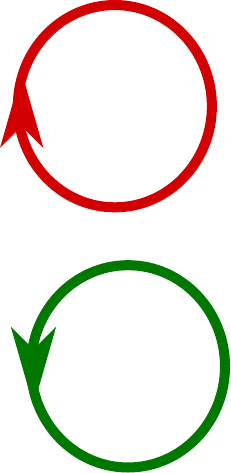}
\newcommand{\rgoxx}{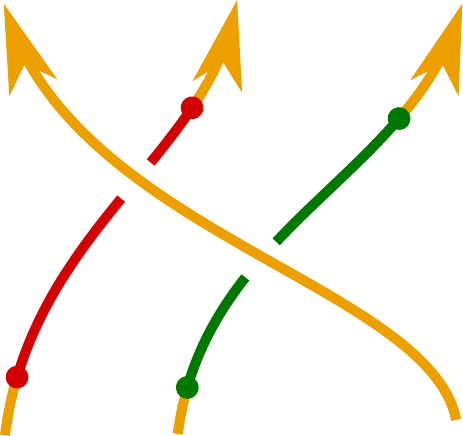}

\newcommand{\blackposx}{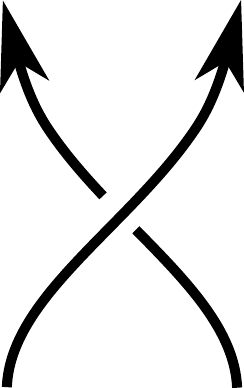}
\newcommand{\blacknegx}{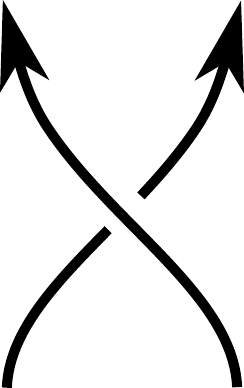}
\newcommand{\blackidid}{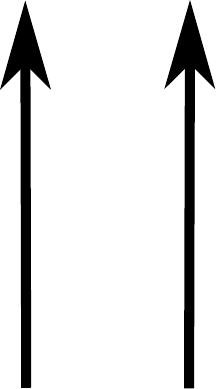}
\newcommand{\blacktwist}{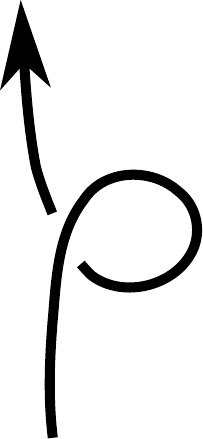}
\newcommand{\blackdim}{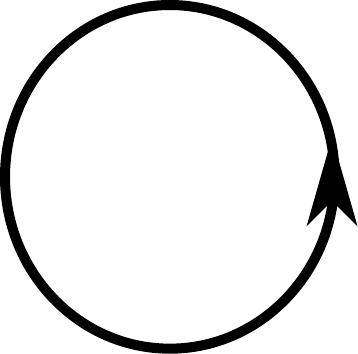}
\newcommand{\blackid}{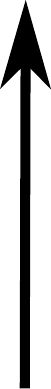}

\newcommand{\MyFigure}[1]{\;\includegraphics[scale=0.2, valign=c]{#1}\;}

\begin{document}

\begin{center}	
	\textbf{\Large{HOMFLY parabolic restriction, defect skein theory and the Turaev coproduct}\\} 
	\vspace{0.6cm}	{\large Juan Ramón \textsc{Gómez García}} \\ \vspace{0.3cm} {Institut de Mathématiques de Jussieu-Paris Rive Gauche}\\ UMR 7586\\ Université Paris Cité, Sorbonne Université, CNRS\\ F-75013 Paris\\ France
\end{center}	
\vspace{0.3cm}

\begin{abstract}
\noindent \textbf{Abstract.} We define a HOMFLY version of the category $\Rep_q\P$ of quantum representations of a parabolic subgroup $\P\subseteq\GL_{m+n}$ of block triangular matrices. Alongside this category, we construct functors that interpolate the usual restriction functors between $\GL_{m+n}$, $\P$ and the subgroup $\L\subseteq\GL_{m+n}$ of block-diagonal matrices, yielding a universal version of the formalism of parabolic restriction. Based on this formalism, we construct central algebras and centred bimodules which serve as algebraic ingredients for defining a skein theory on $3$-manifolds with surface and line defects. We recover the Turaev coproduct on the HOMFLY skein algebra as a particular instance of this theory. In particular, this coproduct is compatible with the cutting and gluing of surfaces.
 \end{abstract}

	\tableofcontents
	\normalsize

	\spaceplease

\section{Introduction}

\subsection{Context and motivation} Shortly after the discovery of the Jones polynomial \cite{jones}, Reshetikhin and Turaev \cite{rtribbon} developed a framework unifying and generalising polynomial invariants of links. Their work showed that the algebraic structure of \emph{ribbon categories} is governed by the topology of so-called \emph{ribbon graphs}, opening the door to a fruitful interaction between low-dimensional topology and representation theory. More precisely, given a ribbon category $\A$, Reshetikhin and Turaev defined a functor 
\begin{equation}\label{eval_intro}
\text{RT}_\A\colon\Rib_\A\to\A
\end{equation}
from a category $\Rib_\mathcal{A}$, whose morphisms are graphs decorated with objects and morphisms of $\A$, to the underlying ribbon category $\A$ itself. It assigns to any knot or link an endomorphism of the monoidal unit $\bold{1}_\A$, which is invariant under isotopy. 

The evaluation functor $\text{RT}_\A$ is full but not faithful. Its kernel is described by \emph{skein relations}, which are the topological counterpart of the algebraic relations in $\A$. Prototypical examples of ribbon categories are categories of representations of quantum groups: if $\text{G}$ is a reductive affine algebraic group and $q\in\Bbbk^\times$ is sufficiently generic (which we will always assume), the category $\Rep_q\text{G}$ is equivalent to the category $\Rep\text{ G}$ of algebraic representations of $\text{G}$, but has a non-trivial ribbon structure which is a deformation of the (trivial) one on $\Rep\;\text{G}.$ For $\text{G}=\GL_N$, skein relations depend in a simple way on the dimension $N\in\mathbb{N}$. Replacing $N$ with a formal parameter $t$ in these relations, one obtains the \emph{HOMFLY relations} \begin{equation}\label{homflyrels}
\MyFigure{\blueposx}-\MyFigure{\bluenegx}=\left(q-q^{-1}\right)\MyFigure{\blueidid}, \qquad
\MyFigure{\bluetwista}=q^t\MyFigure{\blueid},\qquad
\MyFigure{\bknot}=\frac{q^t-q^{-t}}{q-q^{-1}}\;\bold{1}_\emptyset,
\end{equation}
where $\bold{1}_\emptyset$ is the empty diagram. The \emph{HOMFLY category} $\RepqG$ is then defined as the category whose morphisms are tangles in $\R^2\times[0,1]$ modulo the HOMFLY relations. This category is a quantum analogue of the Deligne category $\Rep\;\GL_t$ \cite{deligne,comes} and has been extensively studied in \cite{brundan}. It is a universal version of $\Rep_q\GL_N$ in the sense that the latter can be obtained as (a quotient of) the specialization of $\RepqG$ at $t=N$.

The polynomial link invariant obtained for $\A=\RepqG$ via the Reshetikhin--Turaev construction is the \emph{HOMFLY polynomial} \cite{homfly, pt}. Write $H_{q,t}(D)$ for the HOMFLY polynomial of a link diagram $D$. In \cite{jaeger}, Jaeger gave a composition formula allowing one to compute $H_{q,t_1+t_2}(D)$ as a state sum over $H_{q,t_1}(D_1)$ and $H_{q,t_2}(D_2)$, where $D_1,D_2$ are some subdiagrams of $D$. Concretely, a link diagram $D$ is a $4$-valent graph with two distinguished classes of vertices: positive and negative crossings. Given a vertex $v$, we denote by $a$ (resp. $b$) the upper (resp. lower) incoming edge, and by $c$ (resp. $d$) the upper (resp. lower) outgoing edge, as shown in the following pictures: $$\MyFigure{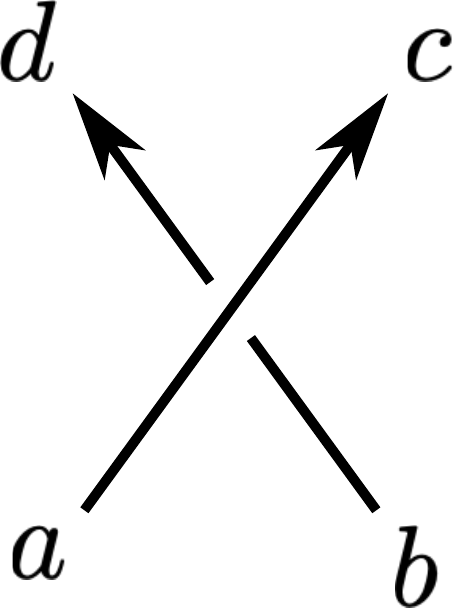}\hspace{5cm}\MyFigure{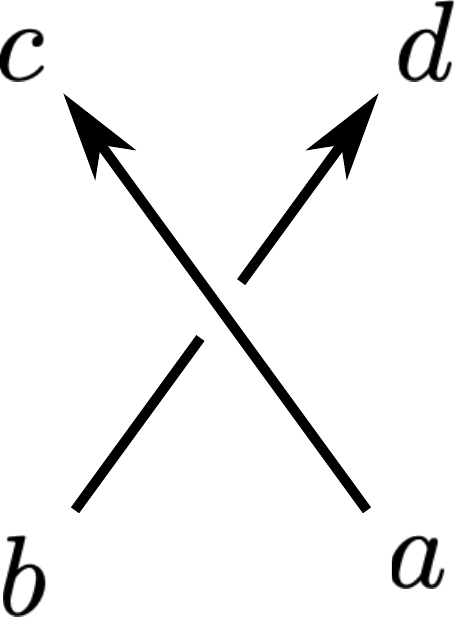}.$$ Let $E_D$ be the set of edges of $D$. An \emph{admissible labelling} is a map $f:E_D\to\{1,2\}$ such that, at any vertex $v$, either $f(a)=f(c)$ and $f(b)=f(d)$, or $f(a)=f(d)>f(b)=f(c).$ When the second condition holds, we say that $v$ is a \emph{cutting vertex}. Given an admissible labelling $f$ and a vertex $v$, we set $$\left\langle v\;|\;f\right\rangle\coloneqq\left\lbrace\begin{array}{ll}
\text{sgn}(v)(q-q^{-1}), & \text{if $v$ is a cutting vertex},\\
1, & \text{otherwise},
\end{array}\right.$$ where $\text{sgn}(v)$ is $1$ if $v$ is a positive crossing and $-1$ otherwise. We define the \emph{interaction of $D$ with $f$} as $$\left\langle D\;|\;f\right\rangle\coloneqq\prod_{v}\left\langle v\;|\;f\right\rangle,$$ where the product runs over the set of vertices of $D$.

\begin{theorem*}[\cite{jaeger}]
For any link diagram $D$, 
\begin{equation}
H_{q,t_1+t_2}(D)=\sum_{\text{$f$ admissible}}\left\langle D\;|\;f\right\rangle\left(q^{-t_1}\right)^{r(D_{f,1})}\left(q^{t_2}\right)^{r(D_{f,2})}H_{q,t_1}(D_{f,1})H_{q,t_2}(D_{f,2}),
\end{equation}
where $D_{f,i}\coloneqq f^{-1}(i)$ and $r(D_{f,i})$ is the rotation number of $D_{f,i}$, for $i\in\{1,2\}.$ \QEDA
\end{theorem*}

This formula can be interpreted as a universal version of the restriction of the natural representation of $\GL_{m+n}$ to $\GL_m\times\GL_n$. Roughly, we have the following picture: 
\begin{figure}[H]
\centering
\includegraphics[scale=0.55]{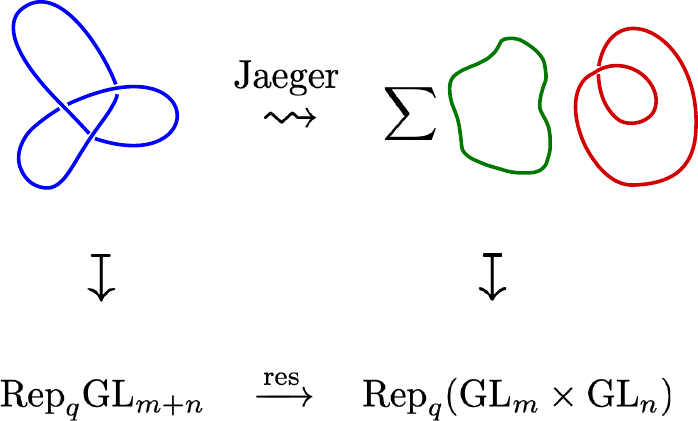}
\caption{Jaeger's formula is a universal version of restriction along $\GL_{m+n}\hookrightarrow\GL_m\times\GL_n.$}
\end{figure}

Inspired by Jaeger's construction, Turaev showed in \cite{turaev} the existence of a coproduct on the \emph{HOMFLY skein algebra} which inspired this paper. Skein theory deals with the study of modules, algebras and categories arising when one tries to generalise quantum link invariants to arbitrary $3$-manifolds. Given a $3$-manifold $M$ and a ribbon category $\mathcal{A}$, the \emph{skein module} $\text{Sk}_\mathcal{A}(M)$ has as generators the isotopy classes of links in $M$ modulo the local/skein relations induced by $\mathcal{A}$ through the Reshetikhin--Turaev construction. They were first introduced in \cite{tur88, prz}. When $M=S\times[0,1]$ for an oriented surface $S$, a multiplication is defined by stacking two copies of the cylinder and rescaling the second coordinate, yielding the notion of the \emph{skein algebra} of the surface $S$. Allowing also open components (i.e. \emph{tangles}) leads to the definition of \emph{skein categories}. Given an oriented surface $S$, the category $\text{SkCat}_{\A}(S)$ \cite{cookethesis,jf,walker} has  configurations of framed coloured points on $S$ as objects and linear combinations of $\A$-coloured tangles in $S\times[0,1]$ as morphisms.

We set $\text{Sk}_{\GL_t}(S)$ for the \emph{HOMFLY skein algebra} of $S$, obtained by taking $\A=\RepqG$. Since the restriction functor is monoidal, Jaeger's formula extends to a \emph{linear} map $$\Delta_f\colon\skg\to\skg\otimes\skg$$ for any framed surface $(S,f).$ This functor is not braided, however, so there is \emph{a priori} no reason for this map to be multiplicative. Yet, Turaev showed the following:

\begin{theorem*}[\cite{turaev}]
Let $(S,f)$ be a framed surface. The linear map
\begin{equation}
\Delta_f\colon\text{Sk}_{\GL_t}(S)\to\text{Sk}_{\GL_t}(S)\otimes\text{Sk}_{\GL_t}(S)
\end{equation}
defined by extending Jaeger's formula is a morphism of algebras. It is coassociative, so $\text{Sk}_{\GL_t}(S)$ is a bialgebra. \QEDA
\end{theorem*}

To illustrate the issue, consider the following picture:
\begin{figure}[H]
\centering
\includegraphics[scale=0.55]{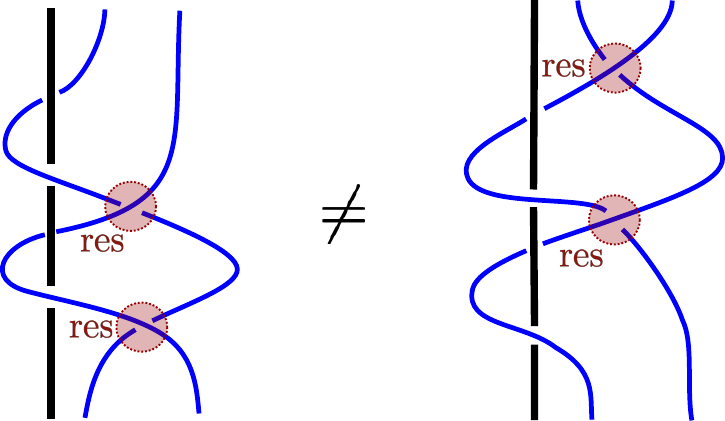}
\caption{The underlying tangles are isotopic. However, because the restriction functor is not braided, it does not preserve crossings, so the resulting graphs no longer represent the same morphism.}
\end{figure}
The naive extension of the restriction functor to the skein category is \emph{not} compatible with composition of tangles. In this paper, we extend Turaev's coproduct to skein categories (i.e. to tangles) using the formalism of \emph{parabolic} restriction and skein theory with defects. Our construction is functorial, explaining the multiplicativity of Turaev's construction. Moreover, it is compatible in a non-obvious way with cutting and gluing surfaces and provides a very local way of computing the coproduct on a given diagram.

\subsection{HOMFLY categories and parabolic restriction} Our first contribution is to define a HOMFLY version of the category $\Rep_q\P$ of quantum representations of the parabolic subgroup $\P\subseteq\GL_{m+n}$ of block triangular matrices (see Section \ref{HOMFLYcats}). Namely, let $\text{G} =\GL_{m+n}$ and $\L=\GL_m\times\GL_n$ be identified with the Levi subgroup of $\text{G}$ consisting of block diagonal matrices. We define the diagrammatic category $\RepqP$ as a certain subcategory of the rigid category underlying $\RepqH=\RepqG\boxtimes\RepqG$, the HOMFLY analogue of $\Rep_q\L$. Together with $\RepqP$, we construct four functors 
\begin{equation*}\label{allthefunctors}
\begin{tikzcd}
                                                             & \RepqP \arrow[rd, "j_t^*"'] &                                                    \\
\RepqG \arrow[ru, "{\iota_t^*}"] \arrow[rr, "\res"] &                             & \RepqH \arrow[lu, "\pi^*_t"', bend right]
\end{tikzcd}, 
\end{equation*} interpolating the ordinary restriction functors between the corresponding categories of representations. In Section \ref{parabolic}, we prove the following (see Theorem \ref{central_functor}):

\begin{theorem}
The functor $\iota_t^*\otimes\pi_t^*$ lifts to a braided monoidal functor $$ \RepqG\boxtimes(\overline{\RepqH})\to Z(\RepqP),$$ where $\overline{(-)}$ stands for the opposite braided category and $Z(\RepqP)$ is the Drinfeld centre of $\RepqP$. Hence, $\RepqP$ is a $(\RepqG,\RepqH)$-central algebra in the sense of \cite{BJS, dgno}. \QEDA
\end{theorem}

In addition to this central algebra structure, there is another important structure underlying our construction. The category $\RepqH$ is a $(\RepqH,\RepqP)$-bimodule category with action of $\RepqP$ induced by the restriction functor $\pi_t^*.$ As we explain in Section \ref{bimodule} this bimodule structure is centred with respect to the central structures of $\RepqP$ and $\RepqH.$ 

\subsection{Defect skein theory} Central algebras and centred bimodules are, respectively, the 1-morphisms and the 2-morphisms of the Morita $4$-category $\textsc{BrTens}$ of cp-rigid braided tensor categories studied in \cite{BJS}. Brochier, Jordan and Snyder showed that this category is $3$-dualizable, hence by the cobordism hypothesis \cite{lurie, baezdolan}, its objects determine $3$-extended framed TFTs attaching categories to surfaces and functors/vector spaces to three-manifolds. Skein theory is a particular consequence of the existence of this TFT \cite{brochier1, jennyben, cookethesis, benjamin}.\footnote{Strictly speaking, \cite{BJS} only defines a framed TFT whose construction relies on the cobordism hypothesis. The skein-theoretic approach bypasses these requirements, allowing for a direct construction of the theory.}

Within this framework, the central algebra $\RepqP$ provides a natural transformation between the $\RepqG$ and the $\RepqH$ theories that can be described skein-theoretically using the formalism of defect skein modules developed in \cite{jennydavid}. Let $M$ be a bipartite $3$-manifold, that is, a $3$-manifold together with a stratum of codimension 1 (cf. Definition \ref{bipartite}). The $3$-dimensional regions on each side of the surface defect are labelled with $\RepqG$ and $\RepqH$ and the surface defect itself is labelled with $\RepqP$. The \emph{parabolic skein module} of the bipartite $3$-manifold $M$ is then spanned by isotopy classes of graphs embedded in $M$ whose components are coloured according to the region in which they lie. Local relations around the surface defect are produced by projecting from the $3$-dimensional regions, replacing the crossings that appear by the corresponding half-morphisms using the central algebra structure and evaluating in $\RepqP.$

Our second contribution extends this formalism by introducing defects of codimension 2 decorated with centred bimodule structures. In Section \ref{planar_theories}, we describe planar theories with defect lines decorated by $\RepqH$. The local model around the defect is induced by the $(\RepqH,\RepqP)$-bimodule structure of $\RepqH$. This bimodule structure is centred with respect to the central algebra structures of $\RepqP$ and $\RepqH$ (see Section \ref{centred_bimodule} for the details), so these theories lift to $3$-dimensional theories as explained in Section \ref{3d_theories} (see Theorem \ref{3TFT}):

\begin{theorem}
There exists a stratified 2-dimensional framed TQFT $$\mathcal{Z}\colon\text{Bord}_2^{\text{bip,mkd}}\to\textsc{Bimod}$$ assigning $\RepqH$ (resp. $\RepqP$) to an interval labelled with $\L_t$ (resp. $\P_t$). To a bipartite framed surface $S$, it assigns a bimodule whose components are vector spaces spanned by stratified ribbon graphs in $S\times I$, modulo the local relations described above. \QEDA
\end{theorem}

\begin{figure}[H]
\centering
\includegraphics[scale=0.25]{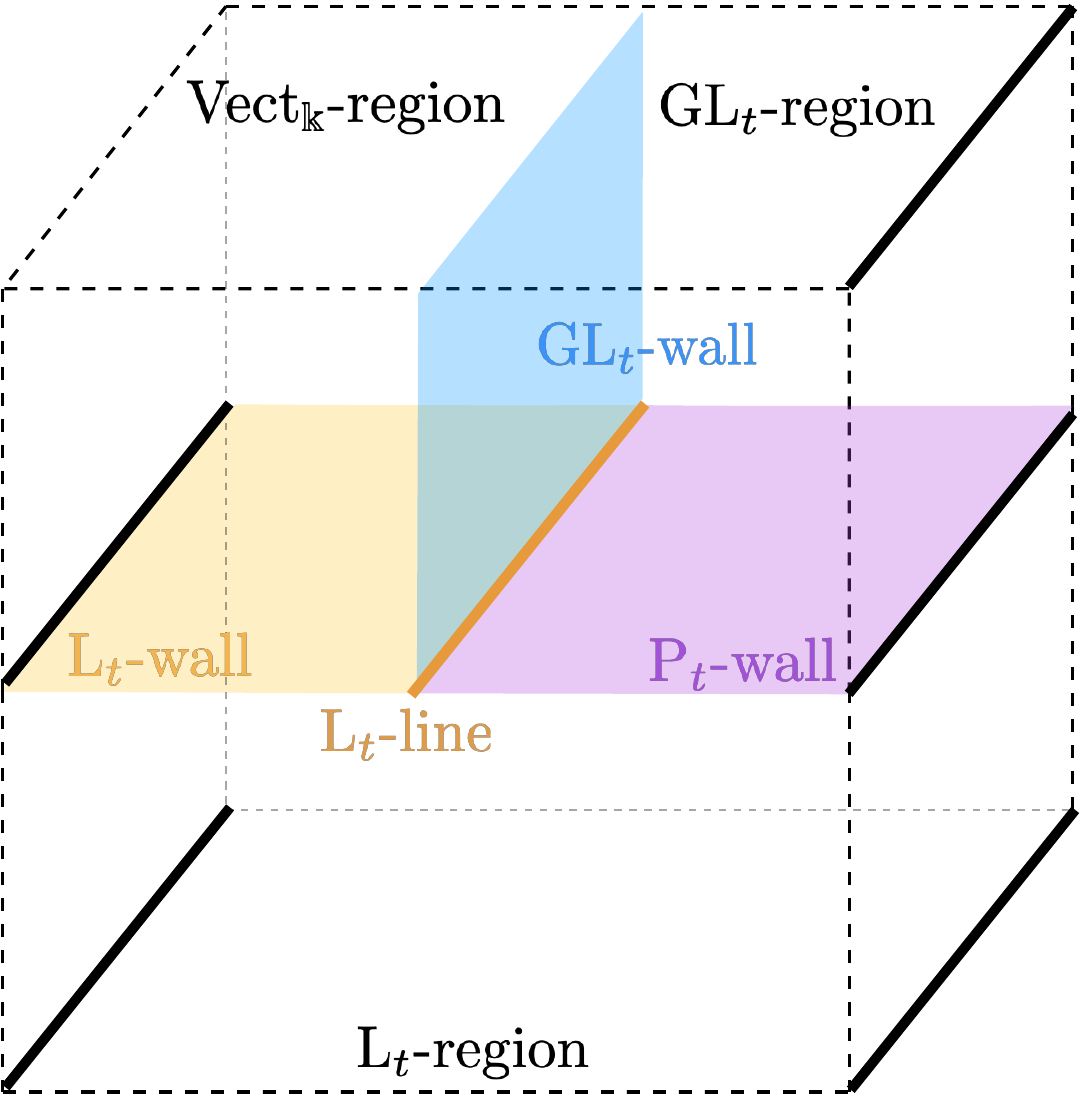}
\caption{Local model around the line defects.}
\end{figure}

\subsection{Turaev's coproduct} Finally, in the last section, we explain how our formalism recovers and extends Turaev's construction in a way which is manifestly compatible with composition of tangles and gluing of surfaces. Taking a well-chosen stratification of a surface $S$ and empty boundary conditions for the TFT above, we get an action of the $\GL_t$-skein algebra on the $\H_t$-skein algebra of $S$ (see Section \ref{coproductsection}):

\begin{theorem}
Let $(S,f)$ be a framed surface. A stratification can be chosen so that $\mathcal{Z}(S)(\emptyset)$ is a $(\skg,\skh)$-bimodule isomorphic to the skein algebra $\skh$ as a vector space. \QEDA
\end{theorem}

The action of $\skg$ on the empty diagram induces the Turaev coproduct $$\Delta_f\colon\skg\to\skg\otimes\skg$$ on the HOMFLY skein algebra. The construction is compatible with gluing surfaces and the functoriality implies automatically that $\Delta_f$ is a morphism of algebras.

\subsection{Outlook} In \cite{jennydavid}, Brown and Jordan study a quantization of the $A$-polynomial using the framework of parabolic defect skein modules. Their construction relies on a module over the boundary skein algebra, defined via parabolic restriction for $\mathrm{SL}_2$, rather than a direct map between algebras. Specialising our construction to $t=2$ (and modulo the reduction from $\mathrm{GL}_2$ to $\mathrm{SL}_2$), we recover the structure of the parabolic central algebra appearing in their work. In contrast to their bimodule approach, we can apply repeatedly Turaev's coproduct to construct a morphism of algebras
$$\Delta^N_f \colon \skg \to \skg \otimes \cdots \otimes \skg.$$
Specialising to $t=N$ on the left and to $t_1=\dots=t_N=1$ on the right (see Remark \ref{specialistationL}), we obtain a genuine morphism (rather than a bimodule)
\begin{equation*}\label{coproduct_tori}
\widetilde{\Delta}_f^N \colon \text{Sk}_{\mathrm{GL}_N}(S) \to \text{Sk}_{\mathbb{C}^N}(S)
\end{equation*}
between the $\mathrm{GL}_N$ and the $\mathbb{C}^N$-skein algebras of the surface. This provides a more natural setting for the quantization of the $A$-polynomial. 
 
On the other hand, the HOMFLY skein algebra provides a canonical quantization of the Goldman--Turaev Lie bialgebra \cite{turaev}, which is constructed over the vector space spanned by homotopy classes of loops on a surface $S$. This Lie bialgebra is equipped with a canonical filtration and it is indeed \emph{formal} \cite{alekseeva,alekseevb}, i.e., isomorphic to its associated graded Lie bialgebra. Moreover, recent works (\emph{loc. cit.}) stablish a suprising connection between the Goldman--Turaev Lie bialgebra, the Kashiwara--Vergne conjecture \cite{kv} and the Kontsevich integral \cite{massuyeau}.  For instance, when $S$ is a ``pair of pants'', the formality assertion is equivalent to the KV conjecture. The associated graded of the HOMFLY skein algebra has been described in \cite{schedler}. We expect that Schedler's coproduct admits a construction based on chord diagrams and similar to the one we give in this paper for the Turaev coproduct. This should be a first step towards proving a quantum version of the Goldman--Turaev formality, for which compatibility with cutting and gluing surfaces will likely play a key role.

\section*{Acknowledgements}
I would like to express my gratitude to my advisor Adrien Brochier for the ideas, the discussions and his constant involvement. I also thank my second advisor Emmanuel Wagner for his support, and Jennifer Brown for having shared with me the draft of her paper that inspired some of my results.

\section{Background}

\subsection{A quick review of rigid categories}\label{rigid_categories} We recall in this section some facts about rigid categories that will be useful throughout the paper. A \emph{monoidal category} is a category $\C$ endowed with a functor $\otimes: \mathcal{C}\times \mathcal{C}\to\mathcal{C}$ which is associative and unital up to given natural isomorphisms. Every monoidal category is equivalent to a \emph{strict} monoidal one, so we will assume that $\otimes$ is associative and unital on the nose.

\begin{definition}
A monoidal category $\mathcal{C}$ is \emph{right rigid} if every object $X$ has a right dual $X^*$, i.e., there are morphisms $$\text{ev}_X:X\otimes X^*\to\boldsymbol{1}\qquad\text{and}\qquad\text{coev}_X:\boldsymbol{1}\to X^*\otimes X$$ satisfying the usual \emph{zig-zags identities.} Similarly, we say that $\C$ is \emph{left rigid} if every object $X$ has a left dual ${}^*X$ and we say that it is \emph{rigid} if it is both left and right rigid. 
\end{definition}

Duals are unique up to canonical isomorphism: if $(X^*,\text{ev}_X,\text{coev}_X)$ and $(Y,\widetilde{\text{ev}}_X,\widetilde{\text{coev}}_X)$ are right duals of $X$, then 
\begin{equation}\label{uniquedual}
\varphi\colon Y\xrightarrow{\text{coev}_X\otimes\id_Y}X^*\otimes X\otimes Y\xrightarrow{\id_{X^*}\otimes\widetilde{\text{ev}}_X}X^*
\end{equation}
 is an isomorphism between them and the (co)evaluations are related by $$\widetilde{\ev}_X=\ev_X\circ(\id_X\otimes\varphi),\quad\widetilde{\coev}_X=(\varphi^{-1}\otimes\id_X)\circ\coev_X.$$ The same holds for left dualities. 
 
 When the category $\C$ is rigid, we will always assume that distinguished left and right duals $\{X^*,\text{ev}_X,\text{coev}_X\}$ and $\{{}^*X,\text{ev}'_X,\text{coev}'_X\}$ have been fixed for every object. We thus have left and right duality functors $${}^*(-),(-)^*:\C^\text{op}\to\C$$ that, in general, do not coincide. 
 
 If $(F,J,J_0):\C\to\mathcal{D}$ is a monoidal functor between rigid categories, then $$\ev^F_X:F(X)\otimes F(X^*)\xrightarrow{J_{X,X^*}}F(X\otimes X^*)\xrightarrow{F(\ev_X)}F(\boldsymbol{1})\xrightarrow{J_0}\boldsymbol{1}$$ and $$\coev_X^F:\boldsymbol{1}\xrightarrow{J_0^{-1}} F(\boldsymbol{1})\xrightarrow{F(\coev_X)}F(X^*\otimes X)\xrightarrow{J_{X^*,X}^{-1}}F(X^*)\otimes F(X)$$ exhibit $F(X^*)$ as a right dual of $F(X)$ (and the analogue is true for left duals). If $F(X)^*$ is the distinguished dual of $F(X)$ in $\mathcal{D}$, the isomorphism 
\begin{equation}\label{Fduals}
\varphi_X:F(X^*)\to F(X)^*
\end{equation} 
from \eqref{uniquedual} is natural in $X$.

\begin{definition}
A \emph{pivotal category} is a right rigid category $\C$ (with distinguished duality) endowed with a monoidal natural isomorphism $$\iota_X: X^{**}\to X.$$ A monoidal functor $F:\C\to\mathcal{D}$ between pivotal categories is \emph{pivotal} if the isomorphism
\begin{equation}\label{Fpivotal}
\eta_X:F(X)^{**}\xrightarrow{\varphi_X^*}F(X^*)^*\xrightarrow{\varphi_{X^*}^{-1}}F(X^{**})\xrightarrow{F(\iota^\C_X)}F(X)
\end{equation}
coincides with the pivotal structure of $\mathcal{D}$.
\end{definition}

Using the pivotal structure, we can exhibit the distinguished right duality of $\C$ as a left duality. Namely, the morphisms $$\ev'_X:X^*\otimes X\xrightarrow{\id_{X^*}\otimes\iota_X^{-1}}X^*\otimes X^{**}\xrightarrow{\ev_{X^*}}\boldsymbol{1}$$ and $$\coev'_X:\boldsymbol{1}\xrightarrow{\coev_{X^*}}X^{**}\otimes X^*\xrightarrow{\iota_X\otimes\id_{X^*}} X\otimes X^*$$ satisfy the zig-zag identities. Therefore, for this choice of distinguished left duality, we get a canonical monoidal isomorphism $(-)^*\cong{}^*(-)$.

Recall that a \emph{braided} monoidal category $\C$ is a monoidal category endowed with a natural isomorphism $$c_{X\otimes Y}: X\otimes Y\to Y\otimes X$$ satisfying the well-known hexagon axioms.

\begin{definition}
A \emph{balanced category} is a braided monoidal category $\C$ endowed with a natural isomorphism $$\theta_{X}:X\to X,$$ called \emph{twist}, such that $$\theta_{X\otimes Y}=c_{Y,X}\circ c_{X,Y}\circ(\theta_X\otimes\theta_Y).$$ We say that $\C$ is \emph{ribbon} if it is rigid, balanced and, moreover, $$\theta_{X^*}=\theta_X^*,$$ for all $X\in\C$.
\end{definition}

\begin{proposition}[{\cite[Corollary A.3]{henriques}}]\label{pivotal2}
Let $\C$ be a rigid braided category. Then, there is a one-to-one correspondence between pivotal structures and  twists on $\C$. There are exactly two ways of establishing this correspondence. \QEDA
\end{proposition}

In particular, if $\C$ is a braided pivotal category, there are two choices of twists given by $$\theta^{(1)}_X=(\id_X\otimes\ev_{X^*})\circ(c_{X^{**},X}\otimes\id_{X^*})\circ(\id_{X^{**}}\otimes\coev_{X})\circ\iota_X$$ and $$\theta^{(2)}_X=\iota_X^{-1}\circ(\ev_X\otimes\id_{X^{**}})\circ(\id_{X^*}\otimes c_{X^{**},X})\circ(\coev_{X^*}\otimes\id_{X}).$$

\begin{proposition}[{\cite[Proposition A.4]{henriques}}]
Let $\C$ be a braided pivotal category. Then, $(\C,\theta^{(i)})$ is ribbon for either $i=1,2$ if, and only if, $\theta^{(1)}=\theta^{(2)}.$ \QEDA
\end{proposition}

\subsection{Graphical languages for monoidal categories} \label{graphical_calculus} The definition of the HOMFLY category $\RepqP$ and its central algebra structure (cf. Sections \ref{HOMFLYcats} and \ref{parabolic}) will rely on the graphical calculus for monoidal categories, that we briefly review here. See \cite{selinger} for an extended survey and \cite[Section 2.1]{henriques} for a nice overview.

\noindent $\boldsymbol{\star}$ \textbf{\underline{Rigid categories}.} For a rigid monoidal category $\mathcal{C}$, we consider directed ``vertical'' planar graphs $\Gamma$ embedded in $\R\times[0,1]$ with endpoints lying in $\R\times\{0,1\}$. By vertical we mean that edges never run horizontally. Local extrema (i.e. cups and caps) will be considered as vertices. The rest of vertices are coupons, with incident edges attached to either the top or the bottom face. For instance, $$\includegraphics[scale=0.8, valign = center]{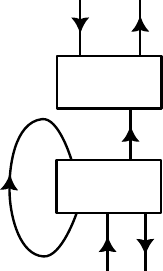}$$ is a planar graph of this type. We decorate edges of the graph with pairs $(C,n)$, where $C\in\mathcal{C}$ and $n\in\mathbb{Z},$ according to the following rule: if an edge is directed upwards, then $n$ will be even; and $n$ will be odd if it is oriented downwards. To make this consistent, $n$ has to change at local maxima and minima: turning counterclockwise increases $n$ by 1, and turning clockwise decreases $n$ by 1. Interpreting $(C,n)$ as $C^{*\overset{\times n}{\cdots}*}$ and $(C,-n)$ by ${}^{*\overset{\times n}{\cdots}*}C$, for every $n\in\mathbb{N}$,  the set of edges incident to the bottom and top faces of a coupon determines two objects $s,t\in\mathcal{C}$. The coupon will thus be decorated with a morphism $f:s\to t$. We consider two planar graphs to be equivalent if they are related by a \emph{rectilinear isotopy} of the plane: this is an isotopy of $\R\times[0,1]$ that does not rotate coupons.

\begin{definition}
The category $\text{Rig}(\C)$ has:
\begin{itemize}
\item \underline{objects}: finite sequences of pairs $(C,n),$ with $C\in\C$ and $n\in\mathbb{Z}$;
\item \underline{morphisms}: equivalence classes of decorated planar graphs as above, whose source and target are determined by the decoration of the bottom and top endpoints, respectively. Composition is given by vertically stacking diagrams. 
\end{itemize}
This category has a monoidal structure given by concatenation of objects and horizontal stacking of morphisms.
\end{definition}

The category $\text{Rig}(\C)$ is the \emph{free rigid category} on $\C$: it comes equipped with a monoidal evaluation functor $$\text{ev}_\C:\text{Rig}(\C)\to\C$$ sending $(C,n)$ to $C^{*\overset{\times n}{\cdots}*}$ and $(C,-n)$ to ${}^{*\overset{\times n}{\cdots}*}C$, for every $n\in\mathbb{N}$. A coupon decorated with $f$ is mapped to $f$ itself, and the image of cups and caps are the distinguished evaluation and coevaluation morphisms of $\C$.

\noindent $\boldsymbol{\star}$ \textbf{\underline{Pivotal categories}.} As explained in the previous section, pivotal structures yield canonical monoidal isomorphisms $(-)^*\cong {}^*(-)$ and $(-)^{**}\cong(-).$ Since multiple duals are canonically identified, a graphical calculus for pivotal categories can be obtained by dropping the integer labels from the graphical calculus for rigid categories. Namely, we consider the same class of graphs as in the previous section, labelled as follows. Edges are coloured just with objects of $\C$. For every coupon, interpreting $(C,\uparrow)$ as $C$ and $(C,\downarrow)$ as $C^*$, incident edges determine again two objects $s,t\in\C$. The coupon is thus labelled with a morphism $f:\tilde{s}\to\tilde{t}$, where $\tilde{s},\tilde{t}$ are objects canonically identified with $s,t$ via the pivotal structure. Two graphs are equivalent if they are related by a planar isotopy. In particular, we allow now isotopies rotating coupons by $2\pi$. Note that this yields a more general class of graphs as before, since, for instance, $$\includegraphics[scale=0.8, valign = c]{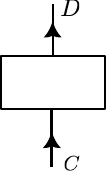}$$ may be decorated by either a morphism of the form $f:C\to D$ or $f:C\to D^{**}$ (among many other possibilities).

As in the previous case, we get a category $\text{Piv}(\C)$ whose morphisms are planar graphs, together with an evaluation functor $$\ev_\C\colon\text{Piv}(\C)\to\C,$$ which is the \emph{free pivotal category} on $\C$ (the pivotal structure being trivial). For instance, $$\ev_\C\left(\includegraphics[scale=0.4, valign=c]{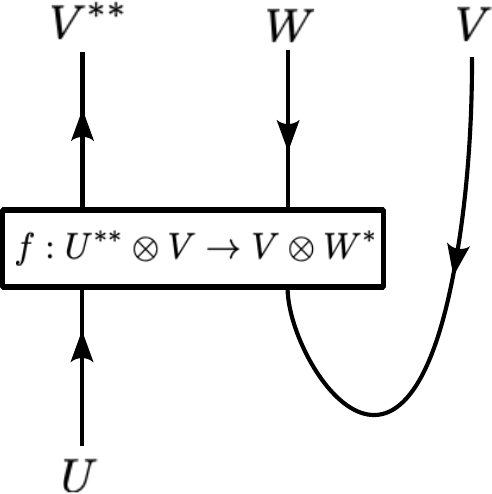}\right)=\left(\iota_V^{-1}\otimes\id_W\otimes\id_V\right)\circ\left(f\otimes\id_{V^*}\right)\circ\left(\iota_U^{-1}\otimes\text{coev}'_V\right),$$ where $\iota$ is the pivotal structure of $\C.$

\noindent $\boldsymbol{\star}$ \textbf{\underline{Braided pivotal categories}.} The graphical calculus for braided pivotal categories is supported by three-dimensional diagrams:

\begin{definition}
Let $\C$ be a braided pivotal category. The category $\text{BrPiv}(\C)$ has:
\begin{itemize}
\item \underline{objects}: finite sequences of pairs $(C,\varepsilon)$, with $C\in\C$ and $\varepsilon\in\{\uparrow,\downarrow\}$;
\item \underline{morphisms}: equivalence classes of directed graphs embedded in $\R^2\times[0,1]$ with endpoints lying in $\R\times\{0\}\times\{0,1\}$ whose vertices are coupons decorated following the same rules as for pivotal categories.
\end{itemize}
\end{definition}

The category $\text{BrPiv}(\C)$ is itself braided pivotal, with trivial pivotal structure and braiding given by crossing strands. Graphs will be represented by diagrams obtained by projecting into $\R\times[0,1]$ and keeping track of the relative position of the strands in crossings. For example,
$$\includegraphics[scale=0.5]{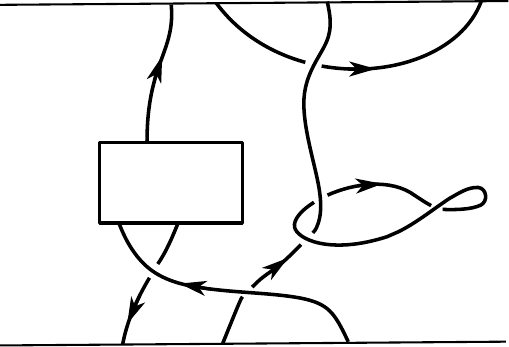}$$
is a diagram representing such a three-dimensional graph. Two graphs are equivalent if they are related by \emph{regular isotopy}, i.e., their diagrams can be obtained from each other by applying a finite sequence of the Reidemeister moves II and III and a planar isotopy. In particular,
\begin{equation}\label{ribbon_twists}
\MyFigure{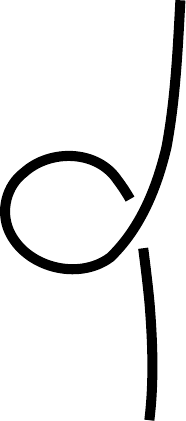}\quad\neq\quad\MyFigure{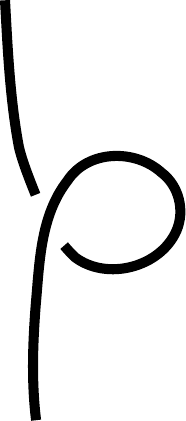}.
\end{equation}
The two diagrams in \eqref{ribbon_twists} represent graphically the two families of twists turning $\text{BrPiv}(\C)$ into a balanced category (cf. Proposition \ref{pivotal2}).

The category $\text{BrPiv}(\C)$ is the \emph{free braided category} on $\C$. Again, we have an evaluation functor $$\text{ev}_\C:\text{BrPiv}(\C)\to\C$$ mapping a positive crossing to the braiding.

\noindent $\boldsymbol{\star}$ \textbf{\underline{Ribbon categories}.} The inequality in \eqref{ribbon_twists} suggests that braided pivotal categories are still not very natural from a topological point of view. Ribbon categories are defined to be the class of braided pivotal categories in which the two terms coincide. The graphical calculus of ribbon categories is supported by so-called ribbon graphs.

A \emph{ribbon graph} is a three-dimensional graph whose edges are thick ribbons. More precisely, it is a compact oriented surface $\Omega$ embedded in $\R^2\times[0,1]$ which decomposes into the following elementary pieces:
\begin{enumerate}
\item \emph{ribbons}, i.e., homeomorphic copies of the square $[0,1]\times[0,1]$ whose core is directed;
\item \emph{annuli}, i.e., homeomorphic copies of the cylinder $\mathbb{S}^1\times[0,1]$ whose core is also directed;
\item \emph{coupons}, with distinguished top and bottom bases.
\end{enumerate}
Ribbons can meet coupons just on their distinguished bases. On the other hand, the choice of an orientation for $\Omega$ determines a ``preferred side''. We demand the free bases of ribbons (those not meeting any coupon) to meet the planes $\R^2\times\{0\}$ and $\R^2\times\{1\}$ orthogonally at $\R\times\{0\}\times\{0\}$ and $\R\times\{0\}\times\{1\}$, respectively, and in such a way that the preferred side of $\Omega$ is turned up. We colour ribbons, coupons and annuli following the same rules as for pivotal categories.

Ribbon graphs will be considered up to isotopy: an \emph{isotopy of ribbon graphs} is an isotopy of $\R^2\times[0,1]$ which is the identity on $\R^2\times\{0,1\}$ and preserves the decomposition into ribbons, annuli and coupons. By applying such an isotopy, any ribbon graph is equivalent to a graph whose ribbons, coupons and annuli are everywhere parallel to the vertical plane. Such a graph can be represented by a \emph{ribbon diagram}, obtained by projecting ribbons and annuli onto their cores and projecting the resulting graph into $\R\times[0,1]$ (see Figure \ref{fig:ribbon_diagram}). Two ribbon diagrams represent the same ribbon graph if they can be obtained from each other by a finite sequence of Reidemeister moves II and III, planar isotopy and movements  of the form $$\MyFigure{drawings/twist1.pdf}\quad\leftrightarrow\quad\MyFigure{drawings/twist2.pdf}.$$

\begin{figure}[t]
\centering
\begin{minipage}{.45\textwidth}
  \centering
  \includegraphics[scale=0.10]{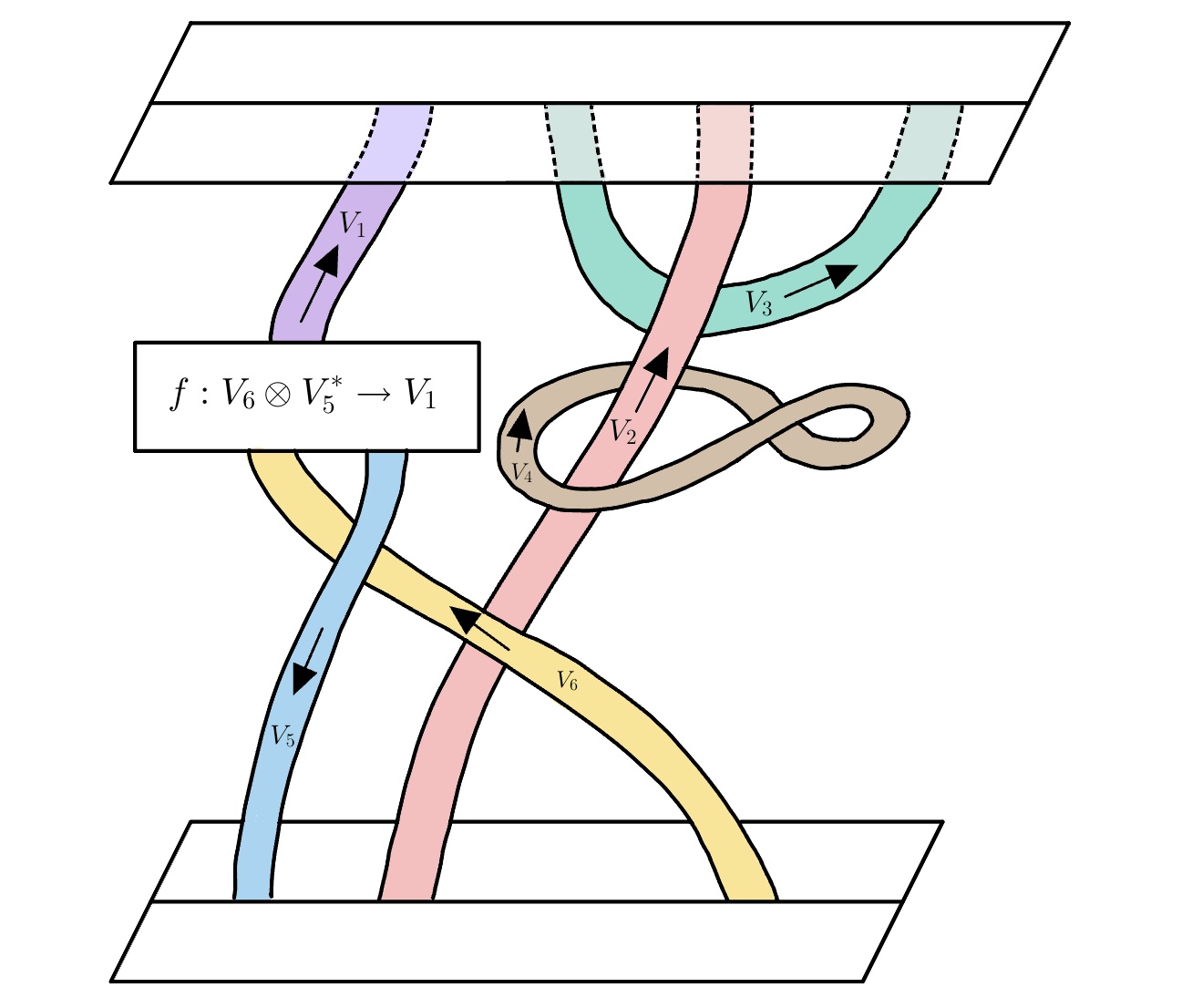}
  \captionof{figure}{Coloured ribbon graph.}
  \label{fig:ribbon_graph}
\end{minipage}
\hfill
\begin{minipage}{.45\textwidth}
  \centering
  \vspace{10pt}
  \includegraphics[scale=0.6]{drawings/ribdig.pdf}
  \captionof{figure}{Ribbon diagram.}
  \label{fig:ribbon_diagram}
\end{minipage}
\end{figure}

\begin{definition}[{\cite[Section I.2]{turaevbook}}]
Let $\C$ be a ribbon category. The category $\Rib(\C)$ has:
\begin{itemize}
\item \underline{objects}: finite sequences of pairs $(C,\varepsilon)$, with $C\in\C$ and $\varepsilon\in\{\uparrow,\downarrow\}$;
\item \underline{morphisms}: isotopy classes of directed ribbon graphs (defined below) embedded in $\R^2\times[0,1]$ and coloured with $\C$ in the usual way.
\end{itemize}
This is a ribbon category, with twist given by $\includegraphics[scale=0.14, valign = c]{drawings/twist1.pdf}$.
\end{definition}

As in the previous cases, $\Rib(\C)$ is the \emph{free ribbon category} on $\C$, and we have an evaluation functor $$\text{ev}_\C:\text{Rib}(\C)\to\C,$$ known as the \emph{Reshetikhin--Turaev evaluation functor}. 

\subsection{Skein categories} The graphical calculus for ribbon categories extends to arbitrary oriented surfaces, yielding the notion of skein category. Skein categories are categorical invariants of oriented surfaces generalising the notion of skein algebras. They were introduced in \cite{walker, jf} and take a ribbon category $\C$ as an algebraic ingredient:

\begin{definition}\textup{(\cite[Definition 1.3]{cookethesis})}
Let $S$ be an oriented surface. A \emph{$\C$-coloured ribbon graph in $S$} is an embedding of a $\C$-coloured ribbon graph into $S\times[0,1]$ such that its free bases are sent to $S\times\{0,1\}$ and the rest lies in $S\times(0,1).$
\end{definition}

\begin{definition}\textup{(\cite[Definition 1.5]{cookethesis})}
Let $\Bbbk$ be a commutative ground ring. The \emph{category $\Rib(\C,S)$ of ribbon graphs in $S$} has:
\begin{itemize}
\item \underline{objects}: finite sets $\left\lbrace x_1^{(C_1,\varepsilon_1)},\ldots,x_m^{(C_m,\varepsilon_m)}\right\rbrace$ of disjoint framed points $x_i\in S$ coloured with pairs $(C_i,\varepsilon_i)$, where $C_i\in\C$ and $\varepsilon_i\in\{\uparrow,\downarrow\}$;
\item \underline{morphisms} from $\left\lbrace x_1^{(C_1,\varepsilon_1)},\ldots,x_m^{(C_m,\varepsilon_m)}\right\rbrace$ to $\left\lbrace y_1^{(D_1,\eta_1)},\ldots,y_n^{(D_m,\eta_n)}\right\rbrace$ are $\Bbbk$-linear combinations of isotopy classes of $\C$-coloured ribbon graphs in $S$ whose bottom and top are, respectively, $\left\lbrace x_1^{(C_1,\varepsilon_1)},\ldots,x_m^{(C_m,\varepsilon_m)}\right\rbrace$ and $\left\lbrace y_1^{(D_1,\eta_1)},\ldots,y_n^{(D_m,\eta_n)}\right\rbrace$.
\end{itemize}
Composition is given by vertically stacking two copies of $S\times[0,1]$ and retracting $S\times[0,2]$ to $S\times[0,1]$. In particular, $\Rib(\C,\R^2)=\Rib(\C).$
\end{definition}

\begin{definition}\textup{(\cite[Definition 1.9]{cookethesis})}
Let $\C$ be a ribbon category and $S$ an oriented surface. The \emph{skein category $\Sk_\C(S)$} is the quotient of $\Rib(\C,S)$ by the following relations: a morphism $\sum_i\lambda_i\Omega_i\sim 0$ if there exists an orientation-preserving embedding $\iota:\R^2\times[0,1]\hookrightarrow S\times[0,1]$ such that:
\begin{enumerate}
\item the intersection of each $\Omega_i$ with the boundary of $\iota(\R^2\times[0,1])$ consists only of transverse ribbons intersecting $\iota(\R\times\{0\}\times\{0\})$ and $\iota(\R\times\{0\}\times\{1\})$;
\item the $\Omega_i$ are equal outside $\iota(\R^2\times[0,1])$;
\item we have $$\sum_i\lambda_i\text{ev}_\C\left(\iota^{-1}\left(\Omega_i\cap\iota(\R^2\times[0,1])\right)\right)=0.$$
\end{enumerate}
\end{definition}

Note that skein categories are not monoidal in general. When $S=C\times I$ for some one-manifold $C$, a tensor product can be defined by stacking two copies of $C\times I$. The following is a direct consequence of the definition:

\begin{proposition}\label{trivialskein}
Let $\C$ be a ribbon category. Then, $$\Sk_\C(\R^2)\simeq\C$$ as ribbon categories. \QEDA
\end{proposition}

\subsection{Parabolic restriction}\label{parabolic_restriction} Let $\Bbbk$ be a field of characteristic $0$ and $q\in\Bbbk^\times$. Fix $N\in\mathbb{N}$ and assume that $q$ is not a root of unity. The algebra $U_q(\mathfrak{gl}_N)$ is defined by generators $e_i,f_i,d_j,d_j^{-1}$, $i=1,\ldots, N-1$, $j=1,2,\ldots,N$, and relations
\begin{equation*}
d_id_j=d_jd_i,\quad d_id_i^{-1}=d_i^{-1}d_i=1,
\end{equation*}
\begin{equation*}
d_ie_jd_i^{-1}=q^{\delta_{ij}-\delta_{i,j+1}}e_j,\quad d_if_jd_i^{-1}=q^{-\delta_{i,j}+\delta_{i,j+1}}f_j,
\end{equation*}
\begin{equation*}
e_if_j-f_je_i=\delta_{ij}\frac{d_id_{i+1}^{-1}-d_i^{-1}d_{i+1}}{q-q^{-1}},
\end{equation*}
\begin{equation*}
e_ie_j=e_je_i,\quad f_if_j=f_jf_i,\quad|i-j|\geq 2,
\end{equation*}
\begin{equation*}
e_i^2e_{i\pm 1}-(q+q^{-1})e_ie_{i\pm 1}e_i+e_{i\pm 1}e_i^2=0,
\end{equation*}
\begin{equation*}
f_i^2f_{i\pm 1}-(q+q^{-1})f_if_{i\pm 1}f_i+f_{i\pm 1}f_i^2=0.
\end{equation*}
This algebra becomes a Hopf algebra with the comultiplication  $\Delta: U_q(\mathfrak{gl}_N)\to U_q(\mathfrak{gl}_N)\otimes U_q(\mathfrak{gl}_N)$ defined on generators by $$\Delta(e_i)=d_i^{-1}d_{i+1}\otimes e_i+e_i\otimes 1,\quad\Delta(f_i)=1\otimes f_i+f_i\otimes d_id_{i+1}^{-1},\quad\Delta(d_i^{\pm1})=d_i^{\pm 1}\otimes d_i^{\pm 1}.$$ The counit $\varepsilon:\Uq\to\Bbbk$  and the antipode $S:\Uq\to\Uq$ are given, respectively, by $$\varepsilon(e_i)=0,\quad\varepsilon(f_i)=0,\quad\varepsilon(d_i)=1$$ and $$S(e_i)=-d_id_{i+1}^{-1}e_i,\quad S(f_i)=-f_id_i^{-1}d_{i+1},\quad S(d_i)=d_i^{-1}.$$

\begin{definition}
The \emph{natural representation} $\rho:\Uq\to\End_{\mathbb{C}}(\mathbb{C}^N)$ of $\Uq$ is defined by $$\rho(d_i)=q E_{i,i}+\sum_{i\neq j}E_{j,j},
\quad\rho(e_i)=E_{i,i+1},\quad\rho(f_i)=E_{i+1,i},$$ where $E_{i,j}$ is the $N\times N$ matrix with $1$ in the $(i,j)$-position and $0$ elsewhere. We denote this representation by $V_N$.
\end{definition}

Fix $m,n\in\mathbb{N}$ and consider the subalgebra $U_q(\mathfrak{l})$ of $U_q(\mathfrak{gl}_{m+n})$ generated by $\{d_i,e_j,f_j\;|\;j\neq m\}.$ We may identify $U_q(\mathfrak{l})\cong U_q(\mathfrak{gl}_m)\otimes U_q(\mathfrak{gl}_n)$ by mapping $$d_i\mapsto d_i,\qquad e_j,f_j\mapsto\begin{cases} e_j\otimes 1,f_j\otimes 1, & \text{for $0\leq j\leq m-1$},\\ 1\otimes e_{j-m}, 1\otimes f_{j-m}, &\text{for $m+1\leq j\leq m+n-1$},\end{cases} $$ so that $U_q(\mathfrak{l})$ acts on $V_m$ and $V_n$. Similarly, let $U_q(\mathfrak{p})$ be the subalgebra of $U_q(\mathfrak{gl}_{m+n})$ generated by $\{d_i,e_i,f_j\;|\; j\neq m\}.$ Note that $U_q(\mathfrak{l})$ is a subalgebra of $U_q(\mathfrak{p})$, but also a quotient by the ideal generated by $e_m$. 

\begin{definition}
We write $\Rep_q\GL_{m+n}$, $\Rep_q\L$ and $\Rep_q\P$ for the categories of locally finite dimensional representations of $U_q(\mathfrak{gl}_{m+n})$, $U_q(\mathfrak{l})$ and $U_q(\mathfrak{p}),$ respectively.
\end{definition}

Restricting along the inclusion and projection morphisms, we get functors
\begin{equation}\label{manyfunctors}
\begin{tikzcd}
                                                                              & \Rep_q\P \arrow[rd, "j^*", bend left] &                                              \\
\Rep_q\GL_{m+n} \arrow[ru, "\iota^*"] \arrow[rr, "\text{res}_{m,n}"] &                                                      & \Rep_q\L. \arrow[lu, "\pi^*"']
\end{tikzcd}
\end{equation}
Note that $\text{res}_{m,n}(V_{m+n})\simeq V_m\oplus V_n$ in $\Rep_q\L.$

For any $N\in\mathbb{N}$, the universal $R$-matrix of $U_q(\mathfrak{gl}_{N})$ is an invertible element $\mathcal{R}$ lying in a completion of $U_q(\mathfrak{gl}_{N})\otimes U_q(\mathfrak{gl}_{N})$. It induces a braiding on $\Rep_q\GL_N$ given, in Sweedler notation, by $$c_{V,W}(v,w)\coloneqq\mathcal{R}_{(2)}w\otimes\mathcal{R}_{(1)}v,$$ for $v\in V$ and $w\in W.$ Applied to the fundamental representation, the braiding yields an automorphism $\beta_{N}\colon V_{N}\otimes V_{N}\to V_{N}\otimes V_{N}$ that can be written explicitly as \begin{equation}
\beta_{N}=q\sum_i(E_{i,i}\otimes E_{i,i})+\sum_{i\neq j}(E_{i,i}\otimes E_{i,j})+\left(q-q^{-1}\right)\sum_{i<j}(E_{i,j}\otimes E_{j,i}).
\end{equation}
Taking $N=m+n$ and restricting to $U_q(\mathfrak{l})$, this isomorphism decomposes as \begin{equation}
\text{res}_{m,n}(\beta_{m+m})=\left(\begin{array}{c|c|c|c}
\beta_m & 0 & 0 & 0\\
\hline
0 & 0 & \sigma_{m,n} & 0\\
\hline
0 & \sigma_{n,m} & \left(q-q^{-1}\right)\text{Id}_{m,n} & 0\\
\hline
0 & 0 & 0 & \beta_n
\end{array}\right)\colon (V_m\oplus V_n)^{\otimes 2}\to(V_m\oplus V_n)^{\otimes 2}
\end{equation}
in $\Rep_q\L,$ hence the monoidal functor $\text{res}_{m,n}$ is not braided. 
 
On the other hand, it is also well-known that the category $\Rep_q\P$ is not braided. Let $W\in\Rep_q\GL_{m+n}$ and $V\in\Rep_q\P$. The universal $R$-matrix $\mathcal{R}$ of $U_q(\mathfrak{gl}_n)$ lies in a completion of $U_q(\mathfrak{b}_+)\otimes U_q(\mathfrak{b}_-)\subseteq U_q(\mathfrak{p})\otimes U_q(\mathfrak{b}_-),$ where $\mathfrak{b}_\pm$ are the positive and negative Borel subalgebras of $\mathfrak{gl}_{m+n}.$ Thus, $\mathcal{R}$ has a well-defined action on $V\otimes\iota^*(W)$ and composing with the flip of tensor factors yields an isomorphism of $V\otimes j^*(W)\to j^*(W)\otimes V$ in $\Rep_q\P$ which is natural in $V$. Similarly, if $U\in\Rep_q\L$, the $R$-matrix of $U_q(\mathfrak{l})$ acts on $V\otimes\pi^*(U)$ providing an  isomorphism $V\otimes\pi^*(U)\to\pi^*(U)\otimes V$ natural in $V$. 

\begin{proposition}
The restriction functors $\iota^*\colon\Rep_q\GL_{m+n}\to\Rep_q\P$ and $\pi^*\colon\Rep_q\L\to\Rep_q\P$ lift to a braided functor $$\iota^*\otimes\pi^*\colon\Rep_q\GL_{m+n}\boxtimes\overline{\Rep_q\L}\to Z\left(\Rep_q\P\right),$$ where $\overline{\Rep_q\P}$ is the opposite of $\Rep_q\L$ as a braided category and $Z(\Rep_q\P)$ is the Drinfeld centre of $\Rep_q\P$. \QEDA
\end{proposition}

\begin{definition}
We call the structure in the previous proposition the \emph{parabolic central algebra.}
\end{definition}

\subsection{Categories and bimodules} We recall in this section a few categorical notions that will appear repeatedly in our constructions. We fix a ground ring $\Bbbk$ and we set $\Vect$ for the category of $\Bbbk$-modules. A \emph{$\Bbbk$-linear category} is a category $\C$ enriched over $\Vect$. The category $\Cat$ of small $\Bbbk$-linear categories and $\Bbbk$-linear functors is symmetric monoidal:

\begin{definition}
Given $\C,\mathcal{D}\in\Cat$, we define $\C\boxtimes\mathcal{D}$ as the $\Bbbk$-linear category whose
\begin{itemize}
\item \underline{objects} are pairs $(c,d)$ with $c\in\C, d\in\mathcal{D}$;
\item \underline{morphisms} are given by $$\Hom_{\C\boxtimes\mathcal{D}}((c_1,d_1),(c_2,d_2))\coloneqq\Hom_\C(c_1,c_2)\otimes_\Bbbk\Hom_{\mathcal{D}}(d_1,d_2).$$
\end{itemize}
\end{definition}

\begin{definition}
Let $F\colon\C\boxtimes\C^\op\to\mathcal{D}$ be a $\Bbbk$-linear bifunctor and suppose that $\mathcal{D}$ is cocomplete. The \emph{coend of $F$} is the object $$\int^{c\in\C}F(c,c)\coloneqq\mathop{\text{colim}}\left(\begin{tikzcd}
{\mathop{\prod}\limits_{f\colon c\to c'}F(c,c')} \arrow[r, "{F(f,c')}", shift left=2] \arrow[r, "{F(c,f)}"'] & {\mathop{\prod}\limits_{c\in\C}F(c,c)}
\end{tikzcd}\right)\in\mathcal{D}.$$
\end{definition}
If $\mathcal{D}=\Vect$, this is explicitly given by $$\int^{c\in\C}F(c,c)=\bigslant{\bigoplus\limits_{c\in\C}F(c,c)}{\sim},$$ where we quotient by the image of morphisms of the form $$F(c,c')\xrightarrow{F(f,c')-F(c,f')}F(c,c)\oplus F(c',c').$$

\begin{definition}
The category $\textsc{Bimod}$ has:
\begin{itemize}
\item \underline{objects}: small $\Bbbk$-linear categories;
\item \underline{morphisms}: a morphism from $\mathcal{C}$ to $\mathcal{D}$ is a functor of the form $F:\mathcal{C}\boxtimes\mathcal{D}^\text{op}\to\Vect.$
\end{itemize}
\end{definition}

The composition of $F:\mathcal{C}\boxtimes\mathcal{D}^\text{op}\to\Vect$ and $G:\mathcal{D}\boxtimes\mathcal{E}^\text{op}\to\Vect$ is given by the coend $$(G\circ F)(c,e)\coloneqq\int^{d\in\mathcal{D}}F(c,d)\otimes_\Bbbk G(d,e).$$ The identity of $\C$ is the bimodule $\Hom_\C(-,-)$. The tensor product of $\Bbbk$-linear categories endows $\textsc{Bimod}$ with the structure of a symmetric monoidal category.

\begin{remark} 
The category $\Cat$ embeds into $\Bimod.$ Indeed, given a $\Bbbk$-linear functor $F\colon\C\to\mathcal{D}$, we get a bimodule $$F^*\colon\C\boxtimes\mathcal{D}^\op\to\Vect,\quad c\boxtimes d\mapsto\Hom_\mathcal{D}(d,F(c)).$$
\end{remark}

Finally, we recall the following notion of completion:

\begin{definition}
The \emph{Cauchy completion} of a category $\C$ is an additive category $\text{Cauchy}(\C)$ together with an embedding $\C\hookrightarrow\text{Cauchy}(\C)$ such that every idempotent splits in $\text{Cauchy}(\C)$ and $\text{Cauchy}(\C)$ is universal for this property.
\end{definition}

\subsection{Stratified spaces} In Sections \ref{planar_theories} and \ref{3d_theories}, we define a skein theory on 3-manifolds with line and surface \emph{defects}. Here, we introduce the terminology and conventions for stratified spaces that will be used later.

\begin{definition} (i) A \emph{stratified space} is a topological space $X$ endowed with a continuous map $$\phi\colon X\to P,$$ where $P$ is a poset endowed with the topology whose open sets are generated by $P_{>p}\coloneqq\{x\in P\mid x>p\},$ for $p\in P.$ For each $p\in P,$ the subset $X_p\coloneqq\phi^{-1}(p)$ is a \emph{stratum}.

(ii) A \emph{morphism of stratified spaces} between $(X,\phi)$ and $(Y,\psi)$ consists of a pair of continuous maps $f\colon X\to Y$, $F\colon P\to Q$ such that the diagram $$\begin{tikzcd}
X \arrow[d, "\phi"'] \arrow[r, "f"] & Y \arrow[d, "\psi"] \\
P \arrow[r, "F"]                    & Q                  
\end{tikzcd}$$ is commutative. We say that it is a \emph{stratified embedding} if $f$ and $f|_{X_p}$ are embeddings for every $p\in P.$

(iii) A \emph{stratified isotopy} between two stratified embeddings $(f,F), (g,G)\colon X\to Y$ is a morphism $(h,H)\colon X\times I\to Y$ of stratified spaces such that $h(t,-)$ is an embedding for every $t\in I$, $h(-,0)=f$ and $h(-,1)=g$. Here, the stratification on $X\times I$ is given by $\varphi(x,t)=\phi(x)$, where $\phi$ is the stratification on $X$. In particular, a \emph{stratified isotopy of $X$} is a continuous family of embeddings $f_t\colon X\to X$ such that $$\begin{tikzcd}
X \arrow[rd, "\phi"'] \arrow[r, "f_t"] & X \arrow[d, "\phi"] \\
                                       & P                  
\end{tikzcd}$$ commutes for every $t\in[0,1]$ and $f_0=\id_X.$
\end{definition}

We will typically deal with stratified three-manifolds and we will say that $X_p$ is a \emph{line defect} when $\dim(X_p)=1$ and that it is a \emph{surface defect} when $\dim(X_p)=2.$

\begin{definition}\label{bipartite}
A \emph{bipartite surface} is an oriented surface $S$ endowed with a stratification $$\phi\colon S\to(A\geq B \leq C)$$ such that $S_B=\phi^{-1}(B)$ is a smooth curve that forms the common boundary of the open strata $S_A$ and $S_C$.
\end{definition}

\section{HOMFLY categories and ordinary restriction}\label{HOMFLYcats} 

Throughout the rest of the paper, we fix the ground ring to be $$\Bbbk=\bigslant{\mathbb{C}\left(q)[q^{\pm t},\delta\right]}{\left\langle\left(q-q^{-1}\right)\delta=q^t-q^{-t}\right\rangle}.$$ In this section, we recall the definition of the HOMFLY category $\RepqG$ (\cite{brundan, deligne}) and its two-coloured version, $\RepqH$. Next, we define a HOMFLY version of the category $\RepqP$ of quantum representations for a parabolic subgroup $\P\subseteq\GL_{m+n}$.

\subsection{The HOMFLY category $\RepqG$} One of our main objects of study will be the HOMFLY category $\RepqG$, which we introduce following \cite{brundan}. Let $\mathcal{T}$ be the category of $1$-coloured framed oriented tangles, whose objects are finite words over $\{\upp,\downn\}$ and whose morphisms are $\Bbbk$-linear combinations of isotopy classes of framed oriented \emph{tangles} (i.e. ribbon graphs without coupons) in $\R^2\times[0,1]$. This category has a ribbon structure. The tensor product is given by concatenation of objects and by horizontally stacking morphisms. The right dual of an object $w=w_1\cdots w_k$ is the object $w^*\coloneqq w_k^*\cdots w_1^*$, where $(-)^*$ interchanges $\upp$ and $\downn$. The braiding, evaluation, coevaluation and twist are represented by the diagrams
\begin{equation*}
\MyFigure{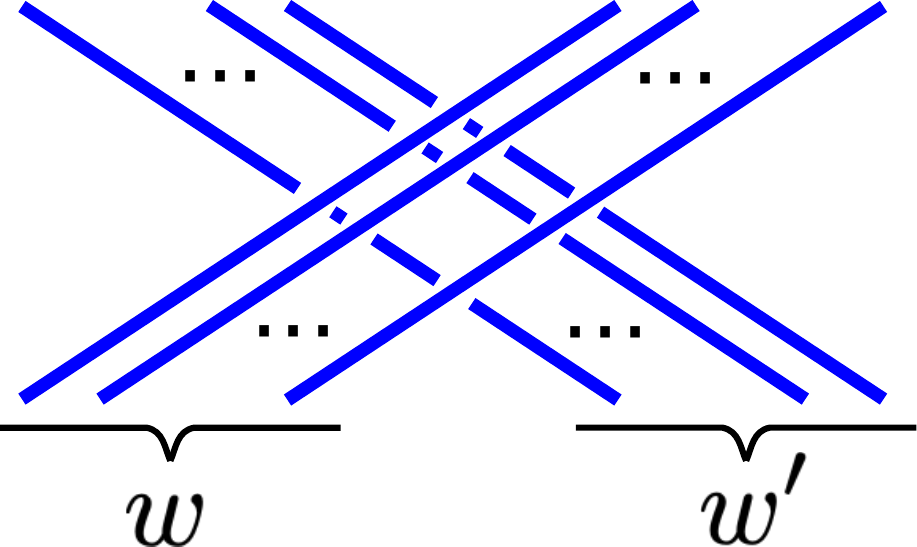},\quad\MyFigure{\evfot},\quad\MyFigure{\coevfot},\quad\MyFigure{\twistfot},
\end{equation*}
where the directions are induced by the depicted words. Exchanging $w$ and $w^*$ in the evaluation and coevaluation maps, we can exhibit $w^*$ as a left dual of $w$; hence, $\mathcal{T}$ has a strictly pivotal structure.

\begin{definition}\label{RepqGdef}
The category $\widetilde{\RepqG}$ is the quotient of $\mathcal{T}$ by the \emph{HOMFLY skein} \eqref{skein}, the \emph{twist} \eqref{twist} and the \emph{dimension} \eqref{dim} relations:
\begin{equation}\label{skein}
\MyFigure{\blueposx}-\MyFigure{\bluenegx}=\left(q-q^{-1}\right)\MyFigure{\blueidid},
\end{equation}
\begin{equation}\label{twist}
\MyFigure{\bluetwista}=q^t\MyFigure{\blueid},
\end{equation}
\begin{equation}\label{dim}
\MyFigure{\bknot}=\delta\;\bold{1}_\emptyset,
\end{equation}
where $\boldsymbol{1}_\emptyset$ is the empty diagram. 
\end{definition}

This category inherits a ribbon structure from $\mathcal{T}$ and it admits a presentation by generators and relations:

\begin{proposition}[{\cite[Theorem 1.1]{brundan}}] \label{HOMFLYgenrel}
$\widetilde{\RepqG}$ is generated, as a $\Bbbk$-linear strict monoidal category, by the objects $\upp$ and $\downn$ and the morphisms $$\MyFigure{\blueposx},\qquad\MyFigure{\blueposz},\qquad\MyFigure{\blueevv},\qquad\MyFigure{\bluecoev},$$ subject to the relations 
\begin{equation}\label{rel1}
\MyFigure{\bluexx}=\left(q-q^{-1}\right)\MyFigure{\blueposx}+\MyFigure{\blueidid},
\end{equation}
\begin{equation}\label{rel2}
\MyFigure{\zigzaga}=\MyFigure{\blueid},\qquad\MyFigure{\zigzagb}=\MyFigure{\blueiddual},
\end{equation}
\begin{equation}\label{rel3}
\left(\MyFigure{\blueposz}\right)^{-1}=\MyFigure{\blueinvz}.
\end{equation}
\begin{equation}\label{rel4}
q^{t}\MyFigure{\bluedimrelation}=\delta\;\boldsymbol{1}_\emptyset.
\end{equation}
\begin{equation}\label{rel5}
\MyFigure{\RIIIa}=\MyFigure{\RIIIb},
\end{equation}\QEDA
\end{proposition}

\begin{definition}
The \emph{HOMFLY category} $\RepqG$ is the Cauchy completion of $\widetilde{\RepqG}$.
\end{definition}

\begin{remark}\label{universal1}  The HOMFLY category $\RepqG$ interpolates between the categories $\Rep_q\GL_N$ of locally finite representations of $U_q(\mathfrak{gl}_N)$. More precisely, for any integer $N\in\mathbb{N}$, consider the evaluation map $$\varphi_{N}: \Bbbk \to \mathbb{C}(q),\quad  q^t\mapsto q^N,\quad \delta\mapsto\frac{q^N-q^{-N}}{q-q^{-1}}.$$ There exists a $\varphi_{N}$-linear full ribbon functor $$\text{ev}^\text{G}_{N}:\RepqG\to\Rep_q\GL_N$$ sending the object $\boldsymbol{\upp}$ to the natural representation $V_N$. The functor is however not faithful: its kernel is the ideal of negligible morphisms of $\RepqG$ (see \cite[Theorem 1.3]{brundan} for a precise statement).
\end{remark} 

\subsection{The HOMFLY category $\RepqH$} Next, we describe diagrammatically the \emph{two-coloured HOMFLY category} $\RepqH$, interpolating the categories $\Rep_q\L$ of locally finite $U_q(\mathfrak{l})$-modules.

\begin{definition}
The \emph{two-coloured HOMFLY category} $\RepqH$ is the Cauchy completion of $\RepqG\boxtimes_{\mathbb{C}(q)}\RepqG.$ 
\end{definition}

In the previous definition, the symbol means $\boxtimes_{\mathbb{C}(q)}$ means that the tensor product is taken in the category of $\mathbb{C}(q)$-linear categories (instead of $\Bbbk$-linear categories). Therefore, $\RepqH$ is enriched over $\Bbbk\otimes_{\mathbb{C}(q)}\Bbbk$. We will set $$q^{t_1}\coloneqq q^t\otimes 1, \quad q^{t_2}\coloneqq 1\otimes q^t,\quad \delta_1\coloneqq \delta\otimes 1,\quad \text{and}\quad\delta_2\coloneqq 1\otimes\delta$$ for the generators of this ring. The following is straightforward:

\begin{proposition}
$\RepqH$ has a ribbon structure given factor-wise by that of $\RepqG.$ \QEDA
\end{proposition}

We will represent this category diagrammatically (cf. Proposition \ref{trivialskein}) using the following colour conventions: $$\gsquare\coloneqq(\bsquare,\emptyset),\quad\redsquare\coloneqq(\emptyset,\bsquare)\quad\text{and}\quad\orsquare\coloneqq\gsquare\oplus\redsquare.$$ The coupons representing the inclusions and projections between the green/red colour and the orange one will be represented as $$\MyFigure{\ginc},\qquad\MyFigure{\gproj},\qquad\MyFigure{\rinc},\qquad\MyFigure{\rproj}.$$
With this notation, the following relations hold in $\RepqH$ by definition:
\begin{itemize}
\item the one-coloured skein, twist and dimension relations:
\begin{equation}\label{RepqGtGt1}
\MyFigure{\blackposx}-\MyFigure{\blacknegx}=\left(q-q^{-1}\right)\MyFigure{\blackidid},
\end{equation}
\begin{equation}
\MyFigure{\blacktwist}=q^{t_\blacksquare}\MyFigure{\blackid},
\end{equation}
\begin{equation}
\MyFigure{\blackdim}=\delta_{\blacksquare}\boldsymbol{1}_\emptyset,
\end{equation}
with all the strands coloured in the same colour $\blacksquare\in\{\gsquare=1,\redsquare=2\}$;
\item the two-coloured crossing relations:
\begin{equation}\label{RepqGtGt4}
\MyFigure{\grposx}=\MyFigure{\grnegx},\qquad\MyFigure{\rgposx}=\MyFigure{\rgnegx};
\end{equation}
\item the relations
\begin{equation}\label{RepqHtrel1}
\MyFigure{\orid}=\MyFigure{\ogip}+\MyFigure{\orip},
\end{equation}
\begin{equation}\label{RepqHtrel2}
\MyFigure{\ggip}=\MyFigure{\gid},\qquad\qquad\MyFigure{\grip}=0,\qquad\qquad\MyFigure{\rgip}=0,\qquad\qquad\MyFigure{\rrip}=\MyFigure{\rid},
\end{equation}
expressing the fact that $\orsquare=\gsquare\oplus\redsquare.$
\end{itemize}
Thanks to relations \eqref{RepqGtGt4}, we will no longer distinguish between positive and negative two-coloured crossings. Moreover, note that the HOMFLY, twist and dimension relations do not hold for the orange object. Instead, we have:

\begin{lemma}\label{ReqHtlemma}
The following relations hold in $\RepqH$:
\begin{equation}
\MyFigure{\oposx}-\MyFigure{\ornegx}=\left(q-q^{-1}\right)\left(\MyFigure{\ogipip}+\MyFigure{\oripip}\right),
\end{equation}
\begin{equation}
\MyFigure{\ortwist}=q^{t_1}\MyFigure{\ogip}+q^{t_2}\MyFigure{\orip},
\end{equation}
\begin{equation}
\MyFigure{\orknot}=\delta_1+\delta_2.
\end{equation}
\end{lemma}

\begin{proof}
Applying \eqref{RepqHtrel1}, we can write
\begin{equation}\label{oposx}
\MyFigure{\oposx}=\MyFigure{\ogposx}+\MyFigure{\orposx}+\MyFigure{\ogrx}+\MyFigure{\orgx},
\end{equation} 
and the same is true for the negative crossing. The first relation of the statement thus follows from the green and red skein relations. Likewise, the blue twist and dimension relations are an easy consequence of the green and red ones.
\end{proof}

\begin{remark}\label{universal2}
Again, $\RepqH$ is universal among the categories $\Rep_q\L$. Specialising $t_1$ and $t_2$ at integer values $m,n\in\mathbb{N}$, we get an evaluation morphism $$\varphi_{m,n}:\Bbbk\otimes_{\mathbb{C}(q)}\Bbbk\to\mathbb{C}(q),$$ and there is a canonical evaluation functor $$\text{ev}^{\text{L}}_{m,n}:\RepqH\to\Rep_q\L$$ mapping $\textcolor{mygreen}{\boldsymbol{\uparrow}}$ and $\textcolor{myred}{\boldsymbol{\uparrow}}$ to the natural representations of $\Uqglm$ and $\Uqgln$, respectively, and $\boldsymbol{\textcolor{myorange}{\uparrow}}$ to their direct sum.
\end{remark}

\begin{lemma}\label{coprod2}
For any oriented surface $S$, we have an equivalence of $(\Bbbk\otimes_{\mathbb{C}(q)}\Bbbk)$-linear categories $$\Sk_{\H_{t,t}}(S)\cong\Sk_{\GL_t}(S)\boxtimes\Sk_{\GL_t}(S).$$ In particular, this induces an algebra isomorphism $$\text{Sk}_{\H_{t,t}}(S)\cong\skg\otimes_{\mathbb{C}(q)}\skg.$$
\end{lemma}

\begin{proof}
We define a functor $\Rib_{\RepqH}(S)\to\Sk_{\GL_t}(S)\boxtimes\Sk_{\GL_t}(S)$ by separating the green and the red colours on the first and second factor, respectively. By the definition of $\RepqH$, it is  compatible with the skein relations, so it factors through the skein category, yielding the desired equivalence. 
\end{proof}

\subsection{The restriction functor} The restriction functor $$\text{res}_{m,n}\colon\Rep_q\GL_{m+n}\to\Rep_q\L$$ admits also a HOMFLY version that we construct diagrammatically in this section.   We first define a coproduct $\Delta:\Bbbk\to\Bbbk\otimes_{\mathbb{C}(q)}\Bbbk$ by setting
\begin{equation}\label{coproduct}
\Delta\left(q^{t}\right)=q^{t_1}q^{t_2},\quad\Delta(\delta)=\delta_1\; q^{t_2}+q^{-t_1}\;\delta_2
\end{equation}
 on generators and extending $\mathbb{C}(q)$-linearly. 

\begin{lemma}
The map $\Delta:\Bbbk\to\Bbbk\otimes_{\mathbb{C}(q)}\Bbbk$ is well-defined and coassociative.
\end{lemma}
\begin{proof}
This is an easy computation. Firstly,
\begin{equation*}
\begin{split}
\Delta(q^t-q^{-t})&=q^{t_1} q^{t_2}-q^{-t_1} q^{-t_2}\\
&=(q^{t_1}-q^{-t_1}) q^{t_2}+q^{-t_1}( q^{t_2}-q^{-t_2})\\
&=(q-q^{-1})(\delta_1 q^{t_2}+q^{-t_1}\delta_2)=(q-q^{-1})\Delta(\delta),
\end{split}
\end{equation*}
so the map is well-defined. For the coassociativity, we have \begin{equation*}
(1\otimes\Delta)\circ\Delta(q^t)=q^{t_1} q^{t_2} q^{t_3}=(\Delta\otimes 1)\circ\Delta(q^t),
\end{equation*}
and
\begin{equation*}
\begin{split}
(1\otimes\Delta)\circ\Delta(\delta)&=\delta_1(q^{t_2} q^{t_3})+q^{-t_1}(\delta_2 q^{t_3}+q^{-t_2}\delta_3)\\
&=(\delta_1 q^{t_2}+q^{-t_1}\delta_2) q^{t_3}+(q^{-t_1} q^{-t_2})\delta_3=(\Delta\otimes 1)\circ\Delta(\delta).
\end{split}
\end{equation*}
\end{proof}

\begin{remark}
Specialising formulas \eqref{coproduct} at $t=m+n$, $t_1=m$ and $t_2=n$, with $m,n\in\mathbb{N}$, we recover the identities $$q^{m+n}=q^m q^n,\qquad [m+n]_q=q^n [m]_q+q^{-m}[n]_q,$$ expressing the quantum dimension of the fundamental representation of $U_q(\mathfrak{gl}_{m+n})$ in terms of the quantum dimensions of the fundamental representations of $\Uqglm$ and $\Uqgln$. Again, we may think of $\Delta$ as interpolating these formulas.
\end{remark}

We introduce the following notation:
\begin{equation}\label{box1}
\MyFigure{\boxx}\;\coloneqq\;\MyFigure{\oposx}+\left(q-q^{-1}\right)\MyFigure{\orgipip},
\end{equation}
\vspace{0.3cm}
\begin{equation}\label{box2}
\MyFigure{\boxz}\;\coloneqq\;\MyFigure{\oposz}-\left(q-q^{-1}\right)q^{-t_2}\MyFigure{\ogrevvcoevv},
\end{equation}
\vspace{0.3cm}
\begin{equation}\label{box3}
\MyFigure{drawings/pivotalbox}\;\coloneqq\;\MyFigure{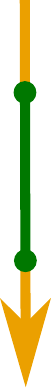}+q^{-t_1}\MyFigure{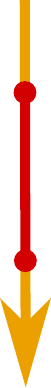},
\end{equation}\label{box4}
\vspace{0.3cm}
\begin{equation}
\MyFigure{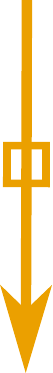}\;\coloneqq\;\left(\MyFigure{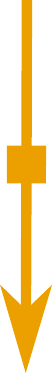}\right)^{-1}\;=\;\MyFigure{drawings/ogipdual.pdf}+q^{t_1}\MyFigure{drawings/oripdual.pdf}.
\end{equation}

\begin{theorem}\label{restrictiontheorem}
 There exists a unique $\Delta$-linear functor 
\begin{equation}\label{restrictionfunctor}
\text{res}_{t}:\RepqG\to\RepqH
\end{equation}
preserving the tensor product and such that
\begin{equation*}
\MyFigure{\blueposx}\mapsto\MyFigure{\boxx},\qquad\MyFigure{\blueposz}\mapsto\MyFigure{\boxz},\qquad\MyFigure{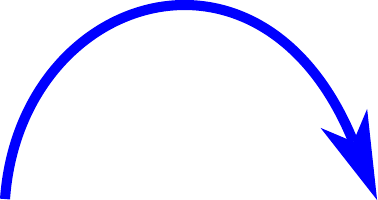}\mapsto\MyFigure{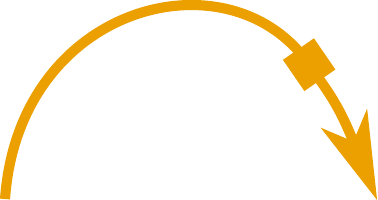},\qquad\MyFigure{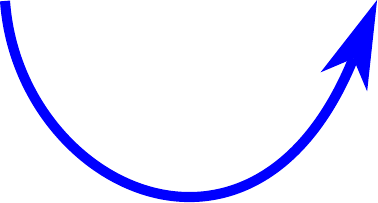}\mapsto\MyFigure{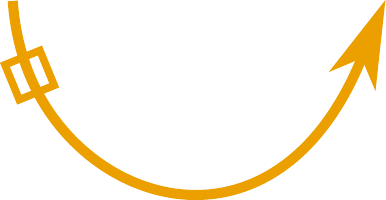}.
\end{equation*}
\end{theorem}

\begin{proof}
Proposition \ref{HOMFLYgenrel} provides a presentation of $\RepqG$ by generators and relations. Since $\text{res}_{t}$ is defined on generators, it will be unique if it is well-defined. We check that it preserves relations \eqref{rel1} -- \eqref{rel5}. For the skein relation \eqref{rel1} and the zigzags \eqref{rel2}, this is an easy computation using the corresponding green and red relations and \eqref{RepqHtrel2}. Applying the functor to the right-hand side of \eqref{rel3}, we have
\begin{equation*}
\text{res}_{t}\left(\MyFigure{\blueinvz}\right)=\MyFigure{\ozinv}+\left(q-q^{-1}\right)q^{t_1}\MyFigure{\ogrevcoev}.
\end{equation*}
Composing with $\MyFigure{\boxz}$ and applying the one-colour twist relations to the terms involving the (co)evaluations appearing, we get the identity, so $\eqref{rel3}$ is preserved. On the other hand,
$$q^{t_1}q^{t_2}\MyFigure{\boxdim}=q^{t_1}q^{t_2}\MyFigure{\gdim}+q^{t_1}q^{t_2}\MyFigure{\rdim}-\left(q-q^{-1}\right)\MyFigure{\rgknot}=\left( q^{t_2}\delta_1+q^{-t_1}\delta_2\right)\boldsymbol{1}_\emptyset,$$ so relation \eqref{rel4} is preserved. Finally, using the relations in Remark \ref{naturality}, we obtain
\begin{equation*}
\begin{split}
\text{ev}_{q,t}\left(\MyFigure{\RIIIa}\right)=\MyFigure{\oRIII}&+\left(q-q^{-1}\right)\left(\MyFigure{\rgoxx}+\MyFigure{\orgxx}\right)\\&+\left(q-q^{-1}\right)^2\left(\;\MyFigure{\orposx}\MyFigure{\ogip}\;+\;\MyFigure{\orip}\MyFigure{\ogposx}\;\right)=\text{ev}_{q,t}\left(\MyFigure{\RIIIb}\right),
\end{split}
\end{equation*}
so the Reidemeister relation \eqref{rel5} is also preserved.
\end{proof}

\begin{remark} Along the same lines as Remarks \ref{universal1} and \ref{universal2}, the functor $\res$ interpolates the restriction functors $\text{res}_{m,n}:\Rep_q\GL_{m+n}\to\Rep_q\L$. That is, the diagram $$\begin{tikzcd}[column sep = large]
\RepqG \arrow[r, "\res"] \arrow[d, "{\text{ev}^{\text{G}}_{m+n}}"'] & \RepqH \arrow[d, "{\text{ev}^{\text{L}}_{m,n}}"] \\
\Rep_q\GL_{m+n} \arrow[r, "{\text{res}_{m,n}}"]          & {\Rep_q\L}                       
\end{tikzcd}$$ commutes for any $m,n\in\mathbb{N}.$
\end{remark}

\begin{remark}\label{specialistationL}
The restriction functor $\res$ is strictly monoidal, but it does not preserve the distinguished dualities. The natural isomorphism from \eqref{Fduals} is determined by $$\varphi_{\upp}=\MyFigure{drawings/pivotalbox.pdf}=\MyFigure{drawings/ogipdual.pdf}+q^{-t_1}\MyFigure{drawings/oripdual.pdf}.$$ Moreover, $\res$ is not pivotal: the categories $\RepqG$ and $\RepqH$ are endowed with trivial pivotal structures, but the isomorphism $$\eta^r_\upp\colon\orup=\res(\upp)^{**}\xrightarrow{(\varphi_\upp^r)^{*}}\res(\upp^*)^*\xrightarrow{(\varphi^r_\downn)^{-1}}\res(\upp^{**})\xrightarrow{\res(\text{id}_\downn)}\res(\upp)=\orup$$ from \eqref{Fpivotal} is represented by 
\begin{equation}
\eta^r_\upp= q^{t_2}\MyFigure{\ogip}+q^{-t_1}\MyFigure{\orip},
\end{equation}
which is not the identity of $\orup$ in $\RepqH$. Similarly, for the left dualities induced by the ribbon structures of $\RepqG$ and $\RepqH$, we get an isomorphism
\begin{equation}
\eta^l_\upp=q^{-t_2}\MyFigure{\ogip}+q^{t_1}\MyFigure{\orip}.
\end{equation}
\end{remark}

\begin{remark}\label{naturality}
Using \eqref{oposx} and the definition of the boxes, one easily checks that 
\begin{equation*}
\MyFigure{\boxgooox}=\MyFigure{\gooox},\qquad\MyFigure{\boxooorx}=\MyFigure{\ooorx},\qquad\MyFigure{\boxroooz}=\MyFigure{\roooz},\qquad\MyFigure{\boxooogz}=\MyFigure{\ooogz},
\end{equation*}
but, for instance,
\begin{equation*}
\MyFigure{\boxogoox}\neq\MyFigure{\ogoox},\qquad\MyFigure{\boxrooox}\neq\MyFigure{\rooox}.
\end{equation*} This is a diagrammatic counterpart of the fact that the $R$-matrix of $U_q(\mathfrak{gl}_{m+n})$ belongs to (a completion of) $U_q(\mathfrak{b}^+)\otimes U_q(\mathfrak{b}^-)$, where $\mathfrak{b}^\pm$ are the corresponding positive and negative Borel subalgebras. Consequently, the restriction of the braiding of $\Rep_q\GL_{m+n}$ to $\Rep_q\L$ is natural in the first factor (resp. in the second) for morphisms of $U_q(\mathfrak{l})$-modules commuting with the action of $U_q(\mathfrak{p})$ (resp. $U_q(\mathfrak{p}^-)$).
\end{remark}

\subsection{The planar category $\RepqP$}\label{repqpsection} Finally, we define a category $\RepqP$ based on the graphical calculus for rigid categories introduced in Section \ref{graphical_calculus}. The category $\Rep_q\P$ of locally finite $U_q(\mathfrak{p})$-modules does not admit a pivotal structure compatible with those of $\Rep_q\GL_{m+n}$ and $\Rep_q\L$. This is because the restriction of the pivotal structure of $\Rep_q\GL_{m+n}$ is not natural for morphisms from $\Rep_q\L$. Hence, we define a universal version of $\Rep_q\P$ using the (planar) graphical calculus of rigid categories. As we will see, this calculus extended to surfaces once a framing has been chosen.

As in Section \ref{graphical_calculus}, we consider rectilinear isotopy classes of vertical planar graphs embedded in $\R\times[0,1]$. Edges are coloured with either $\bsquare$, $\gsquare$ or $\redsquare$, with endpoints labelled by integers following the same rules as defined in that section. Vertices are coupons admitting the following decorations:

\begin{equation}\label{repqp2}
\MyFigure{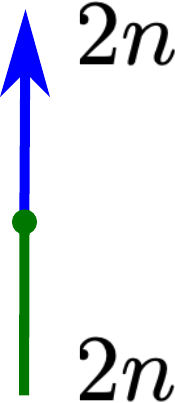},\quad\MyFigure{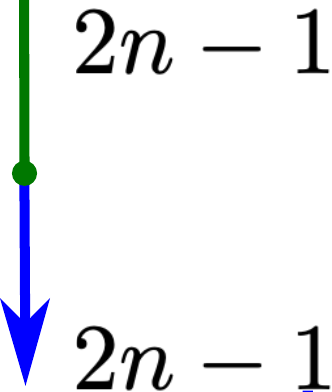},\quad\MyFigure{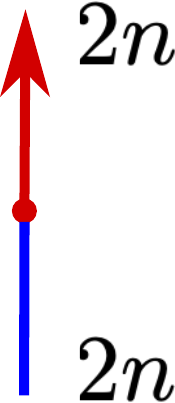},\quad\MyFigure{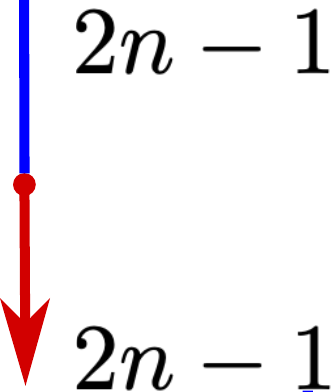},\quad n\in\mathbb{Z};
\end{equation}
\begin{equation}\label{repqp3}
\MyFigure{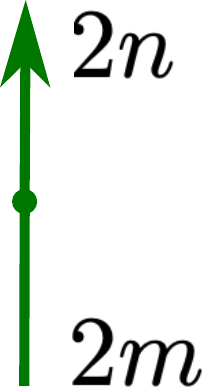},\quad\MyFigure{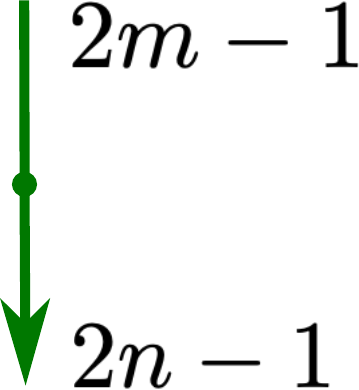},\quad\MyFigure{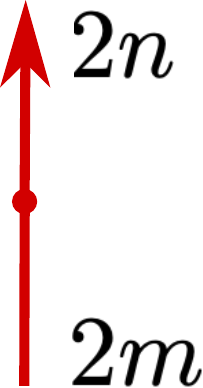},\quad\MyFigure{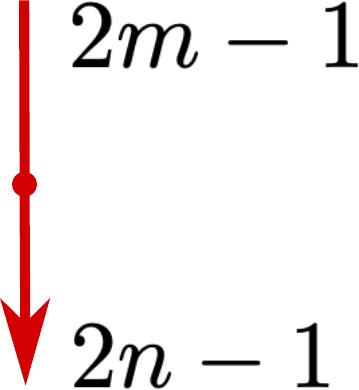},\quad m,n\in\mathbb{Z},\; m\neq n;
\end{equation}
\begin{equation}\label{repqp4}
\MyFigure{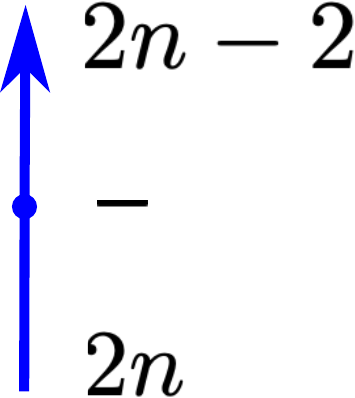},\quad\MyFigure{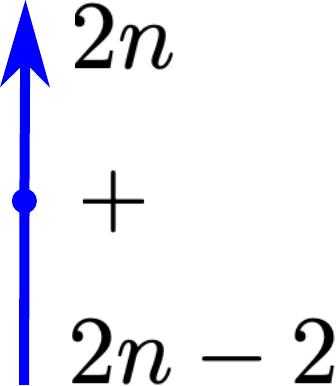},\quad\MyFigure{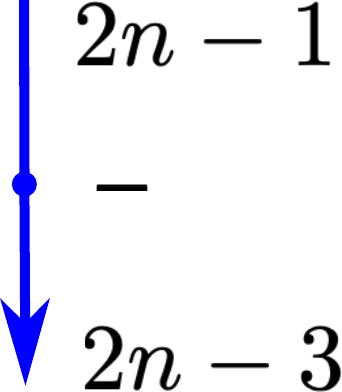},\quad\MyFigure{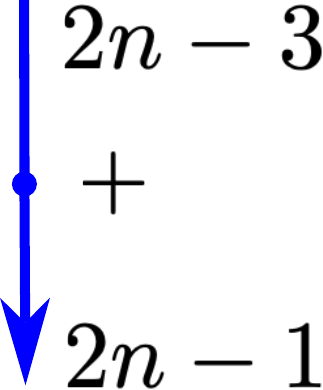},\quad n\in\mathbb{Z}_{\geq 1};
\end{equation}
\begin{equation}\label{repqp5}
\quad\MyFigure{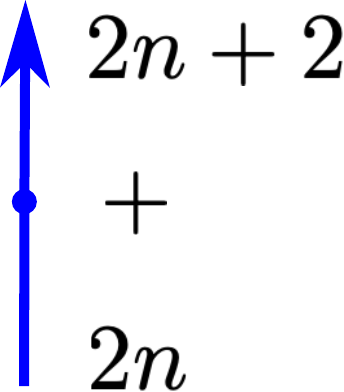},\quad\MyFigure{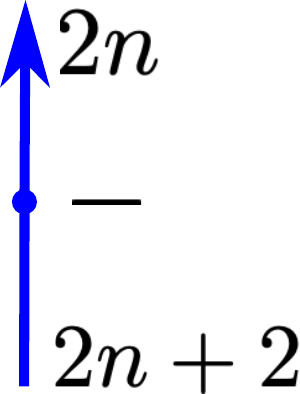},\quad\MyFigure{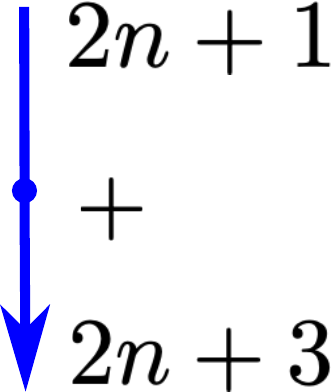},\quad\MyFigure{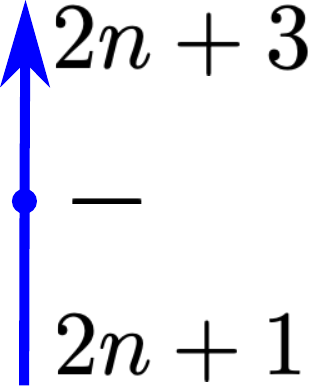},\quad n\in\mathbb{Z}_{\leq -1};
\end{equation}
\begin{equation}\label{repqp6}
\MyFigure{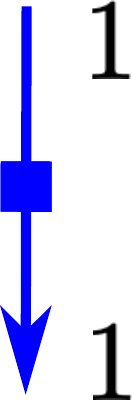},\qquad\MyFigure{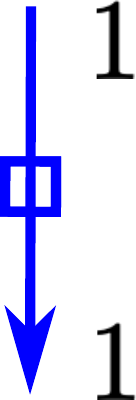}; 
\end{equation}
\begin{equation}\label{repqp7}
\MyFigure{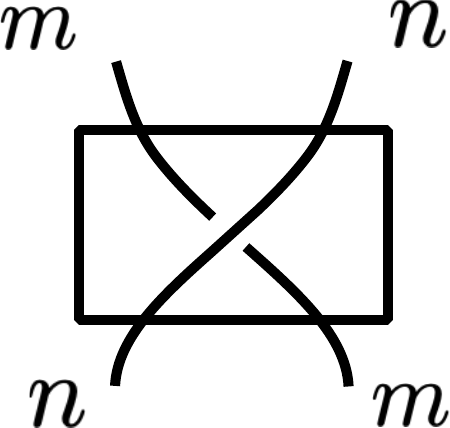},\quad\MyFigure{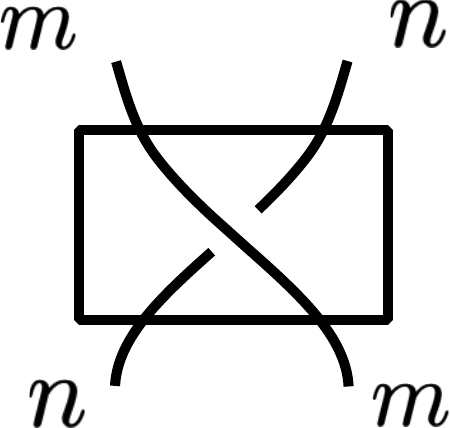}\quad\text{with $m,n\in\mathbb{Z}$ and all possible colourings and orientations}. 
\end{equation}

Let $\text{Rig}(\P)$ be the $(\Bbbk\otimes_{\mathbb{C}(q)}\Bbbk)$-linear category whose objects are sequences of oriented points with integer values compatible with orientations (cf. Section \ref{graphical_calculus}) and coloured with $\bsquare$, $\gsquare$ and $\redsquare$. The hom-spaces are spanned by rectilinear isotopy classes of planar graphs decorated as above. For instance, the diagram $$\includegraphics[scale=0.5, valign = c]{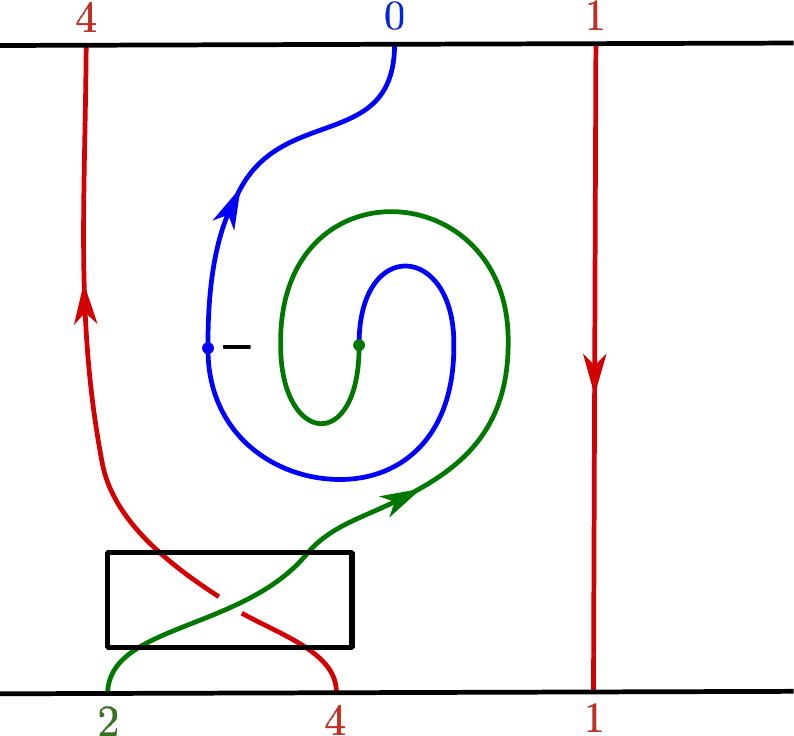}$$ represents a morphism in $\text{Rig}(\P)$. Note that the decoration of the components of the graph is uniquely determined by the source and the target.

We define a functor $$\ev^\P:\text{Rig}(\P)\to\RepqH$$ as follows. On objects, it just forgets the integer labels and switches colours $\bsquare\rightsquigarrow\orsquare$. On morphisms, we interpret generators as follows:
\begin{itemize}
\item the morphisms in \eqref{repqp2} represent the inclusion $\gsquare\equiv V_m\to V\equiv\bsquare$ and the projection $\bsquare\equiv V\to V_n\equiv\redsquare$, so the evaluation functor just forgets the integer labels and switches colours $\bsquare\rightsquigarrow\orsquare$; for instance, $$\MyFigure{drawings/ngbinc.pdf}\mapsto\MyFigure{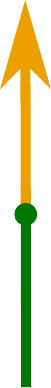};$$
\item the morphisms in \eqref{repqp3} represent the restriction of the pivotal structure of $\RepqH$, which is trivial, so they are all sent to the corresponding identity;
\item the morphisms in \eqref{repqp4} represent the image of the pivotal structure of $\RepqG$ by the restriction functor (see \eqref{Fpivotal}), hence their images by the evaluation functor are  $$\MyFigure{drawings/pivotal.pdf}\;\mapsto\;\eta^r_\upp=q^{t_2}\MyFigure{\ogip}+q^{-t_1}\MyFigure{\orip} \quad\text{and}\quad\MyFigure{drawings/pivotaldual.pdf}\;\mapsto\;(\eta^r_\upp)^*=q^{t_2}\MyFigure{drawings/ogipdual.pdf}+q^{-t_1}\MyFigure{drawings/oripdual.pdf},$$ and the two remaining morphisms are the inverses of these ones;
\item similarly, for $n\in\mathbb{Z}_{\leq-1}$, we set $$\MyFigure{drawings/pivotall}\;\mapsto\;\eta_\upp^l=q^{-t_2}\MyFigure{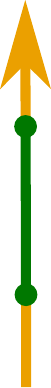}+q^{t_1}\MyFigure{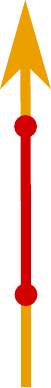}\quad\text{and}\quad\MyFigure{drawings/pivotalldual.pdf}\;\mapsto\;(\eta^l_\upp)^*=q^{-t_2}\MyFigure{drawings/ogipdual.pdf}+q^{t_1}\MyFigure{drawings/oripdual.pdf},$$ and the two remaining morphisms in \eqref{repqp5} are the inverses of these ones;
\item the morphisms in \eqref{repqp6} represent the natural isomorphism from \eqref{Fduals} identifying the duality induced by the restriction functor with the distinguished duality of $\RepqH$, so $$\MyFigure{drawings/rightiso.pdf}\mapsto\;\varphi^r_\upp=\MyFigure{drawings/pivotalbox.pdf}\qquad\text{and}\qquad\MyFigure{drawings/rightisoinv.pdf}\mapsto\;(\varphi^r_\upp)^{-1}=\MyFigure{drawings/pivotalboxinv.pdf},$$ where $\varphi^r$ is the isomorphism in \eqref{box4};
\item for crossings whose strands are both blue, we apply the same rules as for computing the restriction functor \eqref{restrictionfunctor}, i.e.,$$
\MyFigure{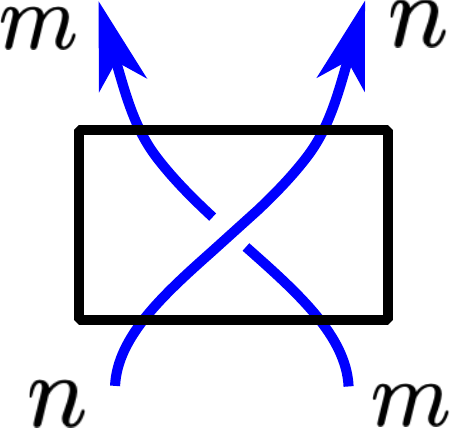}\;\mapsto\;\MyFigure{\oposx}+\left(q-q^{-1}\right)\MyFigure{\orgipip},$$ 
$$\MyFigure{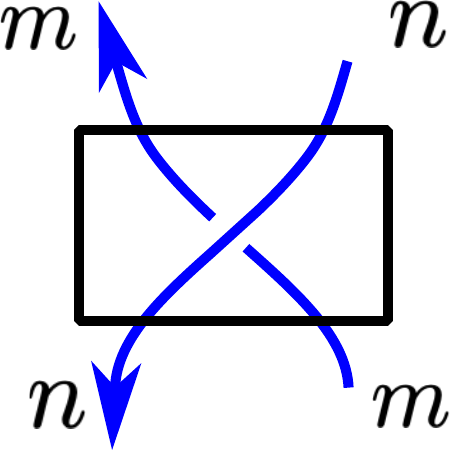}\;\mapsto\;\MyFigure{\oposz}-\left(q-q^{-1}\right)q^{-t_2}\MyFigure{\ogrevvcoevv};$$
\item for the remaining morphisms in \eqref{repqp7}, the evaluation functor just replaces the coupon by the underlying crossing.
\end{itemize}

\begin{lemma}
The evaluation functor $\ev^\P:\text{Rig}(\P)\to\RepqH$ is well-defined.
\end{lemma}

\begin{proof}
This is straightforward, since rectilinear isotopy relation holds in $\RepqH$.
\end{proof}

\begin{definition}
We define the category $\RepqP$ as the quotient of $\text{Rig}(\P)$ by the kernel of $\ev^\P.$
\end{definition}

By construction, the evaluation functor induces a faithful functor 
\begin{equation}\label{inclusion}
j_t^*:\RepqP\to\RepqH
\end{equation}
analogous to the restriction functor $j^*\colon\Rep_q\P\to\Rep_q\L.$

\begin{proposition}
Let $n,m\in\mathbb{N}$ and $\varphi_{m,n}\colon\Bbbk\otimes_{\mathbb{C}(q)}\Bbbk\to\mathbb{C}(q)$ the evaluation morphism defined in Remark \ref{universal2}. Then, there exists a monoidal $\varphi_{m,n}$-linear functor $$\text{ev}^\P_{m,n}\colon\RepqP\to\Rep_q\P$$ mapping $$(\upp,2n)\mapsto\iota^*(V_{m+n}),\quad (\gup, 2n)\mapsto\pi^*(V_m),\quad (\rup,2n)\mapsto \pi^*(V_n),$$ where $\iota^*$ and $\pi^*$ are the restriction functors defined in Section \ref{parabolic_restriction}.
\end{proposition}

\begin{proof}
The restriction functor $j^*\colon\Rep_q\P\to\Rep_q\L$ is faithful so we may describe $\Rep_q\P$ as a subcategory of $\Rep_q\L$. By construction, the image of $$\RepqP\xrightarrow{j^*_t}\RepqH\xrightarrow{\text{ev}^\L_{m,n}}\Rep_q\L$$ lies in this subcategory, so it lifts to a monoidal functor $\text{ev}_{m,n}^\P$ as in the statement.
\end{proof}

We obtain a commutative diagram of functors  $$\begin{tikzcd}
\RepqP \arrow[r, "j^*_t"] \arrow[d, "{\text{ev}^\P_{m,n}}"'] & \RepqH \arrow[d, "{\text{ev}^\text{L}_{m,n}}"] \\
\Rep_q\P \arrow[r, "j^*"]              & \Rep_q\L,                
\end{tikzcd}$$
hence we may interpret again the functor $j^*_t$ as interpolating the functors $j^*$, for $m,n\in\mathbb{N}.$

\section{Universal parabolic restriction}\label{parabolic}

In this section, we construct the algebraic inputs that we will use later for defining our defect skein theory. In the first part of the section, we build a $(\RepqG,\RepqH)$-central algebra structure of $\RepqP$, which is a universal version of the one described in \ref{parabolic_restriction}. Next, we describe a centred $(\RepqH,\RepqP)$-bimodule structure on $\RepqH$.

\subsection{Central algebra structure}\label{central_algebra} We define functors HOMFLY versions $$\iota_t^*:\RepqG\to\RepqP\qquad \text{and} \qquad\pi_t^*:\RepqG\boxtimes\RepqG\to\RepqP$$ of the restriction functors $\Rep_q\GL_{m+n}\xrightarrow{\iota^*}\Rep_q\P$ and $\Rep_q\L\xrightarrow{\pi^*}\Rep_q\P$ from Section \ref{parabolic_restriction}. Let us introduce first some terminology that will be also useful later to extend the graphical calculus for rigid categories to framed surfaces. 

Let $S$ be a framed surface, i.e., a surface $S$ endowed with a trivialization $f=\left(v_p,w_p\right)_{p\in S}$ of its tangent bundle. Let $\alpha:I\to S$ be an immersed curve such that $\dot{\alpha}(0)$ and $\dot{\alpha}(1)$ are in the direction of $w_{\alpha(0)}$ and $w_{\alpha(1)}$, respectively. Consider the map $$u^f_\alpha\colon I\xrightarrow{\dot{\alpha}}\text{T}S\setminus\{0\}\xrightarrow{f} S\times(\R^2\setminus\{0\})\xrightarrow{\text{proj.}}\R^2\setminus\{0\}\xrightarrow{\frac{x}{||x||}}\mathbb{S}^1.$$ By definition, $u^f_\alpha(0)=\pm i$ and we can set $$u^f_\alpha(t)=\exp\left(-i\pi\left(\theta^f_\alpha(t)\mp\frac{1}{2}\right)\right),$$ with $\theta^f_\alpha(0)=0$. The conditions imposed on $\alpha$ imply that $\theta^f_\alpha(1)\in\mathbb{Z}$.

\begin{definition}
The \emph{rotation number} of $\alpha$ with respect to $f$ is the integer $$\rot^f(\alpha)\coloneqq\theta_\alpha^f(1)\in\mathbb{Z}.$$
\end{definition}

For the rest of this subsection, we fix $S=\R\times[0,1]$ with its canonical framing. From a given ribbon diagram $\Gamma$ in $\R\times[0,1],$ we construct a decorated planar diagram $\widehat{\Gamma}$ as follows. Any open strand of $\Gamma$ defines an immersed curve $\alpha:I\to\R\times[0,1]$. If $\alpha$ is oriented upwards at $\alpha(0)$, we decorate this endpoint with $n_{\alpha(1)}=0$; otherwise, we set $n_{\alpha(0)}=1$. A decoration for $\alpha(1)$ is then uniquely determined by setting $n_{\alpha(1)}=n_{\alpha(0)}+\rot^f(\alpha)$. Next, if $\alpha$ is a closed strand (i.e. a link component) with rotation number $r_\alpha=\rot^f(\alpha)$, we add $\left|\frac{r_\alpha}{2}\right|$ dots at $\alpha(0)$. Finally, replace every crossing by the corresponding coupon. For example: $$\Gamma\;=\;\MyFigure{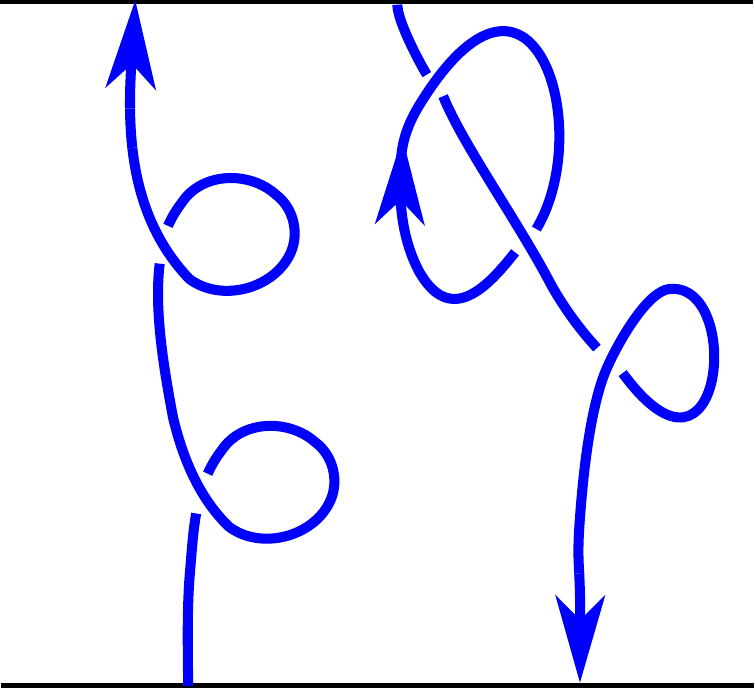}\qquad\rightsquigarrow\qquad\MyFigure{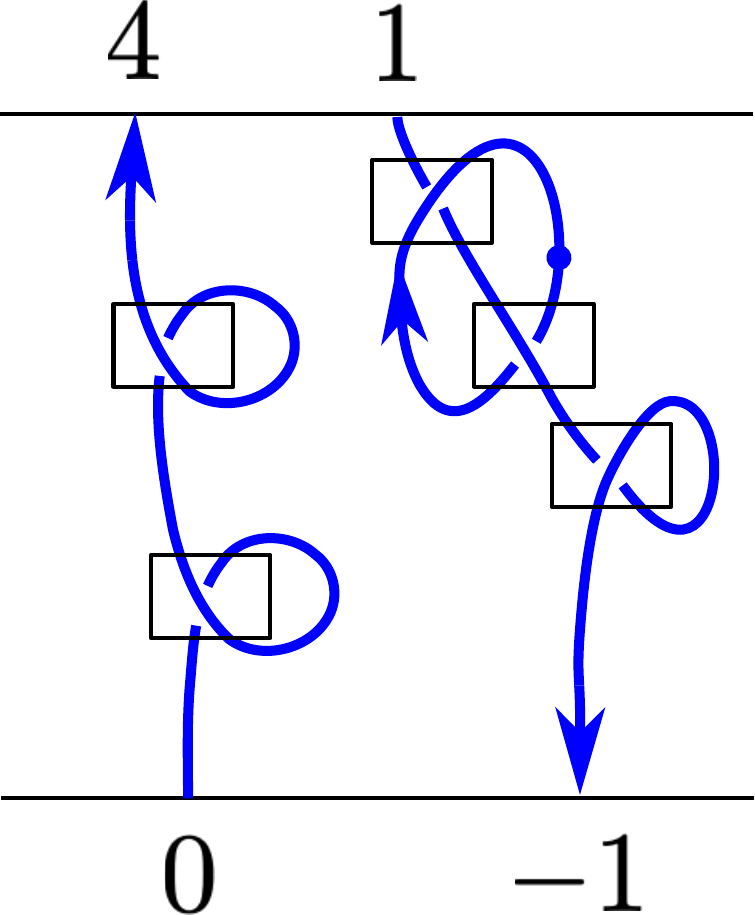}.$$ Precomposing and postcomposing again with the morphisms in \eqref{repqp4}--\eqref{repqp5}, we can modify the source and target of the planar graph obtained so that points oriented upwards (resp. downwards) are all decorated with $0$ (resp. $-1$).  Let $\hat{\Gamma}$ be the decorated planar graph thus obtained. For instance, for the ribbon diagram depicted above, we get $$\hat{\Gamma}=\MyFigure{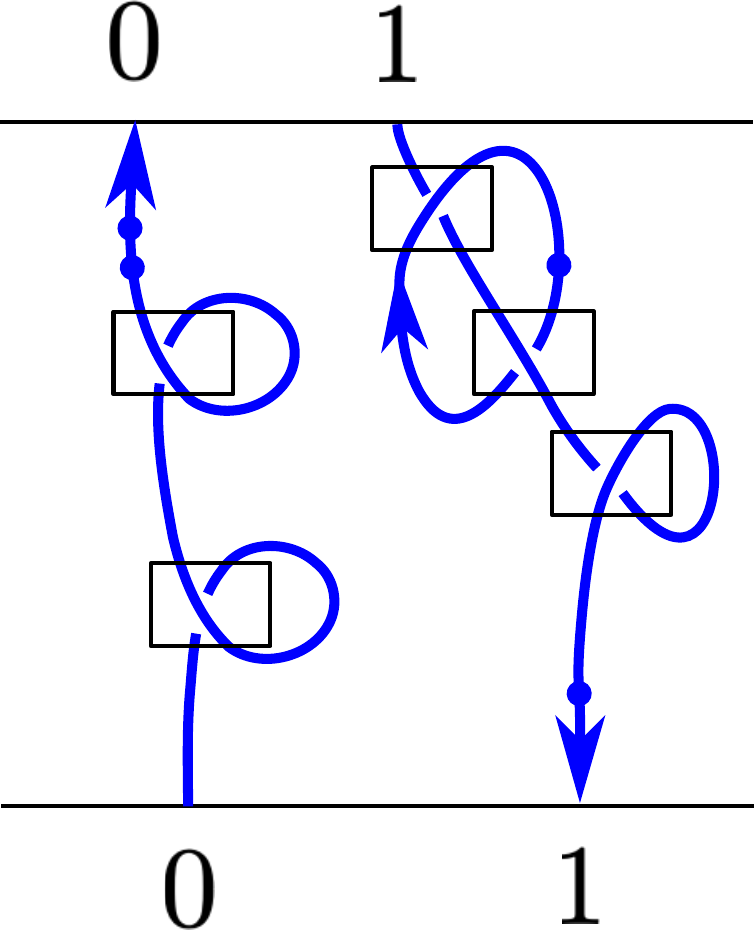}.$$ Roughly speaking, this procedure amounts to forget the braided and the pivotal structure of $\RepqG$, so that the planar graph $\widehat{\Gamma}$ represents the same morphism as $\Gamma$ but in the rigid category underlying the ribbon category $\RepqG$. 

Recall that we have a restriction functor $\iota^*\colon\Rep_q\GL_{m+n}\to\Rep_q\P.$ This functor is faithful, but it does not preserve the distinguished dualities. We reflect this fact diagrammatically by modifying the distinguished duality for $\upp$ as follows. Set $$\varphi_0\coloneqq\MyFigure{\blueid}\qquad\text{and}\qquad\varphi_1\coloneqq\MyFigure{drawings/rightiso.pdf}.$$ Given a sequence $\varepsilon=(\varepsilon_1,\ldots,\varepsilon_k)$ with $\varepsilon_i\in\{0,1\},$ we define 
\begin{equation}\label{dualiso}
\varphi_\varepsilon\coloneqq\varphi_{\varepsilon_1}\otimes\cdots\otimes\varphi_{\varepsilon_k}.
\end{equation}
 Finally, if $D:\varepsilon_1\to\varepsilon_2$ is a planar diagram between two such sequences, we define $$\iota^*_t(D)\coloneqq\varphi_{\varepsilon_2}\circ D\circ\varphi^{-1}_{\varepsilon_1}.$$ 

\begin{proposition}
The assignment $\upp\mapsto\textcolor{blue}{0}$, $\downn\mapsto\textcolor{blue}{1}$, $\textcolor{blue}{\Gamma}\mapsto\iota^*_t(\textcolor{blue}{\widehat{\Gamma}})$ yields a well-defined strict monoidal functor 
\begin{equation}\label{restriction_functorP}
\iota_t^*\colon\RepqG\to\RepqP.
\end{equation}
 Moreover, $\res=j_t^*\circ\iota_t^*,$ where $j_t^*$ is the functor in \eqref{inclusion}.
\end{proposition}

\begin{proof}
First note that the isomorphisms $\eta^r_\upp$, $\eta^l_\upp$ and $\varphi^r_\upp$ commute with every endomorphism of $\upp$ in $\RepqH$. This implies that $\iota_t^*(\widehat{\Gamma_1\circ\Gamma_2})=\iota_t^*(\widehat{\Gamma}_1)\circ\iota_t^*(\widehat{\Gamma}_2)$ for every pair of composable diagrams $\Gamma_1,\Gamma_2$. Indeed, we can slide the dots appearing between $\Gamma_1$ and $\Gamma_2$ in $\iota_t^*(\widehat{\Gamma}_1)\circ\iota_t^*(\widehat{\Gamma}_2)$ so that they all lie at the end of the corresponding strand. By replacing each  dot - inverse dot pair appearing this way by the identity, we get exactly $\iota_t^*(\widehat{\Gamma_1\circ\Gamma_2})$. For instance, $$\iota^*_t(\widehat{\Gamma}_1)\circ\iota^*_t(\widehat{\Gamma}_2)=\MyFigure{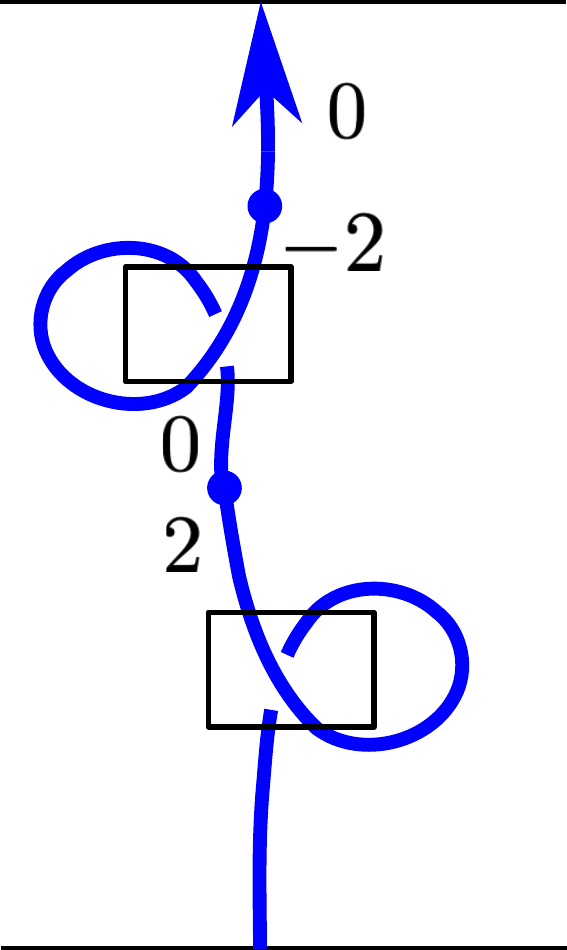}\quad=\quad\MyFigure{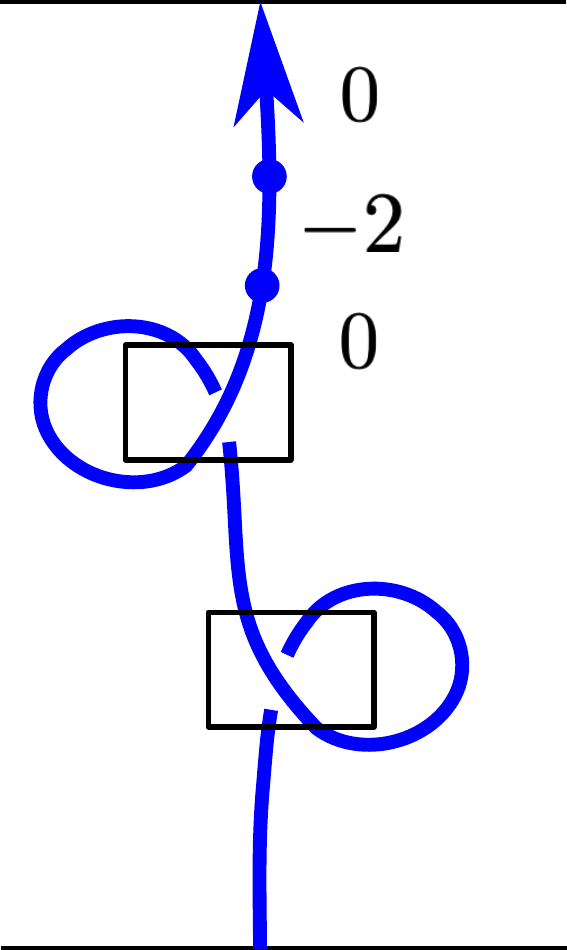}\quad=\quad\MyFigure{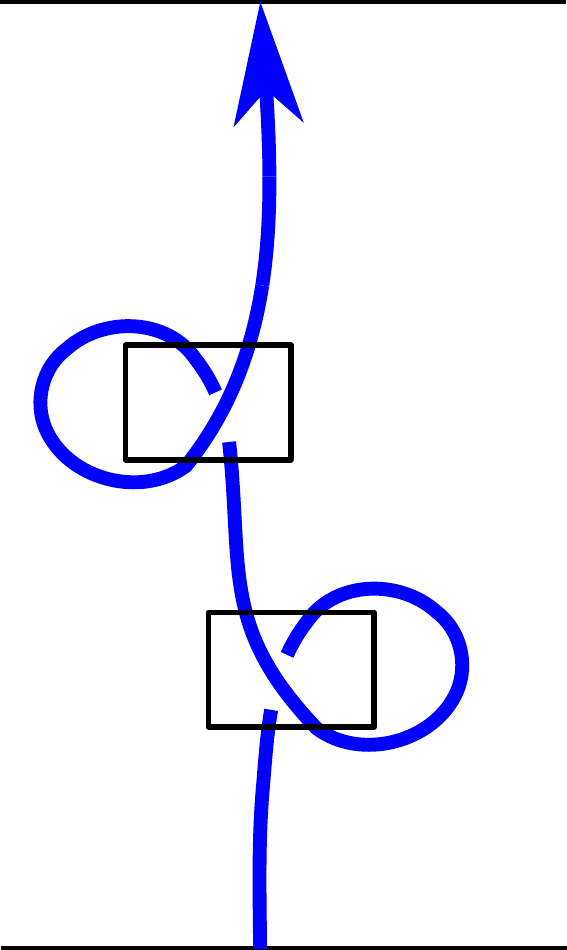}=\iota^*_t(\widehat{\Gamma_1\circ\Gamma_2}).$$
Moreover, we trivially have that $\iota^*_t(\widehat{\Gamma_1\otimes\Gamma_2})=\iota_t^*(\widehat{\Gamma}_1)\otimes\iota_t^*(\widehat{\Gamma}_2),$ hence the functor is strictly monoidal.

On the other hand, it is straightforward to check that, if $\Gamma$ is one of the generators of $\RepqG$ (cf. Proposition \ref{HOMFLYgenrel}), then $j_t^*(\Gamma)\circ\iota_t^*(\widehat{\Gamma})=\res(\Gamma)$. The compatibility with the composition and the tensor product implies then that the same is true for any diagram $\Gamma$. In particular, if $\Gamma_1$ and $\Gamma_2$ are two diagrams representing the same morphism in $\RepqG$, then $$j_t^*\circ\iota_t^*(\widehat{\Gamma}_1)=\res(\Gamma_1)=\res(\Gamma_2)=j_t^*\circ\iota_t^*(\widehat{\Gamma}_2).$$  Since $j_t^*$ is faithful, we get that $\iota_t^*(\widehat{\Gamma}_1)=\iota_t^*(\widehat{\Gamma}_2),$ which proves that the functor is well-defined.
\end{proof}

\begin{proposition}
The assignment $$(\gup,\emptyset)\mapsto\textcolor{mygreen}{0},\quad(\gdown,\emptyset)\mapsto\textcolor{mygreen}{1},\quad\textcolor{mygreen}{\Gamma}\mapsto\textcolor{mygreen}{\widehat{\Gamma}},$$ $$(\emptyset,\rup)\mapsto\textcolor{myred}{0},\quad(\emptyset,\rdown)\mapsto\textcolor{myred}{1},\quad\textcolor{myred}{\Gamma}\mapsto\textcolor{myred}{\widehat{\Gamma}},$$ yields a well defined strict monoidal functor 
\begin{equation}\label{restriction_functor_HtoP}
\pi^*_t\colon\RepqH\to\RepqP.
\end{equation}
\end{proposition}

\begin{proof}
The assignment is functorial and strictly monoidal by the same arguments as in the previous proof. The fact that it is well-defined follows again from the fact that $j_t^*$ is faithful. Indeed, we have that $j_t^*\circ\pi_t^*=\id_{\RepqH}$. So, if $\Gamma_1$ represent the same morphism $\Gamma_2$ in $\RepqH$, then the equality $$j^*_t\circ\pi_t^*(\Gamma_1)=\Gamma_1=\Gamma_2=j^*_t\circ\pi_t^*(\Gamma_2)$$ implies that $\pi_t^*(\Gamma_1)=\pi^*_t(\Gamma_2).$
\end{proof}

To sum up, we have four functors
\begin{equation}\label{allthefunctors}
\begin{tikzcd}
                                                             & \RepqP \arrow[rd, "j_t^*"'] &                                                    \\
\RepqG \arrow[ru, "{\iota_t^*}"] \arrow[rr, "\res"] &                             & \RepqH \arrow[lu, "\pi^*_t"', bend right]
\end{tikzcd}, 
\end{equation}
such that the inner triangle is commutative, but not the outer one. This is a diagrammatic version of the situation described in Section \ref{parabolic_restriction}. We will see now that the functors $\iota_t^*$ and $\pi^*_t$ induce a $(\RepqG, \RepqH)$-central algebra structure on $\RepqP$. By Remark \ref{naturality} and the relations defining $\RepqH$, the families of morphisms given by  $$\MyFigure{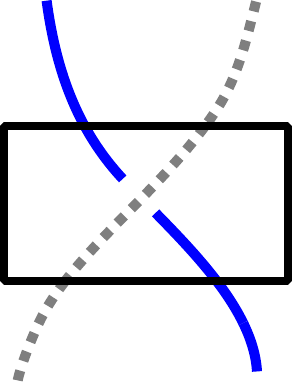},\qquad\MyFigure{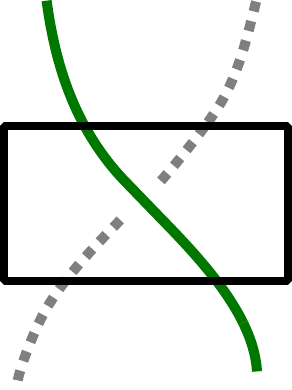}\qquad\text{and}\qquad\MyFigure{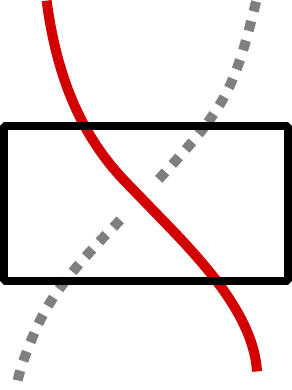},$$ where the dotted strands can be replaced by any colour, are natural isomorphisms in $\RepqP.$ Hence, they define a monoidal functor $$Z\colon\RepqP\to Z(\RepqP).$$ Let
\begin{equation}
\begin{array}{cccc}
F\coloneqq\iota^*_t\boxtimes\pi_t^*: & \RepqG\boxtimes\RepqH & \to & \RepqP,\\
& (\textcolor{blue}{\boldsymbol{u}},\textcolor{mygreen}{\boldsymbol{v}},\textcolor{myred}{\boldsymbol{w}}) & \mapsto & \textcolor{blue}{\boldsymbol{u}}\textcolor{mygreen}{\boldsymbol{v}}\textcolor{myred}{\boldsymbol{w}}\\
& (\textcolor{blue}{\Gamma_1},\textcolor{mygreen}{\Gamma_2},\textcolor{myred}{\Gamma_3}) & \mapsto & \iota^*_t(\textcolor{blue}{\widehat{\Gamma}_1})\textcolor{mygreen}{\widehat{\Gamma}_2}\textcolor{myred}{\widehat{\Gamma}_3}.
\end{array}
\end{equation}

\begin{theorem}\label{central_functor}
$F$ lifts to a braided monoidal functor $$(Z\circ F,J)\colon \RepqG\boxtimes\overline{\RepqH} \to Z(\RepqP),$$ where $\overline{(-)}$ stands for the opposite braided category. In particular, $\RepqP$ is a $(\RepqG,\RepqH)$-central algebra.
\end{theorem}

\begin{proof}
The family of isomorphisms $$\MyFigure{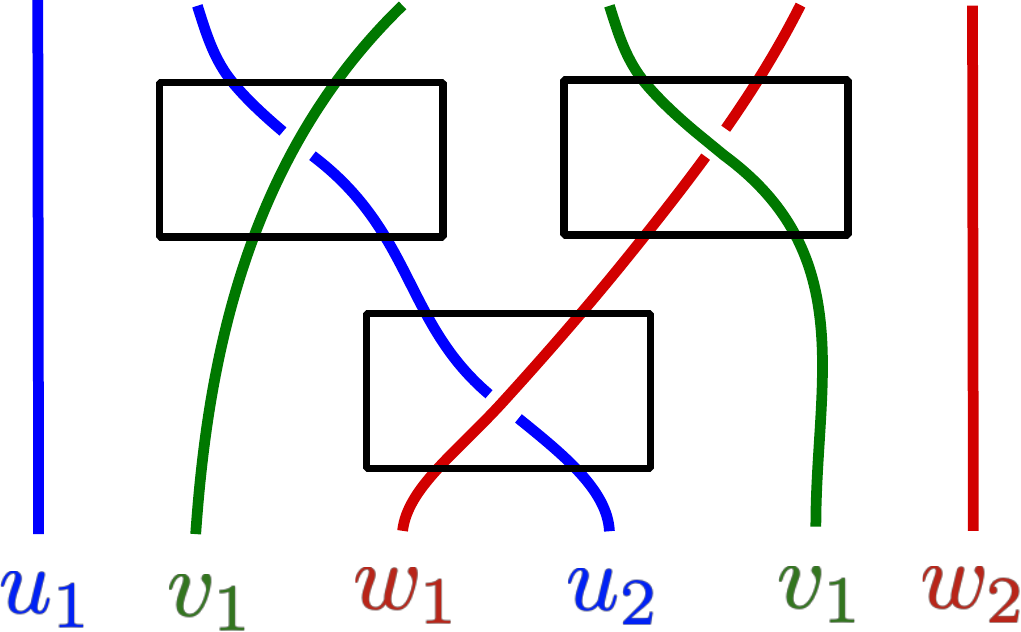}:\; F(\textcolor{blue}{\boldsymbol{u_1}},\textcolor{mygreen}{\boldsymbol{v_1}},\textcolor{myred}{\boldsymbol{w_1}})\otimes F(\textcolor{blue}{\boldsymbol{u_2}},\textcolor{mygreen}{\boldsymbol{v_2}},\textcolor{myred}{\boldsymbol{w_2}})\;\to\;F(\textcolor{blue}{\boldsymbol{u_1u_2}},\textcolor{mygreen}{\boldsymbol{v_1v_2}},\textcolor{myred}{\boldsymbol{w_1w_2}})$$ defines a monoidal structure on $Z\circ F$. Checking that $(Z\circ F,J)$ is braided is an easy computation.
\end{proof}

Note that, for an object $\textcolor{blue}{\boldsymbol{u}}\textcolor{mygreen}{\boldsymbol{v}}\textcolor{myred}{\boldsymbol{w}}$, the half-braiding is given by
$$\MyFigure{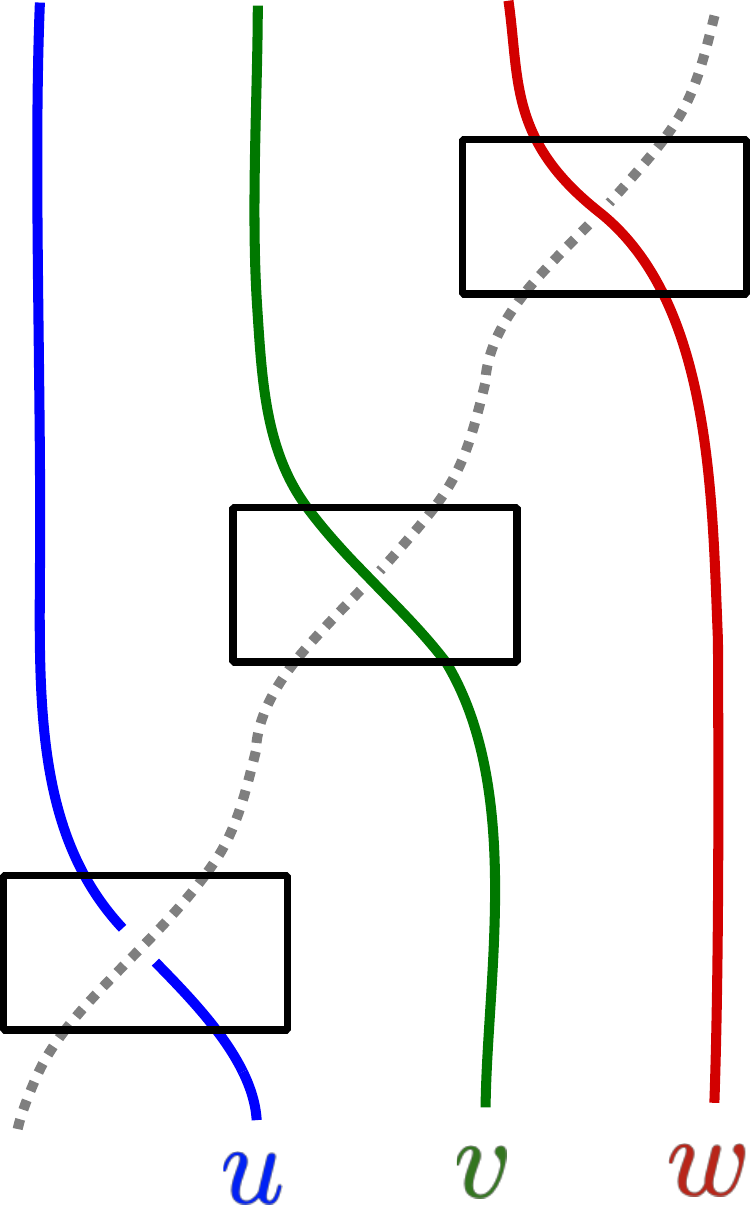}:\; -\otimes\textcolor{blue}{\boldsymbol{u}}\textcolor{mygreen}{\boldsymbol{v}}\textcolor{myred}{\boldsymbol{w}}\;\to\;\textcolor{blue}{\boldsymbol{u}}\textcolor{mygreen}{\boldsymbol{v}}\textcolor{myred}{\boldsymbol{w}}\otimes -.$$ 

\subsection{The centred bimodule $\RepqH$}\label{centred_bimodule} We will now briefly describe additional structure induced by the functors constructed in the previous section, which will serve as motivation for the topological construction we will give later in this paper.  In addition to the $(\RepqG,\RepqH)$-central structure on $\RepqP$ defined in the previous section, we will also consider the following three central algebras:
\begin{itemize}
\item $\RepqG$ is a $(\Vect_\Bbbk,\RepqG)$-central algebra via the central functor $$\Vect_\Bbbk\boxtimes\RepqG\simeq\RepqG\xrightarrow{\text{id}}\RepqG,$$
\item $\RepqH$ is a $(\Vect_\Bbbk,\RepqH)$-central algebra with central functor $$\Vect_\Bbbk\boxtimes\RepqH\simeq\RepqH\xrightarrow{\text{id}}\RepqH,$$
\item and finally $\RepqH$ is also a $(\RepqH,\RepqH)$-central algebra with central structure $$\RepqH\boxtimes\RepqH\xrightarrow{\text{id}\otimes\text{id}}\RepqH.$$
\end{itemize}
These four algebras define $1$-morphisms in the Morita $4$-category $\textsc{BrTens}$ studied in \cite{BJS}. Composing $\RepqP$ with $\RepqG$, we obtain a central functor $$\Vect_\Bbbk\boxtimes\RepqH\to\RepqG\mathop{\boxtimes}\limits_{\RepqG}\RepqP\simeq\RepqP,$$ where the action of $\RepqH$ is that induced by the functor $\pi^*_t.$ Similarly, composing the two central structures on $\RepqH$, we have another central functor $$\Vect_\Bbbk\boxtimes\RepqH\to\RepqH\mathop{\boxtimes}\limits_{\RepqH}\RepqH\simeq\RepqH,$$ where $\RepqH$ acts on itself via the identity functor. On the other hand, note that $\RepqH$ is a ($\RepqP,\RepqH)$-bimodule via the functor $$\RepqP\boxtimes\RepqH\xrightarrow{j_t^*\otimes\text{id}}\RepqH$$ and it is straightforward to check that the braiding induces a $(\RepqP,\RepqH)$-centred structure (see \cite[Section 3]{BJS} for the definitions). Centred structures are the $2$-morphisms in $\textsc{BrTens}$, so we have the following diagram in $\textsc{BrTens}$: $$\begin{tikzcd}[column sep = large, row sep = large]
\Vect_\Bbbk \arrow[r, "\RepqG"] \arrow[d, "\RepqH"'] & \RepqG \arrow[d, "\RepqP"] \arrow[ld, "\RepqH" description, Rightarrow] \\
\RepqH \arrow[r, "\RepqH"]                           & \RepqH.                                                           
\end{tikzcd}$$ In the language of factorisation algebras (see \cite[Section 3]{BJS} and the figures therein), these structures are governed by embedding disks in the following stratified space:

\begin{figure}[H]
\centering
\includegraphics[scale=0.3]{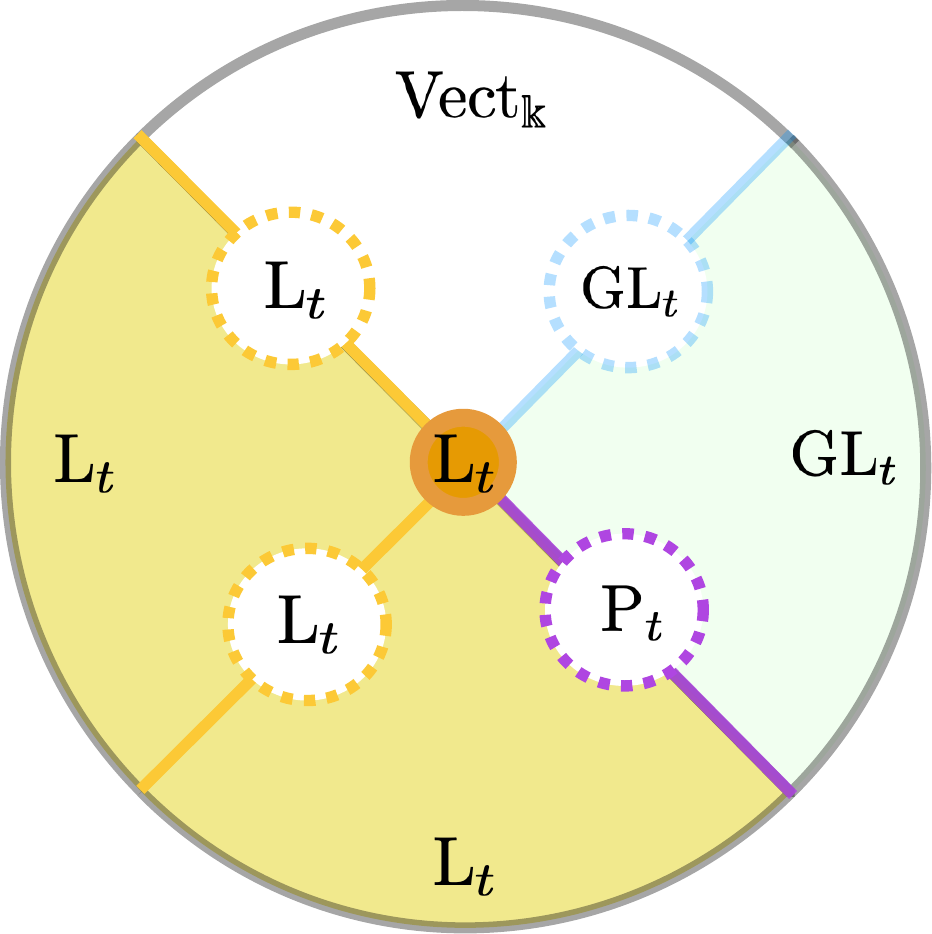}
\caption{Centred $(\RepqP,\RepqH)$-bimodule structure of $\RepqH$ as a factorisation algebra.}
\end{figure}

\section{Defect planar theories} \label{planar_theories}

The graphical calculus for rigid categories extends to surfaces with a chosen framing, producing 2-dimensional framed TFTs from a rigid monoidal category’s algebraic data. The $(\RepqH,\RepqP)$-module structure on $\RepqH$ produces a morphism between the theories associated with $\RepqH$ and $\RepqP$, that we will describe in this section using line defects on surfaces.

\subsection{One-coloured theories}\label{TFT} We describe here the planar theory induced by $\RepqP$. A similar description holds for the rigid category underlying $\RepqH.$

\begin{definition}
Let $M$ be a $d$-dimensional oriented manifold, and let $n \geq d$. An \emph{$n$-dimensional framing} of $M$ is defined as a homotopy class of orientation-preserving isomorphisms $f: TM \oplus \mathbb{R}^{n-d} \to M \times \mathbb{R}^n$ between vector bundles over $M$. When $n = d$, we refer to this as a \emph{framing}.
\end{definition}

Note that if $M=I$ or $M=\mathbb{S}^1$, each choice of orientation admits exactly one $2$-framing. We write $I$, $\mathbb{S}^1$ for the standard positive orientation and $\overline{I}$, $\overline{\mathbb{S}^1}$ for the negative one. 

\begin{definition}
A \emph{marking} on an oriented surface $S$ is a (possibly empty) collection $\mathcal{B}$ of orientation-preserving embeddings $C\hookrightarrow\partial S,$ with $C$ a $1$-dimensional connected oriented manifold. Every marking has a splitting $\mathcal{B}=\mathcal{B}_+\sqcup\mathcal{B}_-$ into positively and negatively oriented submanifolds.
\end{definition}

\begin{definition}
\label{rectilinear_embedding} A \emph{rectilinear embedding} is an embedding $M\hookrightarrow N$ between framed manifolds which preserves the framing up to rescaling in each direction.
\end{definition}

Let $(S,f,\mathcal{B})$ be a marked framed surface. We consider the following class of planar graphs $\Gamma$ embedded in $S$ (see Figure \ref{fig:decorated_planar_graph}):
\begin{itemize}
\item edges are coloured with $\bsquare$, $\gsquare$ or $\redsquare$;
\item the endpoints $\alpha(0)$ and $\alpha(1)$ of every edge $\alpha$ are attached to either a coupon or one of the marked boundary components. Moreover, they are labelled with integers $n_{\alpha(0)}$ and $n_{\alpha(1)}$ such that $n_{\alpha(1)}=n_{\alpha(0)}+\rot^f(\alpha)$;
\item coupons are rectilinearly embedded in the surface and decorated with morphisms of $\RepqP.$
\end{itemize}

\begin{figure}[t]
\centering
\includegraphics[scale=0.8]{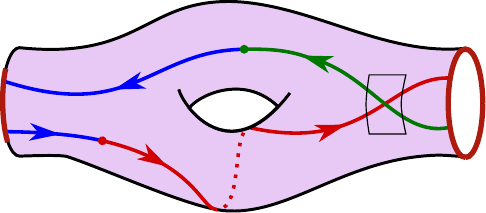}
\caption{Decorated planar graph on a framed surface (we omit the integer labels for simplicity).}
\label{fig:decorated_planar_graph}
\end{figure}

Set $\mathcal{B_+}=\{C_i\}_{i=1,\ldots,m}$ and $\mathcal{B}_-=\{C'_j\}_{j=1,\ldots,n}.$ The intersection of $\Gamma$ with the marking determines a family of configurations of points $(a_{C_1},\ldots,a_{C_m},b_{C'_1},\ldots,b_{C'_n})$ that we refer to as \emph{boundary conditions}. We denote by $\mathcal{Z}^{\P_t}(S,f,\mathcal{B})(a_{C_1},\ldots,a_{C_m},b_{C'_1},\ldots,b_{C'_n})$ the $(\Bbbk\otimes_{\mathbb{C}(q)}\Bbbk)$-linear space generated by planar graphs with given boundary conditions, modulo the following relations:
\begin{itemize}
\item \underline{rectilinear isotopies}: these are isotopies of the surface $S$ fixing $\partial S$ such that coupons remain parallel to the framing throughout the isotopy;
\item \underline{framed skein relations}: $\sum_i\lambda_i(\varphi,\Gamma_i)\sim 0$ if there exists a rectilinear embedding $\iota:[0,1]^2\hookrightarrow S$ such that $\sum_i\lambda_i\iota^{-1}\left(\Gamma\cap\iota([0,1])\right)=0$ in $\RepqP$.
\end{itemize}

As a particular case, we have:
\begin{definition}
Let $(\mathbb{A},f_{\text{rad}})$ be the annulus endowed with its radial framing. The \emph{cylinder category} $\text{Cyl}(\RepqP)$ is the category with:
\begin{itemize}
\item \underline{objects}: finite configurations of coloured directed points on the circle $\mathbb{S}^1$ decorated by integers compatibly with the orientation;
\item \underline{morphisms}: $(\Bbbk\otimes_{\mathbb{C}(q)}\Bbbk)$-linear combinations of planar diagrams on $(\mathbb{A},f_\text{rad}),$ modulo rectilinear isotopy and framed skein relations.
\end{itemize}
Composition is given by inserting one copy of $\mathbb{A}$ into another one.
\end{definition}

We set $$\mathcal{Z}^{\P_t}(I)=\RepqP\quad\text{and}\quad\mathcal{Z}^{\P_t}(\mathbb{S}^1)=\text{Cyl}(\RepqP),$$ and negatively oriented $1$-manifolds are assigned the opposite categories. If $C\in\mathcal{B}_+,$ then $\mathcal{Z}^{\P_t}(C)$ acts on $\mathcal{Z}^{\P_t}(S,f,\mathcal{B})$ by gluing rectangles/cylinders. Namely, every configuration of points on $C$ determines an object $b$ of $\mathcal{Z}^{\P_t}(C)$. If $\Gamma\colon a\to b$ is a morphism in $\mathcal{Z}^{\P_t}(C)$, then gluing $C\times I$ along the marking induces a linear map $$\mathcal{Z}^{\P_t}(S,f,\mathcal{B})(\Gamma)\colon\mathcal{Z}^{\P_t}(S,f,\mathcal{B})(-,b,-)\to\mathcal{Z}^{\P_t}(S,f,\mathcal{B})(-,a,-).$$ The analogue holds for negative orientations, so we obtain a functor $$\mathcal{Z}^{\P_t}(S,f,\mathcal{B}):\left(\bigboxtimes_{C\in\mathcal{B}_+}\mathcal{Z}^{\P_t}(C)\right)\boxtimes\left(\bigboxtimes_{\overline{C}\in\mathcal{B}_-}\mathcal{Z}^{\P_t}(C)\right)^\text{op}\to\Vect.$$

\begin{lemma}
Let $(S,f,\mathcal{B})$ be a marked framed surface with $C,\overline{C}\in\mathcal{B}$ for some $1$-dimensional manifold $C$. Let $\left(\gl_C(S),\gl_C(f),\gl_C(\mathcal{B})\right)$ be the marked framed surface obtained by gluing $S$ along $C$. Then, \begin{equation*}
\mathcal{Z}^{\P_t}\left(\gl_C(S),\gl_C(f),\gl_C(\mathcal{B})\right)(-,-)=\int^{c\in\mathcal{Z}^{\P_t}(C)}\mathcal{Z}^{\P_t}(S,f,\mathcal{B})(-,c,c,-).
\end{equation*}
\end{lemma}

\begin{proof}
The proof is the same as in \cite[Theorem 4.4.2]{walker}.
\end{proof}

We have now all the ingredients to define a framed TFT with target category $\textsc{Bimod}$. Let $\text{Bord}^\text{fr}_1(2)$ be the category of two-dimensional framed cobordisms, whose objects are $1$-dimensional $2$-framed manifolds and whose morphisms are marked framed cobordisms. It follows straightforwardly that:

\begin{theorem}
The construction above defines a symmetric monoidal functor $$\mathcal{Z}^{\P_t}\colon\text{Bord}_1^\text{fr}(2)\to\textsc{Bimod},$$ hence a 2-dimensional framed TFT. \QEDA
\end{theorem}

Considering the category $\RepqH$ just as a rigid monoidal category, we can define an associated $2$-dimensional framed TFT that can be described along the same lines as $\mathcal{Z}^{\P_t}$. We will denote this theory  by $\mathcal{Z}^{\H_t}.$

\subsection{Planar model around the defect line}\label{bimodule} The ($\RepqH$, $\RepqP$)-bimodule structure of $\RepqH$ has a topological interpretation in terms of planar diagrams on a stratified square. We ignore for now its $(\RepqH,\RepqP)$-centred structure, which will be described later in Section \ref{local_model_strat} when we address three-dimensional theories on surfaces.  Recall the functor $j^*\colon\RepqP\to\RepqH$ defined in \eqref{inclusion}. It induces two morphisms in \textsc{Bimod}, namely the bimodule functors $$F_1\coloneqq\Hom_{\RepqH}(-,j^*(-))\colon\RepqP\boxtimes \RepqH^\op\to\Vect$$ and $$F_2\coloneqq\Hom_{\RepqH}(j^*(-),-)\colon\RepqH\boxtimes\RepqP^\op\to\Vect.$$ 

Let $I^*$ be the unit interval with a point marked in the middle and $E=I\times I^*$ a square with a horizontal 
 line defect $I\times\left\{\frac{1}{2}\right\}$. We will call the \emph{$\H_t$-region} (resp. \emph{$P_t$-region}) the half-square under (resp. above) the defect. The defect itself will be decorated by $\H_t$ as well and we will call it the \emph{$\H_t$-line}. We consider planar graphs $\Omega$ in $E$ of the following form (see Figure \ref{fig:strat_square}):
\begin{itemize}
\item endpoints are attached to the top and bottom bases of the square;
\item $\Omega$ is coloured by $\RepqP$ (resp. by $\RepqH$) on the $\P_t$-region (resp. on the $\H_t$-region);
\item edges can meet transversally the line defect  at a coupon decorated with a morphism of $\RepqH$ in a compatible way: if the edges coming from the $\H_t$-region (resp. the $\P_t$-region) are coloured with $X_1,\ldots,X_m$ (resp. $Y_1,\ldots, Y_k$), then the coupon is decorated with a morphism $f\colon X_1\otimes\cdots\otimes X_m\to j^*(Y_1)\otimes\cdots\otimes j^*(Y_k).$
\end{itemize}

\begin{figure}[t]
\centering
\includegraphics[scale=0.35]{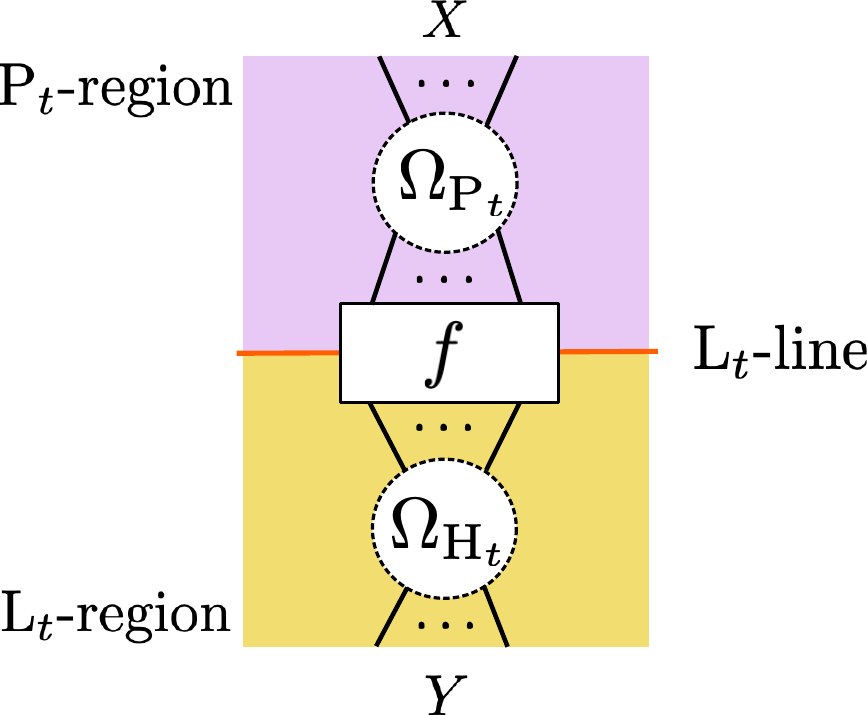}
\caption{Planar graph on a stratified square.}
\label{fig:strat_square}
\end{figure}

Fixing $X$ and $Y$, the assignment $$\includegraphics[scale=0.35, valign=c]{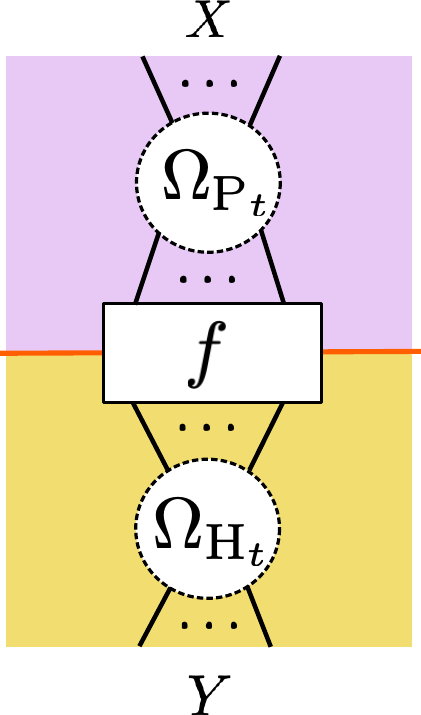}\quad\mapsto\quad j^*(\Omega_{\P_t})\circ f\circ \Omega_{\H_t}\in\Hom_{\RepqH}(Y,j^*(X)),$$ extends to a linear map defined on the $(\Bbbk\otimes_{\mathbb{C}(q)}\Bbbk)$-linear space spanned by those diagrams. Its kernel is generated by:
\begin{itemize}
\item planar skein relations induced by $\RepqH$ and $\RepqP$ on each side of the defect line;
\item relations of the form 
\begin{equation}\label{morphrels}
\includegraphics[scale=0.35, valign=c]{drawings/strat_eva.pdf}\quad=\quad\includegraphics[scale=0.35, valign=c]{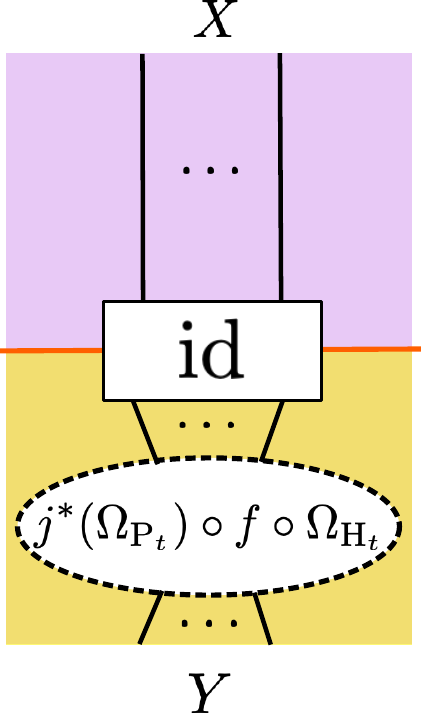}\;.
\end{equation} 
\end{itemize}
Taking the quotient by all of them, we have a diagrammatic description of $F_1$, where the actions of $\RepqP$ and $\RepqH$ are just given by stacking morphisms above and below. Note that, since every stratified diagram $\Omega$ represents a morphism in $\RepqH$, the relations in \eqref{morphrels} implies that $\Omega$ is equivalent to another diagram $\widetilde{\Omega}$ where $\widetilde{\Omega}_{\P_t}$ and $f$ are just identities, that is:
\begin{equation*}
\includegraphics[scale=0.35, valign=c]{drawings/strat_eva.pdf}\quad=\quad\includegraphics[scale=0.35, valign=c]{drawings/strat_rel.pdf}\;.
\end{equation*} 
We will not represent the identity coupon in the pictures, so whenever a set of strands crosses the $\H_t$-line (changing their colour), they must be thought of as attached to an identity coupon. The bimodule $F_2$ can be described equally by just exchanging the $\P_t$ and the $\H_t$-regions in the diagrams.  Similarly, we have functors $$G_1\colon\text{Cyl}(\RepqP)\boxtimes\text{Cyl}(\RepqH)^\op\to\Vect,$$ $$G_2\colon\text{Cyl}(\RepqH)\boxtimes\text{Cyl}(\RepqP)^\op\to\Vect,$$ mapping an object $X\boxtimes Y$ to the linear space spanned by planar diagrams with boundary conditions $(X,Y)$ on a stratified cylinder $\mathbb{S}^1\times I^*,$ modulo skein relations on each side of the defect and the local relations depicted in \eqref{morphrels} near the defect.
 
\begin{lemma}
In $\textsc{Bimod}$, we have $$F_2\circ F_1=\text{id}_{\RepqH}\quad\text{and}\quad G_2\circ G_1=\text{id}_{\text{Cyl}(\RepqH)}.$$
\end{lemma}
\begin{proof}
By definition, $$(F_2\circ F_1)(X\boxtimes Y)=\int^{Z\in\RepqP}\Hom_{\RepqH}(X,j^*(Z))\otimes\Hom_{\RepqH}(j^*(Z),Y),$$ and the canonical map $f\otimes g\to g\circ f$ yields an isomorphism between the coend on the right-hand side and $\Hom_{\RepqH}(X,Y).$ The same holds for $G_2\circ G_1.$
\end{proof}

\begin{remark}\label{modifing defect}
We can interpret the proof pictorially. The coend can be represented via diagrams in a square with two horizontal $\H_t$-defects. The region between them is decorated by $\P_t$, while the regions above and below are decorated by $\H_t$. Applying relations \eqref{morphrels}, we have $$\includegraphics[scale=0.3, valign = c]{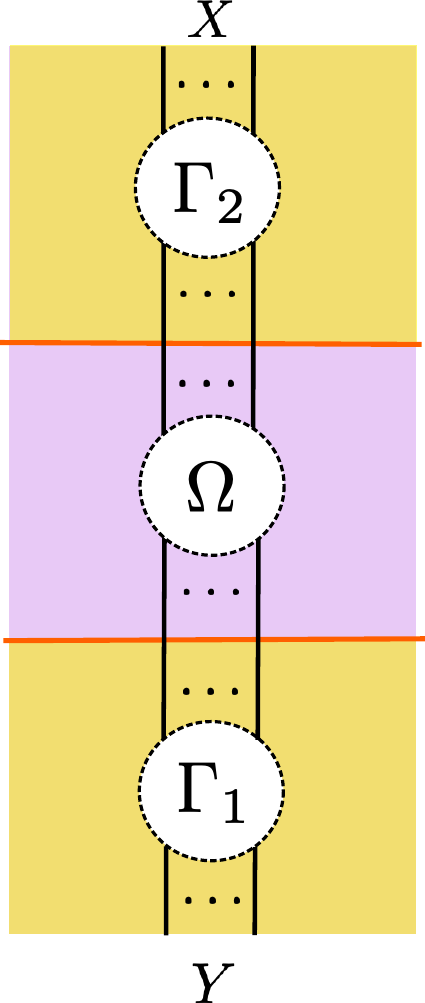}\quad=\quad\includegraphics[scale=0.3, valign = c]{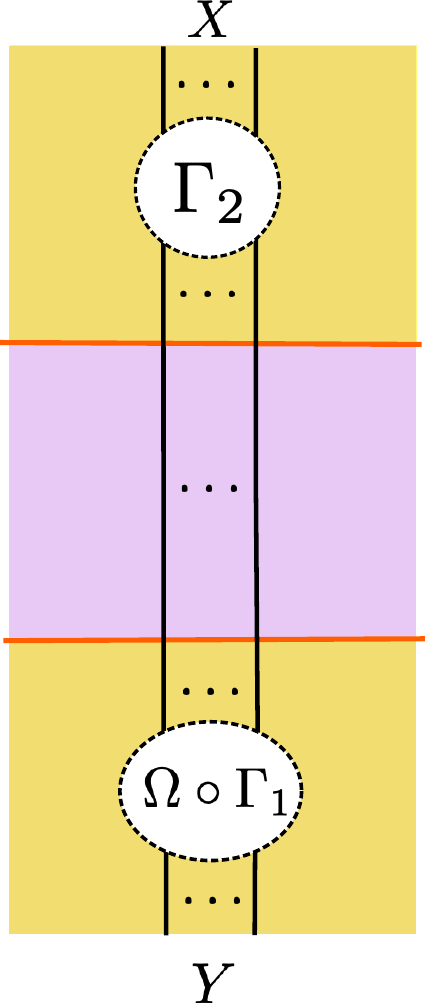}\quad\equiv\quad\includegraphics[scale=0.3, valign = c]{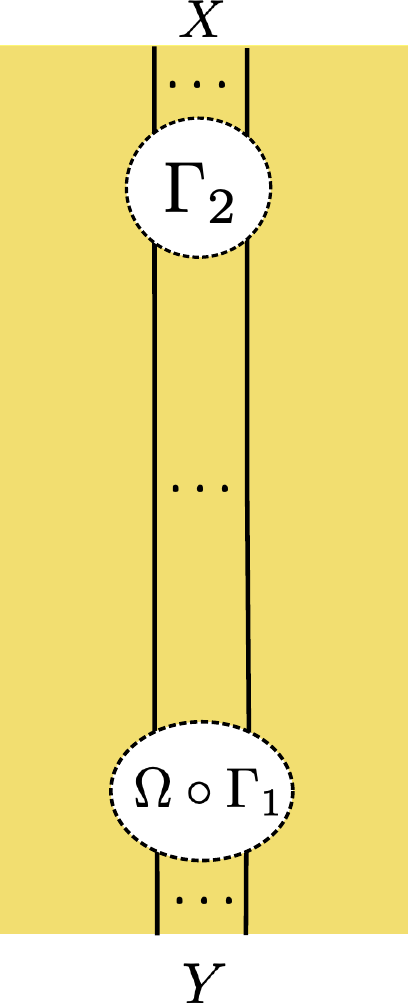}\;.$$ 
This suggests that it may be possible to modify stratifications making regions appear/disappear  decorated by $\P_t$ in the planar theory associated with $\H_t$. As we will see, the $\GL_t$-skein algebra acts on the planar theory associated with $\P_t$. We may use this to induce an action of the $\GL_t$-skein algebra on the $\H_t$-theory, by making $\P_t$-regions appear and acting on them. These ideas will be made precise later on by describing topologically the central algebras and the centred bimodule from Sections \ref{central_algebra} and \ref{centred_bimodule}.
\end{remark}

\subsection{Planar defect theory} The planar theories $\mathcal{Z}^{\P_t}$ and $\mathcal{Z}^{\H_t}$ induced by the rigid categories $\RepqP$ and $\RepqH$, respectively, glue together via the bimodule structure, yielding a two-dimensional theory on stratified surfaces.  Recall that a bipartite surface (cf. Definition \ref{bipartite}) is a surface $S$ with a stratification $$\phi\colon S\to\left(A\geq B\leq C\right).$$ We label $\phi^{-1}(A)$ and $\phi^{-1}(C)$ with $\H_t$ and $\P_t$, respectively, and we call them the \emph{$\H_t$-region} and the \emph{$\P_t$-region}. The $1$-dimensional stratum $\phi^{-1}(B)$ is decorated with $\L_t$ and we refer to it as the \emph{$\H_t$-line.}

\begin{definition}
A \emph{compatible marking} of $(S,f,\phi)$ is a marking $\mathcal{B}$ such that each of its components is contained either in the $\H_t$-region or the $\P_t$-region. Therefore, we have a splitting $\mathcal{B}=\mathcal{B}_{\L_t}\sqcup\mathcal{B}_{\P_t}$.
\end{definition}

Let $(S,f,\phi,\mathcal{B})$ be a bipartite framed surface with a compatible marking. For each $C\in\mathcal{B}$, we set $$\mathcal{Z}(C)\coloneqq\begin{cases}
\mathcal{Z}^{\H_t}(C), & \text{if $C\in\mathcal{B}_{\H_t}$},\\ 
\mathcal{Z}^{\P_t}(C), & \text{if $C\in\mathcal{B}_{\P_t}$}.\\ 
\end{cases}
$$ Take $U_{C}\in\mathcal{Z}(C)$ for each $C\in\mathcal{B}$ and consider the $(\Bbbk\otimes_{\mathbb{C}(q)}\Bbbk)$-linear space spanned by planar graphs as in Section \ref{TFT}, with given boundary conditions, coloured with $\RepqH$ (resp. $\RepqP$) in the $\H_t$-region (resp. the $\P_t$-region) and meeting the defect transversally at coupons decorated with morphisms from $\RepqH$ as explained in Section \ref{bimodule}. We set $\mathcal{Z}(S,f,\phi,\mathcal{B})\left((U_C)_{C\in\mathcal{B}}\right)$ for its quotient by 
\begin{itemize}
\item skein relations induced by $\RepqH$ and $\RepqP$ in the interior of the $2$-dimensional regions;
\item stratified skein relations around the $\H_t$-line: these are the relations induced by the local model described in the previous section.
\end{itemize}
The assignment $$\begin{array}{cccc}
\mathcal{Z}(S,f,\phi,\mathcal{B})\colon & \left(\bigboxtimes\limits_{C\in\mathcal{B}^+}\mathcal{Z}(C)\right)\boxtimes\left(\bigboxtimes\limits_{\bar{C}\in\mathcal{B}_-}\mathcal{Z}(C)\right)^\op & \to & \Vect,\\ & (U_C)_{C\in\mathcal{B}} & \mapsto & \mathcal{Z}(S,f,\phi,\mathcal{B})\left((U_C)_{C\in\mathcal{B}}\right)
\end{array}$$ is well-defined and functorial, since all the relations from $\RepqH$ and $\RepqP$ hold. We shall write simply $\mathcal{Z}(S)$ whenever there is no risk of confusion.

\begin{lemma}[{\cite[Theorem 4.4.2]{walker}}]
Let $(S,f,\phi,\mathcal{B})$ be a bipartite framed surface with compatible marking.  Let $\gl_C(S)$ be the surface obtained by gluing $S$ along a boundary component $C$, with $C,\overline{C}\in\mathcal{B}_{\P_t}$ (alternatively, in $\mathcal{B}_{\L_t}$). Then, $$\mathcal{Z}(\gl_C(S))(-)=\int^{U\in\mathcal{Z}(C)}\mathcal{Z}(S)(-,U,U),$$ where the components decorated by $U$ on the RHS are $C$ and $\overline{C}$. \QEDA
\end{lemma}

\begin{theorem}\label{planar_theory}
The assignment $$\mathcal{Z}\colon\text{Bord}^{\text{bip,fr}}_1(2)\to\textsc{Bimod}$$ defines a symmetric monoidal functor from the category of $2$-dimensional bipartite framed cobordisms, hence a $2$-dimensional framed TFT.
\end{theorem}

\begin{proof}
By the previous lemma, the construction is compatible with gluing surfaces along boundary components of a single colour.
For the general case, since we are working with compatible markings, it suffices to observe that we may first glue along the subregion coloured by $\P_t$, and then along that coloured by $\L_t.$ The rest is clear.
\end{proof}

\section{Defect three-dimensional theories}\label{3d_theories}

In this section, we extend the theories introduced above to three-dimensional theories. The local models around the surface and line defects are given by the algebraic structures described in Section \ref{parabolic_restriction}.

\subsection{Parabolic surface defects}\label{local_model} 

Following \cite{jennydavid}, we construct a local model for surface defects using the central algebra structure on $\RepqP$ introduced in Section \ref{central_algebra}. Let $S=[0,1]^2$ be endowed with its canonical framing and let $I^*$ be the unit interval marked at $1/2$, so that $E^*\coloneqq[0,1]^2\times I^*$ is a stratified cube with a \emph{framed defect wall} $[0,1]^2\times\left\{1/2\right\}.$ We consider the following class of stratified ribbon graphs in $E^*$ (see Figure \ref{fig:planar}):
\begin{itemize}
\item endpoints are attached to one of the following regions: $$E_0\coloneqq[0,1]\times\{0\}\times\{0\},\qquad E_{1/2}\coloneqq [0,1]\times\{1/2\}\times\{0,1\}\quad\text{or}\quad E_1\coloneqq[0,1]\times\{1\}\times\{1\};$$
\item the defect wall cuts the ribbon graph $\Omega$ into three pieces:
\begin{enumerate}
\item $\Omega^\text{L}\coloneqq\Omega\cap\left([0,1]^2\times\left[0,\frac{1}{2}\right)\right)$ represents a morphism of $\text{SkCat}_{\text{L}_t}([0,1]^2)$;
\item $\Omega^\text{G}\coloneqq\Omega\cap\left([0,1]^2\times\left(\frac{1}{2},1\right]\right)$ represents a morphism of $\text{SkCat}_{\GL_{t}}([0,1]^2)$;
\item $\Omega^\P\coloneqq\Omega\cap\left([0,1]\times\{1/2\}\right)$ is a planar graph as defined in Section \ref{repqpsection}, with the only difference that we allow endpoints to lie in the interior of $S$ if they are attached to either $\Omega^\text{G}$ or $\Omega^\text{L}$;
\end{enumerate}
\item $\Omega^\text{G}$ and $\Omega^\text{L}$ meet the defect transversally at a coupon of $\Omega^\P$ compatible with the colour of the strands.
\end{itemize}

If $\Omega$ is such a ribbon graph, we set $$V\coloneqq\Omega\cap E_0,\quad W\coloneqq\Omega\cap E_1,\quad a\coloneqq\Omega\cap[0,1]\times\{1/2\}\cap\{0\},\quad b\coloneqq\Omega\cap[0,1]\times\{1/2\}\cap\{1\}$$ for the objects of $\RepqG$, $\RepqH$ and $\RepqP$ determined by its endpoints and we say that $\Omega$ is a \emph{stratified $(a,b,V,W)$-ribbon graph on the disk}. Consider the $(\Bbbk\otimes_{\mathbb{C}(q)}\Bbbk)$-linear space $\mathcal{L}_{a,b,V,W}$ spanned by stratified $(a,b,V,W)$-ribbon graphs modulo
\begin{itemize}
\item  local relations occurring on either one side of the defect wall or the defect wall itself. Explicitly, $\sum_i\lambda_i\Omega_i\sim 0$ if any of the following holds: $\sum_i\lambda_i\Omega^\text{G}_i=0$ in $\RepqG$; $\sum_i\lambda_i\Omega_i^\text{L}=0$ in $\RepqH$; or $\sum_i\lambda_i\Omega_i^\P=0$ via the local relations introduced in Section \ref{repqpsection} applied to the defect;
\item \emph{framed stratified isotopy}: that is, by isotopies of $E^*$ fixing the boundary, preserving the interior of the defect and such that coupons in the defect remain parallel to the framing throughout the isotopy.
\end{itemize}
We define an evaluation morphism for stratified ribbon graphs in the following way. To a stratified $(a,b,V,W)$-ribbon graph $\Omega$ as above, we associate a morphism in $\RepqP$ as follows:
\begin{itemize}
\item we modify $\Omega^\text{G}$ as in the definition of the restriction functor \eqref{restriction_functorP}, that is:
\begin{enumerate}
\item we label the initial point of any open strand $\alpha$ by $n_{\alpha(0)}=0$ or $n_{\alpha(0)}=1$ depending on the orientation;
\item we label the endpoint by $n_\alpha+\rot^f{\alpha}$;
\item we use the pivotal structure of $\RepqG$ (morphisms in \eqref{repqp4} and \eqref{repqp5}) to switch the label of every endpoint to $0$ or $1$;
\item we replace the graph $\widetilde{\Omega}^\text{G}$ thus obtained by $\varphi_{\varepsilon_2}^{-1}\circ\widetilde{\Omega}^\text{G}\circ\varphi_{\varepsilon_1},$ where $\varphi_\varepsilon$ is the isomorphism in \eqref{dualiso};
\end{enumerate}
\item we apply the same procedure to $\Omega^\L$ (here the pivotal structure is the trivial one);
\item we switch the decoration of every endpoint of $\Omega^\P$ to either $0$ or $1$ by concatenating with the pivotal structures of $\RepqG$ and $\RepqH$ (morphisms in \eqref{repqp3}, \eqref{repqp4} and \eqref{repqp5});
\item we project the resulting ribbon graph into the defect wall and replace all the crossings appearing by the corresponding half-twist from $Z(\RepqP)$ (cf. Section \ref{TFT}).
\end{itemize}

Let $\widehat{\Omega}$ be obtained by applying this procedure to $\Omega$. Projecting $E_0$ and $E_1$ onto $E_{1/2}$, we have two configurations of labelled points $V_a$ and $W_b$.  The planar graph $\widehat{\Omega}$ represents thus a morphism $f_\Omega\in\Hom_{\RepqP}(V_a,W_b)$ (see Figure  \ref{fig:planar} for an example). We define the \emph{evaluation morphism} by 
\begin{equation}\label{evaluation_morphism}
\begin{array}{cccc}
\text{ev}_{a,b,V,W}: & \mathcal{L}_{a,b,V,W} & \to & \Hom_{\RepqP}(V_a,W_b),\\
& \Omega & \mapsto & f_\Omega,
\end{array}
\end{equation}
on generators and extending linearly.

 \begin{figure}[t]
   \begin{minipage}{0.45\textwidth}
   		\centering
   		\includegraphics[scale=0.3]{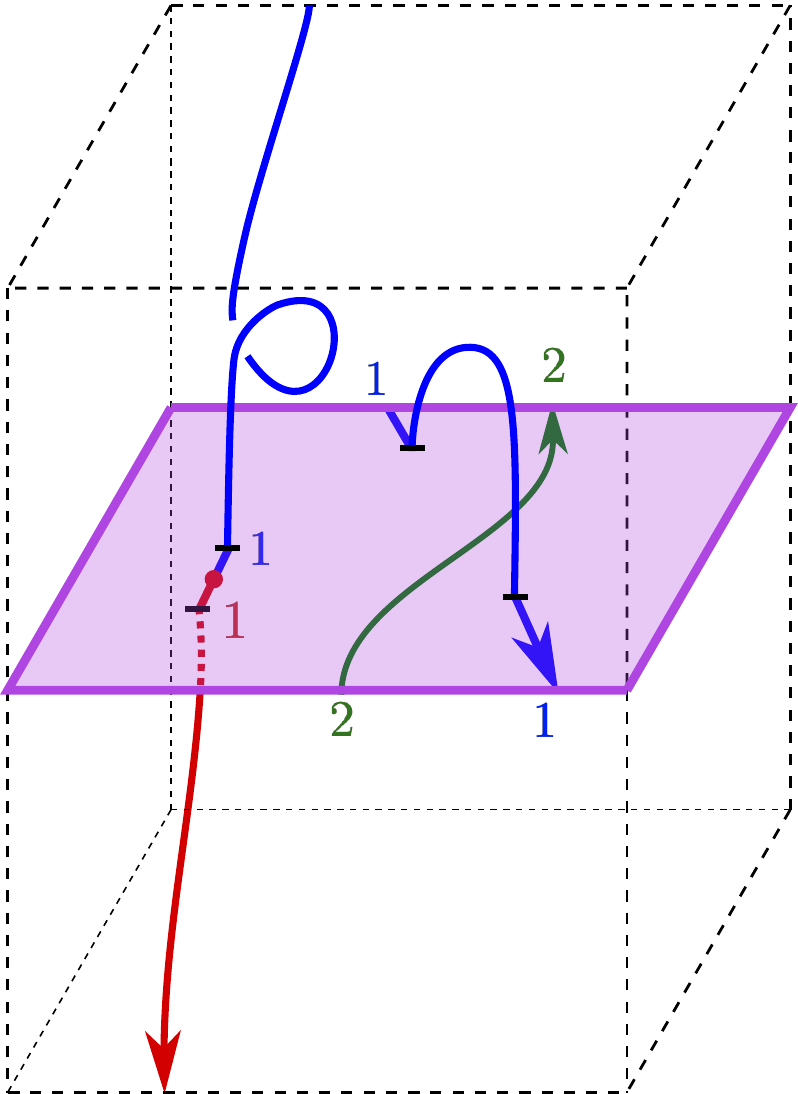}
   \end{minipage}\hfill
   \begin {minipage}{0.1\textwidth}
   		\centering
     	$\mapsto$
   \end{minipage}\hfill
   \begin {minipage}{0.45\textwidth}
   		\centering
     	\includegraphics[scale=0.25]{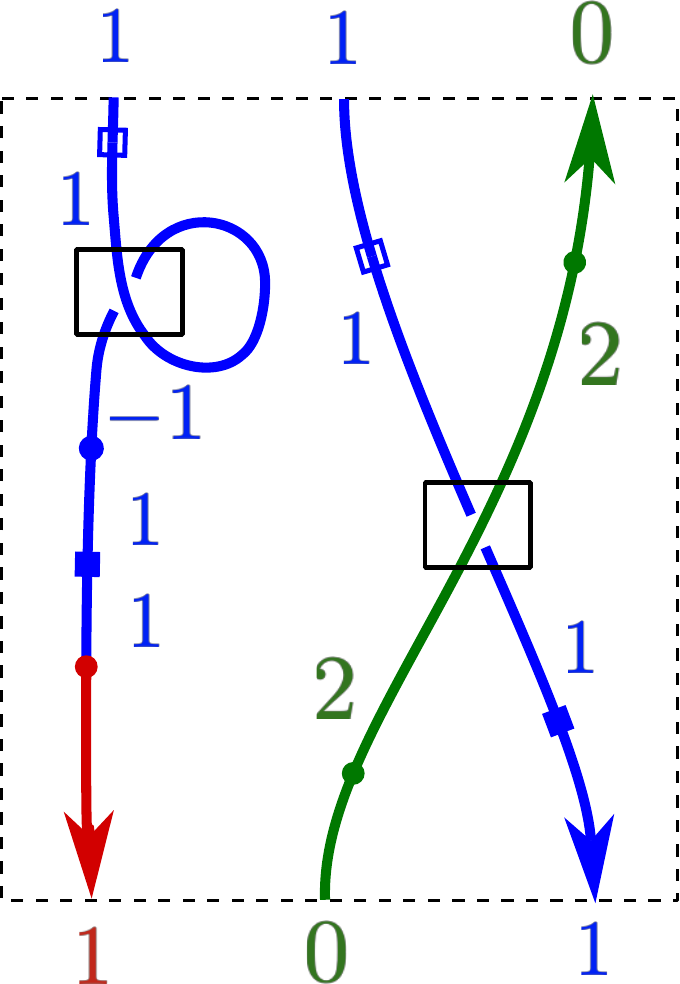}
   \end{minipage}
   \caption{On the left, a stratified ribbon graph on the disk; on the right, its evaluation in $\RepqP$.}
   \label{fig:planar}
\end{figure}

\begin{lemma}\label{evaluation_morphism_lemma}
The evaluation morphism $\text{ev}_{a,b,V,W}$ is well-defined, i.e., $f_\Omega$ does not depend on the equivalence class of $\Omega$ in $\mathcal{L}_{a,b,V,W}$.
\end{lemma}

\begin{proof}
We have to check that $f_\Omega$ is invariant under skein and isotopy relations applied to $\Omega$. For skein relations, this is just the fact that the restriction functors $\iota_t^*$ and $\pi_t^*$ (cf. diagram \eqref{allthefunctors}) are well-defined. For the topological relations, first note that stratified isotopy does not fix the defect wall, so the position where a strand intersects the defect may change. In particular, we may move a crossing between two strands intersecting the defect transversally  from one side to the other by applying such an isotopy. For instance, the stratified ribbon graphs represented (as seen from the front) by \begin{equation*}
\MyFigure{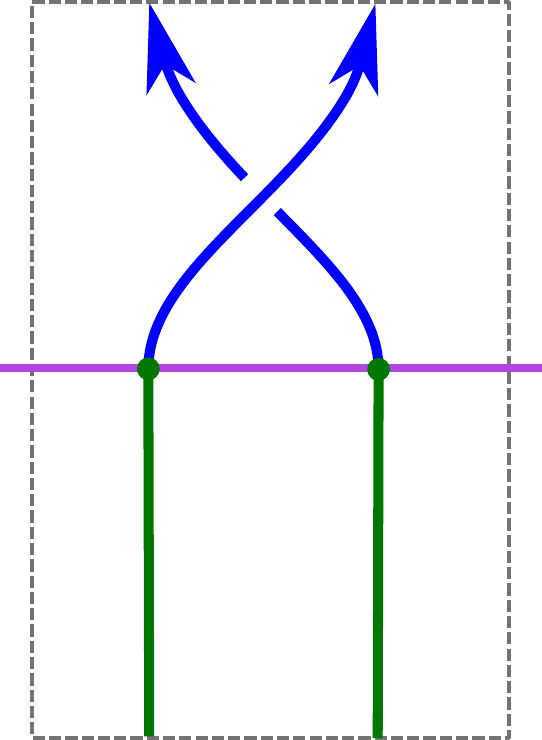}\qquad\text{and}\qquad\MyFigure{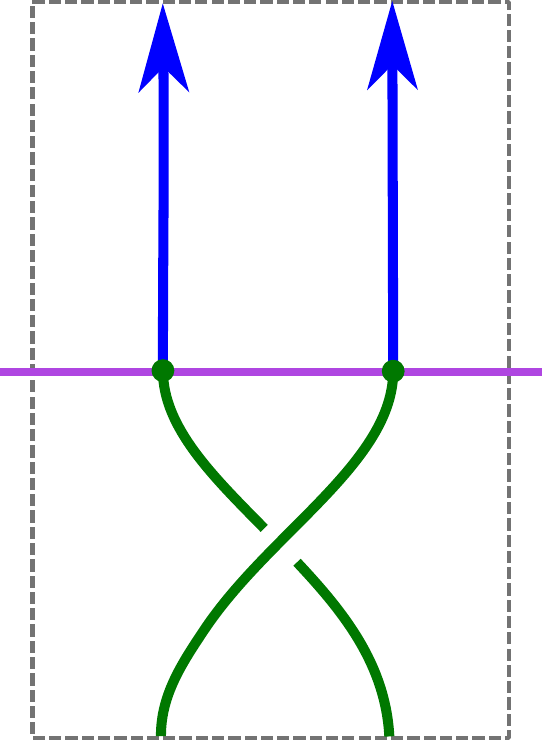}
\end{equation*}
are isotopic, however their evaluations are $$\MyFigure{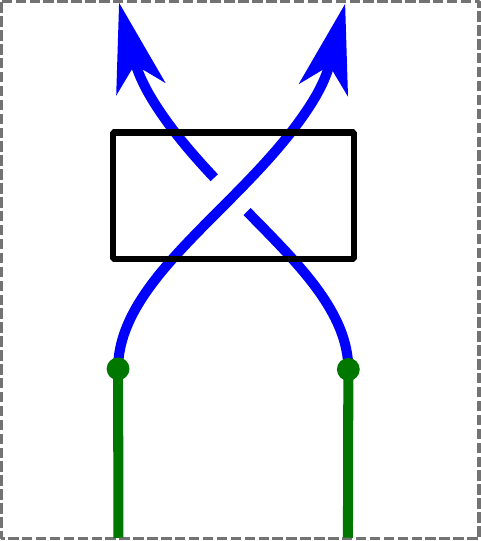}\quad\text{and}\quad\MyFigure{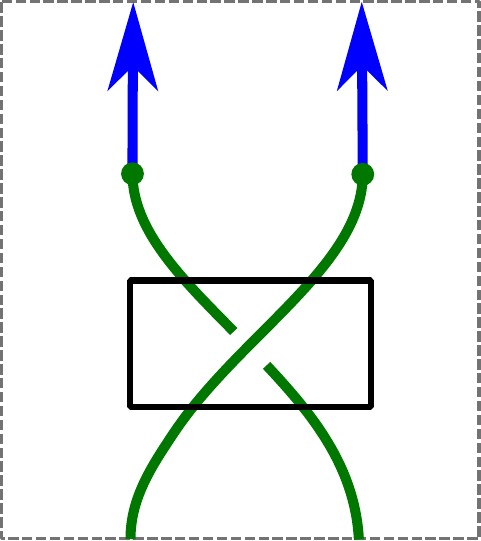},$$ which are not related by isotopy in $\RepqP.$ The naturality of the blue crossing (cf. Remark \ref{naturality}) guarantees that they have the same evaluation. Therefore, since crossings can pass through the defect, we just have to prove that the evaluation is invariant under framed Reidemeister moves happening either on one side of the defect or in the defect itself. For the Reidemeister moves II and III, it suffices to note that crossings are sent to half-braidings in $Z(\RepqP)$, hence they are invertible and satisfy the corresponding Yang-Baxter equation. For the  framed Reidemeister move I $$\MyFigure{drawings/twist1.pdf}\quad\leftrightarrow\quad\MyFigure{drawings/twist2.pdf},$$ this is a consequence of the restriction \ref{restrictionfunctor} being well-defined and the fact that all relations holding in $\RepqP$ come from relations in $\RepqH.$
\end{proof}

\subsection{Defect lines}\label{local_model_strat} We finally introduce defect lines decorated by $\RepqH$ and we use the $(\RepqH,\RepqP)$-centred structure to define local relations around them. Let $E=I^3$ be the cube, which we endow with the stratification defined by the following figure:
\begin{figure}[H]
\centering
\includegraphics[scale=0.3]{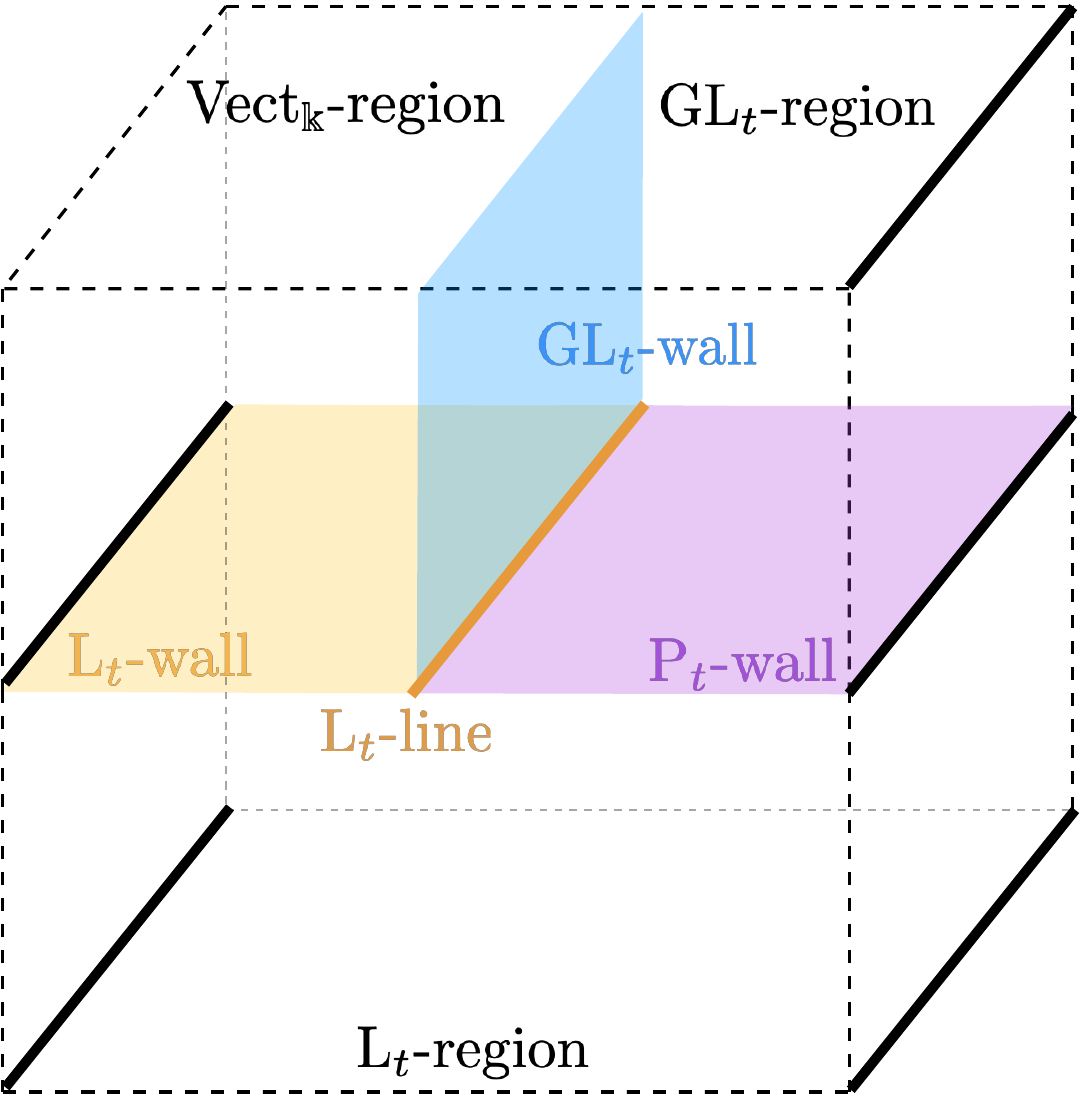}
\caption{We consider two horizontal defect walls decorated with $\RepqH$ and $\RepqP$, and a vertical wall decorated with $\RepqG$. These walls meet at a $1$-dimensional stratum decorated with $\RepqH$. The region below the horizontal wall is the $\H_t$-region, the one over the $\L_t$-wall is the $\Vect_\Bbbk$-region and, finally, the region over the $\P_t$-wall is the $\GL_t$-region.}
\label{fig:doublestrat1}
\end{figure}

We consider the following class of ribbon graphs $\Omega$, which generalises those introduced in the previous section:
\begin{itemize}
\item the endpoints of $\Omega$ lie on the bold intervals and we suppose that they are all at different depths, so that projecting them into the horizontal wall produces no intersections;
\item each part of the graph is decorated according to its region. In particular, the $\Vect_\Bbbk$-region is empty;
\item every edge lies entirely in one of the strata and can be attached to a coupon lying in a stratum whose dimension differs at most by one (that is, they cannot go through the defect line directly from one of the three-dimensional regions).
\end{itemize}
As usual, we consider these graphs up to framed stratified isotopies of the cube. 

Fix a configuration of coloured oriented points $\mathcal{P}$ on the bold intervals and let $E(\mathcal{P})$ be the $(\Bbbk\otimes_{\mathbb{C}(q)}\Bbbk)$-linear space spanned by isotopy classes of ribbon graphs whose endpoints match $\mathcal{P}$. Projecting $\mathcal{P}$ into the horizontal walls, we get two configurations $X$ and $Y$ representing objects of $\RepqH$ and $\RepqP$, respectively. Let $\Omega\in E(\mathcal{P})$ be a ribbon graph. Following the definition of the evaluation morphism in \eqref{evaluation_morphism}, we can project $\Omega$ into the horizontal defects, replace the crossings that appear by the corresponding half-braiding and introduce the pivotal structures, so that we obtain a planar graph $\widetilde{\Omega}$ on a stratified square as in Figure \ref{fig:strat_square}, representing a morphism $f_\Omega\in\Hom_{\RepqH}(Y,j^*(X)).$

\begin{lemma}
The  linear  map defined by the assignment $$\begin{array}{ccc}
E(\mathcal{P}) & \to & \Hom_{\RepqH}(Y,j^*(X)),\\
\Omega & \mapsto & f_\Omega,
\end{array} $$ is well-defined, i.e., it does not depend on the isotopy class of $\Omega$.
\end{lemma}

\begin{proof}
Extending the vertical wall up to the bottom of the cube cuts the ribbon graph $\Omega$ into two subgraphs: $\Omega_1$, lying under the $\H_t$-wall, and $\Omega_2$, lying under and over the $\P_t$-wall. The planar diagram $\widetilde{\Omega}$ can then be obtained by projecting $\Omega_1$ and $\Omega_2$ separately and gluing them along the $\L_t$-defect line. The projections $\widetilde{\Omega}_1$ and $\widetilde{\Omega}_2$ are obtained as in the definition of the evaluation morphism in \eqref{evaluation_morphism}, using the trivial $(\Vect,\RepqH)$-central structure of $\RepqH$ for $\widetilde{\Omega}_1$ and the $(\RepqG,\RepqH)$-central structure of $\RepqH$ for $\widetilde{\Omega}_2.$ Hence, by the same arguments as in Lemma \ref{evaluation_morphism_lemma}, they do not depend on the isotopy class of $\Omega_1$ and $\Omega_2$. This shows that $f_\Omega$ is invariant under isotopies of the stratified cube occurring entirely on one of the sides of the vertical wall. It remains to check that $f_\Omega$ is also invariant under isotopies making a part of the diagram move through the vertical wall. This is the diagrammatic counterpart of the fact that the bimodule $\RepqH$ is centred: the $\L_t$-region can act on the defect line both through the $\P_t$-wall and through the $\L_t$-wall, and we want to show that both actions in fact coincide. This follows from relation \eqref{morphrels} and the definition of the central structures. Indeed, when we project $\Omega$ into the defect, we replace all the crossings that appear by the corresponding half-braidings. In particular, when one of the strands involved comes from the $\H_t$-region (this is the only case of interest), the half-braiding is the braiding $\beta$ of $\RepqH$ (cf. Section \ref{central_algebra}). Therefore, projecting into the $\P_t$-wall makes a coupon decorated with $\pi^*(\beta)$ appear, while applying an isotopy and projecting into the $\H_t$-wall makes a coupon decorated with $\beta$ itself appear. Relation \eqref{morphrels} and the fact that $j^*\circ\pi^*=\text{id}_{\RepqH}$ imply that both diagrams represent the same morphism in $\Hom_{\RepqH}(Y,j^*(X)).$
\end{proof}

\subsection{Three-dimensional defect theory} We finally introduce the three-dimensional theory with defects. Let $(S,f,\phi)$ be a bipartite framed surface (cf. Definition \ref{bipartite}). We label $\phi^{-1}(A)$ and $\phi^{-1}(B)$ with $\L_t$ and $\phi^{-1}(C)$ with $\P_t$. The stratification $\phi$ induces a stratification of the cylinder $S\times I$ decorated as follows (see Figure \ref{fig:strat_decoration}):
\begin{itemize}
\item the horizontal wall $S\times\left\{\frac{1}{2}\right\}$ is decorated as $(S,\phi)$;
\item the cylinder $\phi^{-1}(C)\times\left(\frac{1}{2},1\right]$ and the vertical walls $\phi^{-1}(B)\times\left(\frac{1}{2},1\right]$ are labelled with $\GL_t$;
\item the cylinder $\phi^{-1}(A)\times\left(\frac{1}{2},1\right]$ is labelled with  $\Vect_\Bbbk$;
\item the region $S\times\left[0,\frac{1}{2}\right)$ under the horizontal wall is entirely decorated with $\H_t$.
\end{itemize}

\begin{figure}[t]
\centering
\includegraphics[scale=0.45]{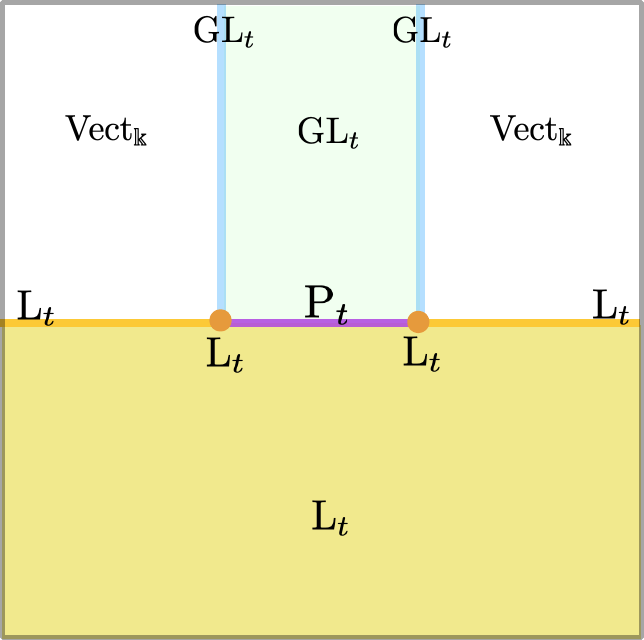}
\caption{Schematic view of the section of a decorated cylinder. The horizontal line represents the defect wall. The orange points are sections of the defect lines. The blue vertical lines represent the $\GL_t$-wall.}
\label{fig:strat_decoration}
\end{figure}
Let $\mathcal{B}$ be a compatible marking of the surface. We consider isotopy classes of ribbon graphs $\Omega$ on $S$ (see Figure \ref{fig:stratangle})  such that:
\begin{itemize}
\item $\Omega$ does not intersect the $\Vect_\Bbbk$-region and it is coloured according to the labelling on the other regions;
\item the endpoints of $\Omega$ are attached to either $S\times\{\pm 1\}$ or one of the components marked on the boundary of the horizontal wall;
\item every edge lies entirely in one of the strata and can be attached to a coupon rectilinearly embedded in a stratum whose dimension differs at most by one (this means, again, that they cannot go directly from the $\GL_t$-region to the $\H_t$-line).
\end{itemize}
Call $S^{\P_t}$ the region of $S$ decorated with $\P_t$ and fix boundary conditions $X\in\text{SkCat}_{\GL_t}(S^{\P_t})$, $Y\in\text{SkCat}_{\H_t}(S)$ and $U_C\in\mathcal{Z}(C),$ for every $C\in\mathcal{B}$.  We set $E^S\left(X,Y,(U_C)_{C\in\mathcal{B}}\right)$ for the linear space spanned by ribbon graphs as above, modulo local relations induced by $\RepqG$ and $\RepqH$ on the  three-dimensional regions, and by the local models described in Sections \ref{local_model} and \ref{local_model_strat} near the defects. The assignment $$\begin{array}{cccc}
\mathcal{A}(S,f,\phi,\mathcal{B})\colon & \left(\bigboxtimes\limits_{C\in\mathcal{B}}\mathcal{Z}(C)\right)\boxtimes\text{SkCat}_{\GL_t}(S^{\P_t})\boxtimes\Skh^\op & \to & \Vect,\\
& \left((U_C)_{C\in\mathcal{B}},X,Y\right) & \mapsto & E^S\left(X,Y,(U_C)_{C\in\mathcal{B}}\right),
\end{array}$$ is functorial. The action on morphisms is given by horizontally gluing cylinders along the components marked on the defect wall, and by vertically stacking diagrams representing morphisms in the skein categories.

\begin{figure}[t]
\centering
\includegraphics[scale=0.25]{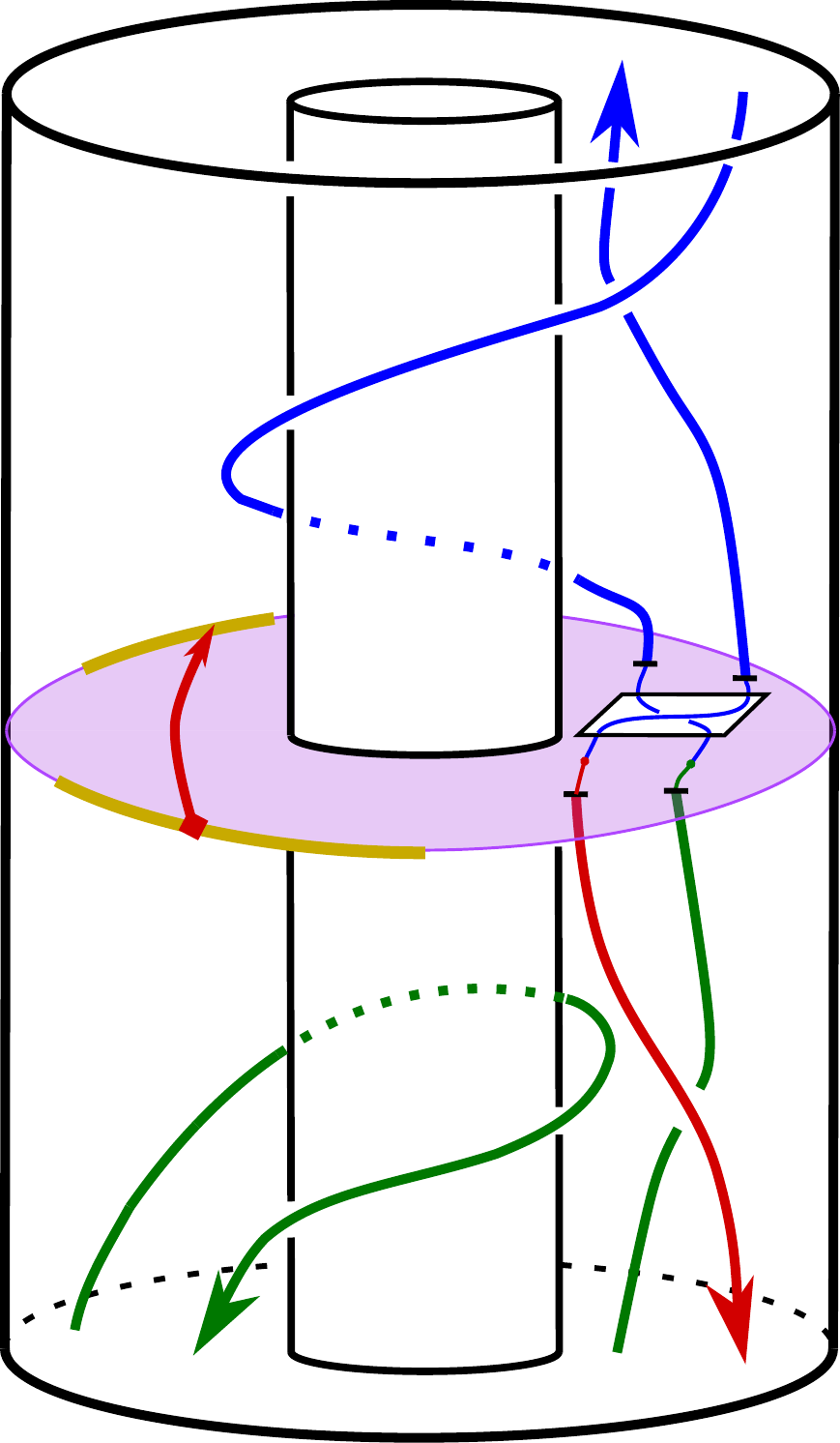}
\caption{Stratified ribbon graph on an annulus entirely labelled by $\P_t$.}
\label{fig:stratangle}
\end{figure}

\begin{lemma}
Let $(S,f,\phi,\mathcal{B})$ be a bipartite framed surface with compatible marking.  Let $\gl_C(S)$ be the surface obtained by gluing $S$ along a boundary component $C$, with $C,\overline{C}\in\mathcal{B}_{\P_t}$ (alternatively, in $\mathcal{B}_{\L_t}$). Then, $$\mathcal{A}(\gl_C(S))(-)=\int^{U\in\mathcal{Z}(C)}\mathcal{A}(S)(-,U,U),$$ where the components decorated with $U$ on the RHS are $C$ and $\overline{C}$. 
\end{lemma}

\begin{proof}
Again, this is essentially the same proof as \cite[Theorem 4.4.2]{walker}. Gluing along $C$ induces morphisms $$\gl_U\colon\mathcal{A}(S)(-,U,U)\to\mathcal{A}(\gl_C(S))(-)$$ for every $U\in\mathcal{Z}(C)$. Moreover, if $f\colon U\to V$ is a morphism in $\mathcal{Z}(C)$, $\Omega$ is a stratified ribbon graph representing an element of $\mathcal{A}(S)(-,U,V)$ and $f\star\Omega$ (resp. $\Omega\star f$) are the graphs obtained by gluing $f$ to $\Omega$ along $C$ (resp. $\overline{C}$), then the graphs $\gl_U(f\star\Omega)$ and $\gl_V(\Omega\star f)$ are related by an isotopy supported on a tubular neighbourhood of $\partial S$. Therefore, they define the same element of $\mathcal{A}(\gl_C(S))(-)$ and the diagram
\begin{equation}\label{coend_diagram}
\begin{tikzcd}[column sep = small]
                                                                  & {\mathcal{A}(S)(-,U,U)} \arrow[rd, "\gl_U"]  &                                            \\
{\mathcal{A}(S)(-,U,V)} \arrow[ru, "-\star f"] \arrow[rd, "f\star -"'] &                                          & {\mathcal{A}(\gl_C(S))(-)} \\
                                                                  & {\mathcal{A}(S)(-,V,V)} \arrow[ru, "\gl_V"'] &                                           
\end{tikzcd}
\end{equation}
commutes. We shall show that $\mathcal{A}(\gl_C(S))(-)$ is universal for this property. Let $E$ be a vector space together with linear maps $$\varphi_U\colon\mathcal{A}(S)(-,U,U)\to E$$ making the equivalent diagrams commute. We have to prove that there is a unique linear map $$\varphi\colon\mathcal{A}(\gl_C(S))(-)\to E$$ such that $\varphi_U=\varphi\circ\gl_U$ for every $U\in\mathcal{Z}(C).$ Let $\Omega$ be a stratified ribbon graph in $\gl_C(S)\times I$ representing an element of $\mathcal{Z}(\gl_C(S))(-)$. Applying isotopy and stratified skein relations, we may locally project $\Omega$ onto the defect, so that we may suppose that there is a neighborhood of $C\times I$ where $\Omega$ is represented by a planar graph contained in the defect (see Figure \ref{fig:cut}). Let $U$ be the intersection of $\Omega$ with $C\times\left\{\frac{1}{2}\right\}$, which determines an object of $\mathcal{Z}(C)$. Cutting along $C\times I$, we get a ribbon graph $\Omega_U$ in $S\times I$ defining an element of $\mathcal{A}(S)(-,U,U)$. If $\varphi$ exists, then $\varphi(\Omega)=\varphi_U(\Omega_U)$ so it is uniquely defined.

\begin{figure}[t]
\centering
\includegraphics[scale=0.2, valign = c]{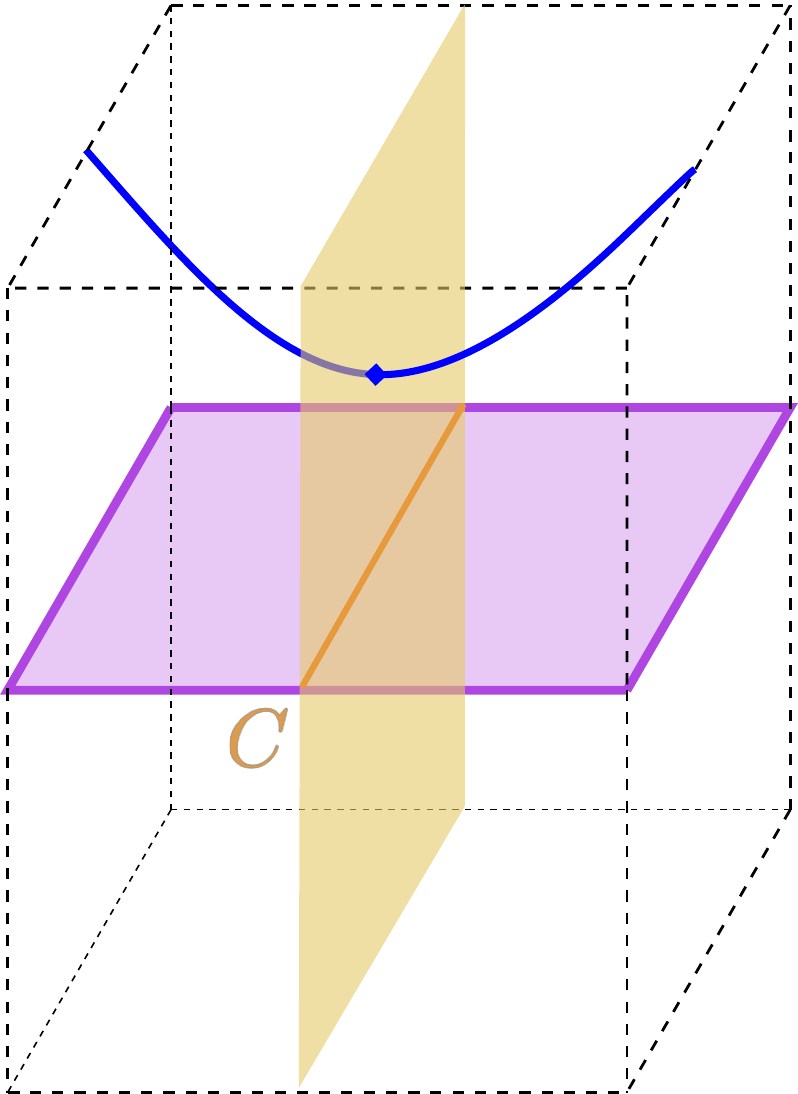}\qquad =\qquad \includegraphics[scale=0.2, valign = c]{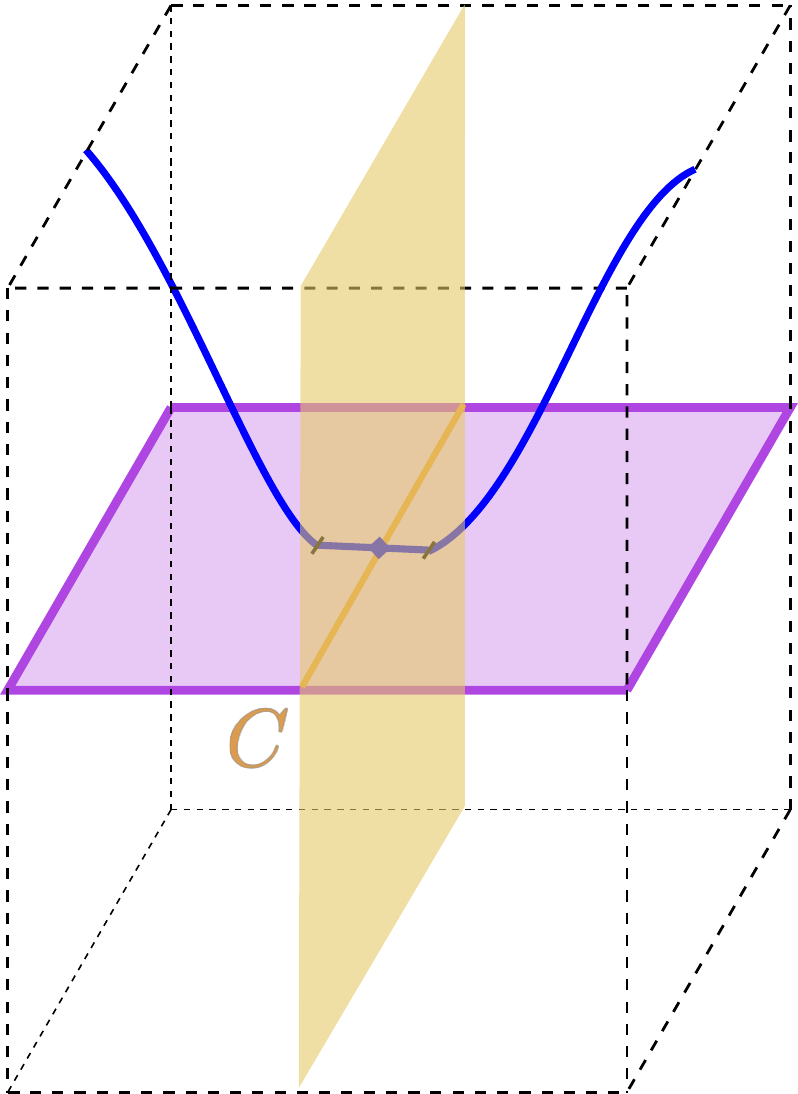}
\caption{Applying isotopy and skein relations, we may always cut a graph along a planar component.}
\label{fig:cut}
\end{figure}

To prove the existence, suppose first that $\widetilde{\Omega}$ is another stratified ribbon graph representing the same element of $\mathcal{A}(\gl_C(S))(-)$ and related to $\Omega$ by an isotopy shifting a collar neighbourhood through $C\times I$. Then, cutting along $C\times I$ yields a ribbon graph $\widetilde{\Omega}_V$ in $S\times I$ and, since we may suppose that it is planar around $C\times I$, the commutativity of \eqref{coend_diagram} implies that $\gl_U(\Omega_U)=\gl_V(\widetilde{\Omega}_V)$. On the other hand, if $\Omega$ and $\widetilde{\Omega}$ are related by an isotopy of $\gl_C(S)$ supported in a cube not intersecting $C\times I$, then $\Omega_U$ and $\widetilde{\Omega}_U$ are related by the same isotopy of $S$, so that they represent the same element of $\mathcal{A}(S)(-,U,U)$ and $\varphi_U(\Omega_U)=\varphi_U(\widetilde{\Omega}_U)$. Therefore, $\varphi$ is compatible with isotopies. Finally, if $\Omega$ and $\widetilde{\Omega}$ are related by a skein relation on a cube not intersecting $C\times I$, then $\Omega_U$ and $\widetilde{\Omega}_U$ are related by the same relation and they represent the same element. If the skein relation occurs in a cube intersecting $S$, we can apply an isotopy moving it off $C\times I$. We get stratified ribbon graphs $\Omega'$ and $\widetilde{\Omega}'$ related by a skein relation on a cube not intersecting $C\times I$, so $\gl_{U'}(\Omega_{U'}')=\gl_{U'}(\widetilde{\Omega}_{U'}')$. Moreover, the isotopy invariance of $\varphi$ implies that $\gl_U(\Omega_U)=\gl_{U'}(\Omega_{U'}')$ and $\gl_U(\widetilde{\Omega})=\gl_{U'}(\widetilde{\Omega}_{U'}).$ This proves that $\varphi$ is compatible with skein relations. 
\end{proof}

\begin{theorem}\label{3TFT}
The assignment $$\mathcal{A}\colon\text{Bord}^{\text{bip,fr}}_1(2)\to\textsc{Bimod}$$ defines a symmetric monoidal functor from the category of $2$-dimensional bipartite framed cobordisms, hence a $2$-dimensional framed TFT. It coincides with the planar theory $\mathcal{Z}$ on one-dimensional manifolds.
\end{theorem}

\begin{proof}
As in Theorem \ref{planar_theory}, the compatibility with gluing/cutting surfaces follows from the previous lemma and the fact that we are considering compatible markings.
\end{proof}

\section{The HOMFLY skein bialgebra} \label{homfly_bialgebra}

We finally recover the Turaev coproduct on the HOMFLY skein algebra as a particular case of the formalism of skein theory with defects developed in the previous sections.

\subsection{The coproduct map}\label{coproductsection} Consider a bipartite framed surface $(S,f,\phi,\mathcal{B})$ with compatible marking. As above, denote by $S^{\P_t}$ the region of $S$ decorated with $\P_t$. Applying the three-dimensional TFT $\mathcal{A}$ defined above, we get a functor $$\mathcal{A}(S)\colon\left(\bigboxtimes\limits_{C\in\mathcal{B}}\mathcal{Z}(C)\right)\boxtimes\text{SkCat}_{\GL_t}(S^\P_t)\boxtimes\Skh\to\Vect.$$ Choose boundary conditions $U_C\in\mathcal{Z}(S)$ for every component $C\in\mathcal{B}$ marked on the boundary, and fix the empty configurations $\emptyset\in\text{SkCat}_{\GL_t}(S^{\P_t})$ and $\emptyset\in\Skh$ on the top and bottom bases of $S\times I$. By construction, the vector space $E^S((U_C)_{C\in\mathcal{B}})$ obtained from $\mathcal{A}(S)$ comes with a left and a right action of the skein algebras $\text{Sk}_{\GL_t}(S^{\P_t})$ and $\skh$, respectively. By choosing an appropriate stratification $\phi$, this vector space is indeed isomorphic to $\skh$:

\begin{figure}[t]
\centering
\includegraphics[scale=0.4, valign=c]{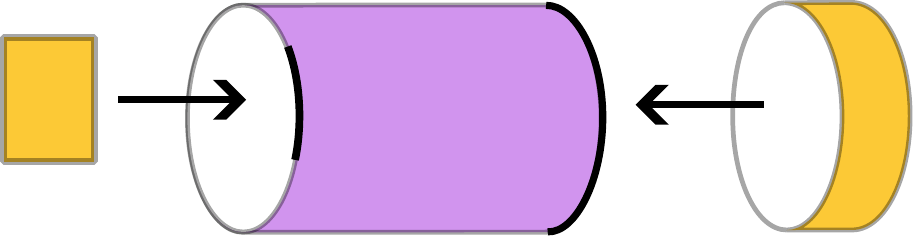} $\qquad\rightsquigarrow\qquad$ \includegraphics[scale=0.4, valign=c]{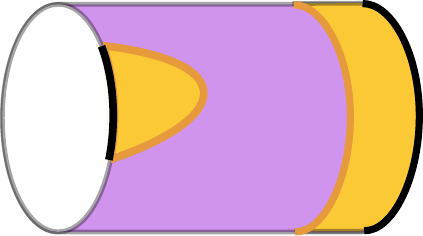}
\caption{Gluing stratified cylinders along the boundary modifies the stratification.}\label{fig:gluing_strat}
\end{figure}

\begin{proposition}\label{skeinalgebra}
Suppose that $\mathcal{B}\neq\emptyset$ and that $\phi$ consists of a family of cylinders inserted along the marked components as in Figure \ref{fig:gluing_strat}. Then, there is a natural family of isomorphisms of vector spaces $$\mathcal{A}(S)(-,\emptyset,\emptyset)\cong\mathcal{Z}^{\H_t}(-).$$ 
\end{proposition}

\begin{proof}
Fix boundary conditions $U$ and let $\Omega$ be a stratified ribbon graph representing an element of $\mathcal{A}(S)(U,\emptyset,\emptyset)$. Since there is no strand attached to $S\times\{\pm 1\}$, the $3$-dimensional stratified skein relations allow to project $\Omega$ entirely onto the defect, so that it can be represented by a planar graph on $S\times\left\lbrace\frac{1}{2}\right\rbrace$, coloured by $\RepqP$ in the $\P_t$-region and by $\RepqH$ in the $\H_t$-region. We first prove that this graph can be written as a linear combination of diagrams with no strand coloured in $\bsquare.$ Indeed, suppose that $\alpha$ is a component coloured in blue. If $\alpha$ is an open strand, then it is attached to one of the intervals/circles marked on the boundary, hence it traverses the $1$-dimensional $\H_t$-defect. Crossing these defect lines switches its colour to $\orsquare$ near its endpoints. We can then apply relation \eqref{RepqHtrel1} inside the $\H_t$-disks/cylinders and make one of the inclusion/projections appearing on each term cross the defect (relation \eqref{morphrels}) and slide along $\alpha$ (by naturality, cf. Remark \ref{naturality}) to reach the opposite endpoint (see figure below). The part of $\alpha$ lying in the $\P_t$ region is now entirely coloured with $\gsquare$ and $\redsquare$:
\begin{equation*}
\begin{split}
\includegraphics[scale=0.4, valign=c]{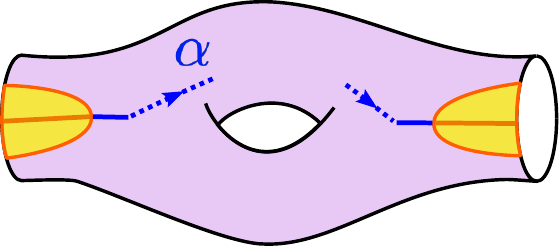}\quad&\stackrel{(\ref{RepqHtrel1})}{=}\quad\includegraphics[scale=0.4, valign=c]{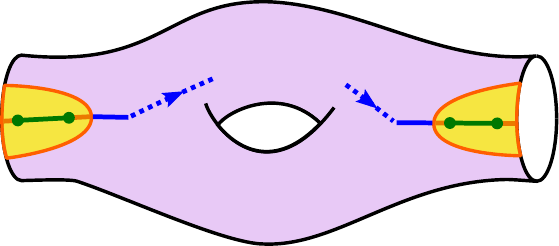}\quad+\quad\includegraphics[scale=0.4, valign=c]{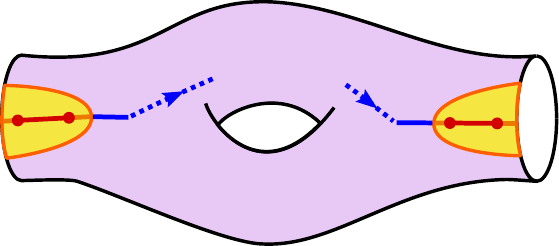}\\
&\stackrel{(\ref{morphrels})}{=}\quad\includegraphics[scale=0.4, valign=c]{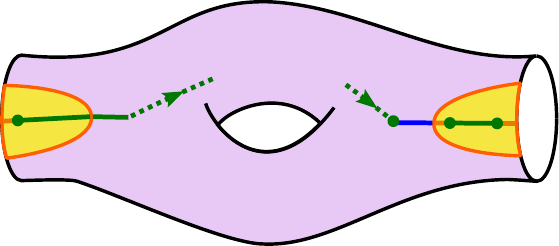}\quad+\quad\includegraphics[scale=0.4, valign=c]{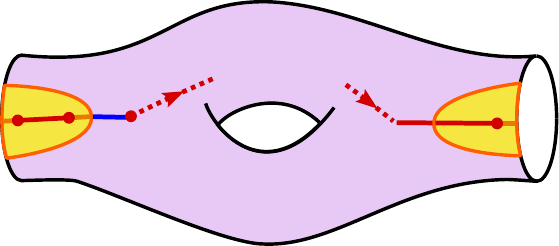}\\
&\stackrel{(\ref{morphrels})}{=}\quad\includegraphics[scale=0.4, valign=c]{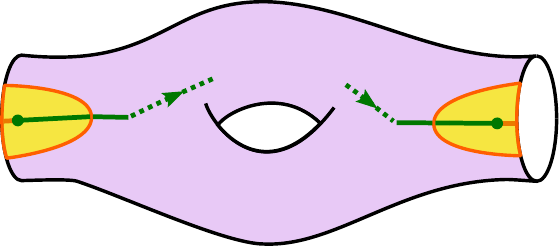}\quad+\quad\includegraphics[scale=0.4, valign=c]{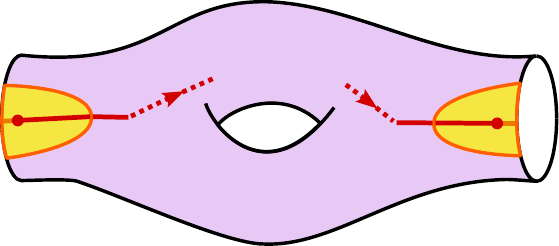}\;.
\end{split}
\end{equation*}
Similarly, if $\alpha$ is a closed strand, we can push it close to a marked component by applying and isotopy. Then, relation \eqref{morphrels} allows us to project it partially into the $\H_t$-region where we can apply the same trick to write it as a combination of a green and a red component. 

Let us call $\widetilde{\Omega}$ a diagram obtained by $\Omega$ by applying the previous transformations. Since it is entirely coloured by $\orsquare$, $\gsquare$ and $\redsquare$, it represents an element of $\mathcal{Z}^{\H_t}(S)(U)$. This element is well-defined, since all the relations applied to switch $\bsquare$ to $\orsquare$ are local relations from $\RepqH$, hence they hold in $\mathcal{Z}^{\H_t}(S)(U)$. We get a linear map $$\mathcal{A}(S)(U,\emptyset,\emptyset)\to\mathcal{Z}^{\H_t}(S)(U),\quad\Omega\mapsto\widetilde{\Omega}.$$ An inverse of this map is constructed as follows. Given a planar diagram on  $\Gamma$ on $S$ coloured by $\RepqH$, we can apply isotopies and naturality to make all the ``forbidden'' morphisms slide along open strands and lie arbitrarily close to the marked components of the boundary, so that they are all inside one of the embedded disks/cylinders defining an $\H_t$-region. Note that this can be done in a unique way so that the intersection of the strand with the $\P_t$-region is coloured in either $\gsquare$ or $\redsquare$. We get this way a well-defined graph $\widehat{\Gamma}$ on the stratified defect wall and it is straightforward to see that $\Gamma\mapsto\widehat{\Gamma}$ defines an inverse for of the linear map below. Naturality follows from the fact that these maps do not modify diagrams in a tubular neighborhood of the marked boundary components.
\end{proof}

Therefore, the presence of regions coloured by $\RepqH$ near the boundary of the defect wall allows one to apply relation \eqref{RepqHtrel1} to components coloured with $\bsquare$, yielding an identification $\bsquare=\orsquare.$ In the case where there are open strands attached to $S\times\{\pm 1\}$, this identification still holds, but only on closed components:

\begin{proposition}\label{quotient}
Let $\mathcal{B}$ and $\phi$ be as in the previous proposition and fix empty boundary conditions on each of the components marked on the defect. Then, $$\mathcal{A}(S)(\emptyset,-,-)\cong\bigslant{\pres(S)(\emptyset,-,-)}{\langle \bsquare=\orsquare\text{ on closed components on the defect wall}\rangle}.$$
\end{proposition}

\begin{proof}
The idea of the proof is the same as in the previous proposition: we can construct linear maps between both spaces by pushing a part of the diagram near one the $\H_t$-regions on the boundary, then apply relation \eqref{morphrels} to project into the $\H_t$-region and make morphisms slide along closed components. This allows us to decompose closed components coloured with $\bsquare$ into linear combinations of green and red diagrams, yielding the identification $\bsquare = \orsquare$. On open strands, this induces no new relations: since they are all attached to $S\times\{\pm 1\}$, ``forbidden'' morphisms cannot be simplified, so making them appear just yields a different way of writing the same morphism, but they induce no relation.
\end{proof}

Combining both propositions, we see that $\mathcal{A}(S)(\emptyset)$ is isomorphic as a linear space to $\skh,$ and by functoriality of $\mathcal{A}(S)$ we have now actions of both $\skg$ and $\skh$:

\begin{theorem}\label{skeinalgebras}
Let $(S,f,\mathcal{B})$ be a marked framed surface with $\mathcal{B}\neq\emptyset$. Then, the $\skg-\skh$-bimodule $\mathcal{A}(S)(\emptyset)$ is isomorphic to $\skh$ as a linear space. In particular, $\skh$ becomes a left $\skg(S)$-module under this identification. \QEDA
\end{theorem}

We treat now the case of the torus $S=\mathbb{T}^2.$ Let $(\mathbb{T}^2,f,\phi)$ be a framed torus endowed with a stratification $\phi$ whose $H_t$-region consists of only a disk embedded in $\mathbb{T}^2.$ By the same argument as in the proof of Propositions \ref{skeinalgebra} and \ref{quotient}, $$\mathcal{A}(\mathbb{T}^2,f,\phi)(\emptyset)\cong\text{Sk}_{\H_t}(\mathbb{T}^2)$$ Moreover, the $\P_t$-region is homeomorphic to $\mathbb{T}^2\setminus\mathbb{D}^2$, so we have an action $$\text{Sk}_{\GL_t}(\mathbb{T}^2\setminus\mathbb{D}^2)\otimes\text{Sk}_{\H_t}(\mathbb{T}^2)\to\text{Sk}_{\H_t}(\mathbb{T}^2)$$ of the $\GL_t$-skein algebra of this puncture torus. 

\begin{proposition}\label{torus}
The action above descends to an action of the $\GL_t$-skein algebra of the torus.
\end{proposition}

\begin{proof}
A loop $\alpha$ around the puncture acts by multiplication by $$\pi^*_t\circ\res\left(\MyFigure{\bknot}\right)$$ on the $\H_t$-disk hence the action descends to the quotient $$\bigslant{\text{Sk}_{\GL_t}(\mathbb{T}^2\setminus\mathbb{D}^2)}{\left\langle\alpha=\MyFigure{\bknot}\right\rangle}\cong\text{Sk}_{\GL_t}(\mathbb{T}^2).$$
\end{proof}

Let $(S,f)$ be any framed surface and choose $\mathcal{B}\neq\emptyset$ if $\partial S\neq\emptyset.$ Let $\phi$ be a stratification of $S$ as in Propositions \ref{skeinalgebra} and \ref{torus}, so that $\skg$ acts on $\skh$:

\begin{corollary}\label{coprod1}
The morphism $\skg\to\mathcal{A}(S)(\emptyset,\emptyset)$ defined by acting on the empty diagram induces a $\mathbb{C}(q)$-algebra homomorphism $$\widetilde{\Delta}_f\colon\skg\to\skh.$$ \QEDA
\end{corollary}

Composing $\widetilde{\Delta}_f$ with the isomorphism in Lemma \ref{coprod2}, we obtain a $\mathbb{C}(q)$-algebra morphism
\begin{equation}\label{turaev_coprod}
\Delta_f\colon \skg\to\skg\otimes_{\mathbb{C}(q)}\skg
\end{equation}
whose restriction to $\Bbbk\cong\Bbbk\cdot\boldsymbol{1}_\emptyset$ is \eqref{coproduct}. By construction, it can be computed on links by taking a diagram on the surface where every crossing, cup and cap has been replaced by their image by the restriction functor \eqref{restrictionfunctor}, and applying relations \eqref{box1}--\eqref{box4} and \eqref{RepqHtrel1} to split the diagram into two coloured diagrams, which we consider as lying in two different copies of $\skg.$ Here by cup and cap we mean a portion of the diagram where the rotation number increases or decreases by $1$.

\begin{proposition}
The coproduct in \eqref{turaev_coprod} is coassociative.
\end{proposition}

\begin{proof}
The map $$(\Delta_f\otimes\id)\circ\Delta_f\colon\skg\to\skg\otimes_{\mathbb{C}(q)}\skg\otimes_{\mathbb{C}(q)}\skg$$ can be computed as follows. Given a link diagram $\Gamma$, we first compute $\widetilde{\Delta}_f(\Gamma)\in\skh$ as in the previous paragraph. Then, we split the green part into two colours (green again for the first copy of $\skg$ and violet for the second one) using the same rules: replacing all green crossings, cups and caps by their images by the restriction functor \eqref{restrictionfunctor} and applying relations \eqref{box1}--\eqref{box4} and \eqref{RepqHtrel1} with a proper choice of colours. Since applying $\widetilde{\Delta}$ does not produce new crossings, we can do all this at once from the initial diagram using the following relations:
\begin{equation}
\MyFigure{\boxx}\quad\mapsto\quad\sum_{\gsquare,\vsquare,\redsquare}\;\MyFigure{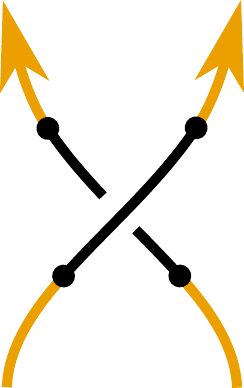}\;+\;\left(q-q^{-1}\right)\left(\MyFigure{\orgipip}+\MyFigure{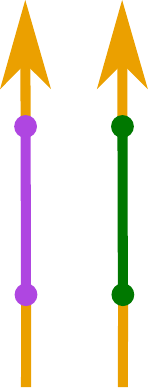}+\MyFigure{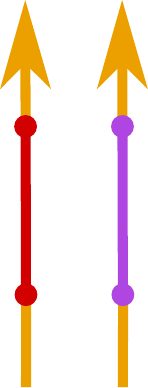}\right),
\end{equation}
\begin{equation}
\MyFigure{\boxz}\quad\mapsto\quad\sum_{\gsquare,\vsquare,\redsquare}\;\MyFigure{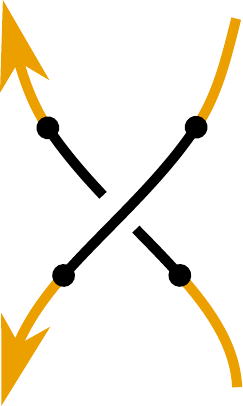}\;+\;\left(q-q^{-1}\right)\left(q^{-t_2}q^{-t_3}\MyFigure{\ogrevvcoevv}+q^{-t_2}\MyFigure{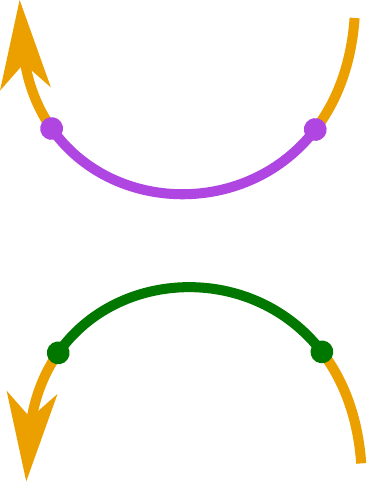}+\MyFigure{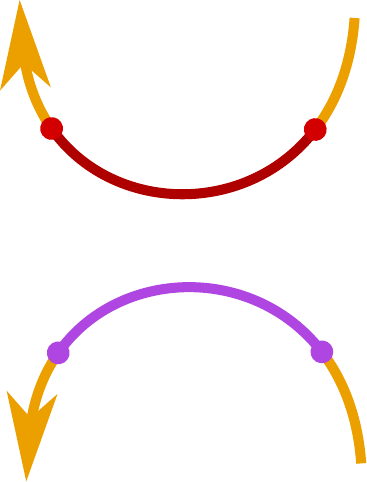}\right),
\end{equation}
\begin{equation}
\MyFigure{drawings/pivotalcoev.pdf}\quad\mapsto\quad\MyFigure{\ogocoev}\;+\;q^{t_1}\MyFigure{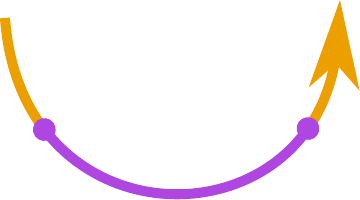}\;+\;q^{t_1}q^{t_2}\MyFigure{\orocoev},
\end{equation}
\begin{equation}
\MyFigure{drawings/pivotalev}\quad\mapsto\quad \MyFigure{\ogoevv}\;+\;q^{-t_1}\MyFigure{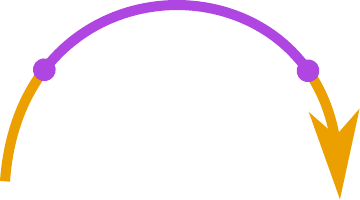}\;+\; q^{-t_1}q^{-t_2}\MyFigure{\oroevv},
\end{equation}
where the sums are taken over all the possible combinations of colours for the involved strands and $$q^{t_1}=q^t\otimes 1\otimes 1,\qquad q^{t_2}=1\otimes q^t\otimes 1, \qquad q^{t_3}=1\otimes 1\otimes q^t,$$ in $\Bbbk\otimes_{\mathbb{C}(q)}\Bbbk\otimes_{\mathbb{C}(q)}\Bbbk.$
The same formulas compute $(\id\otimes\Delta_f)\circ\Delta_f$, hence the coproduct is coassociative.
\end{proof}

\begin{remark}\label{framings}
The coproduct $\Delta_f$ depends on the choice of the framing $f$. For instance, let $S=\mathbb{A}$ be the annulus, which can be equipped with either a radial framing or the framing induced by $\R^2$. Consider first the case where $f$ is the radial framing. We have 
\begin{equation*}
\begin{split}
\includegraphics[scale=0.15, valign = c]{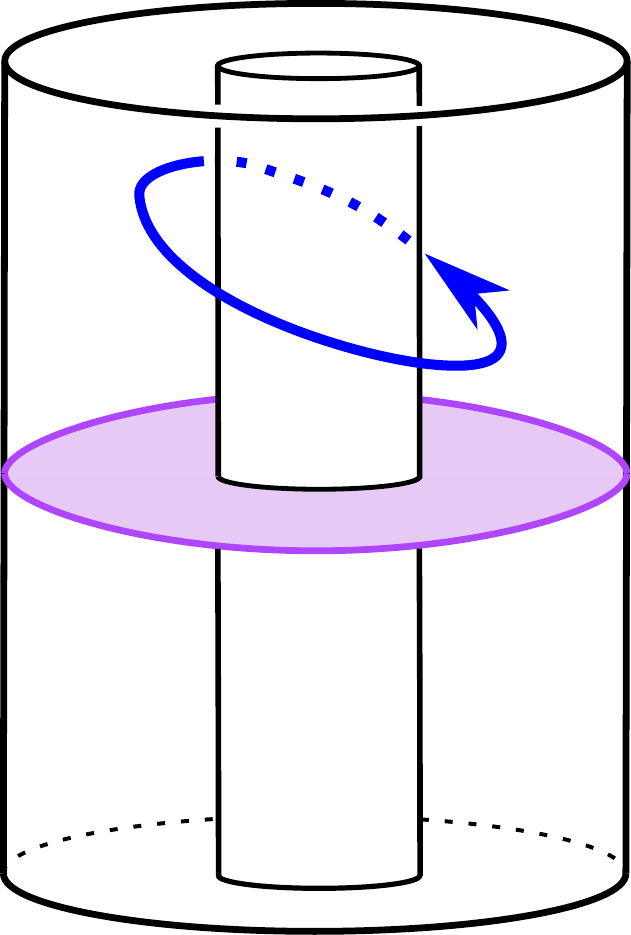}\;=\;\MyFigure{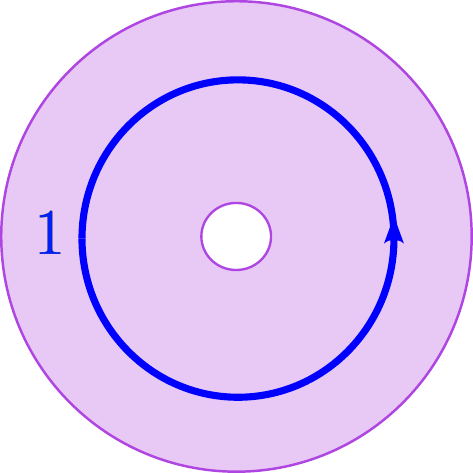}\;&=\;\MyFigure{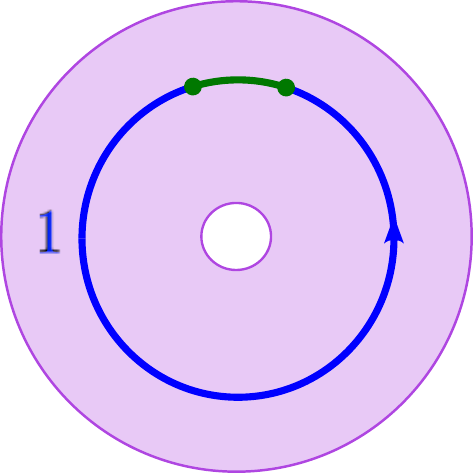}\;+\;\MyFigure{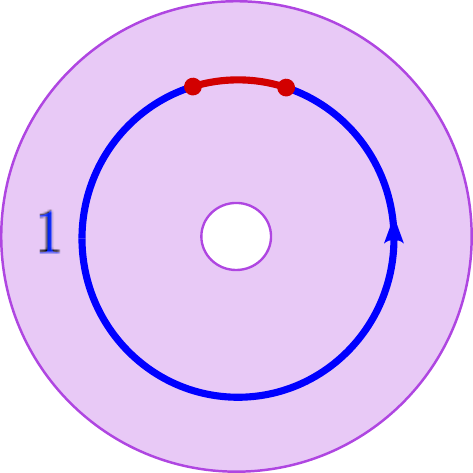}\\
&=\;\MyFigure{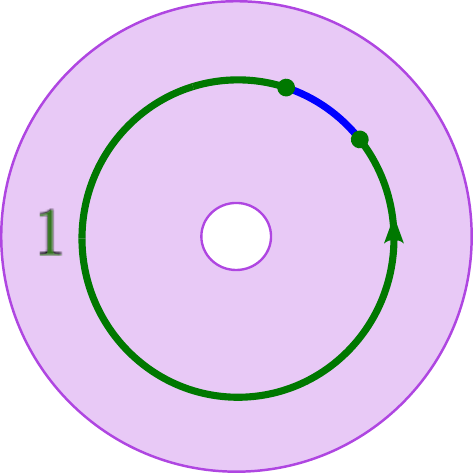}\;+\;\MyFigure{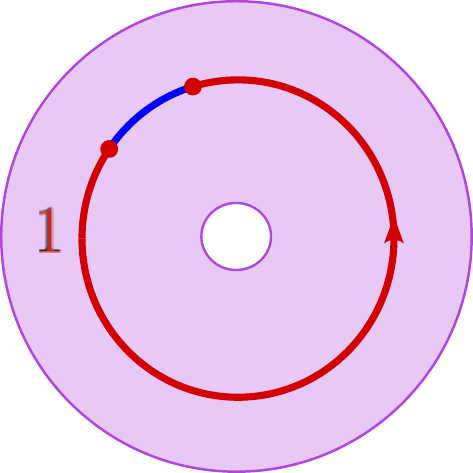}\\
&=\;\MyFigure{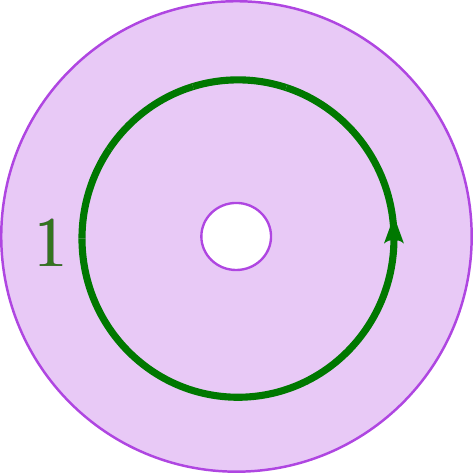}\;+\;\MyFigure{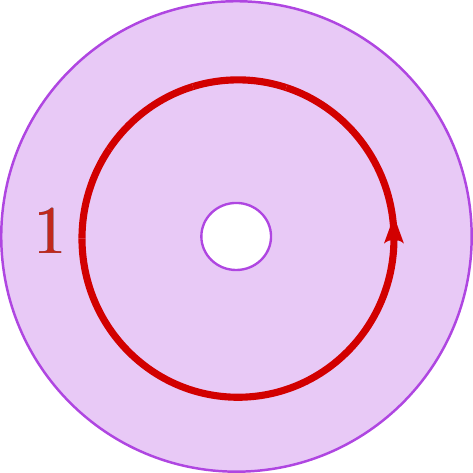}\;=\;\includegraphics[scale=0.15, valign = c]{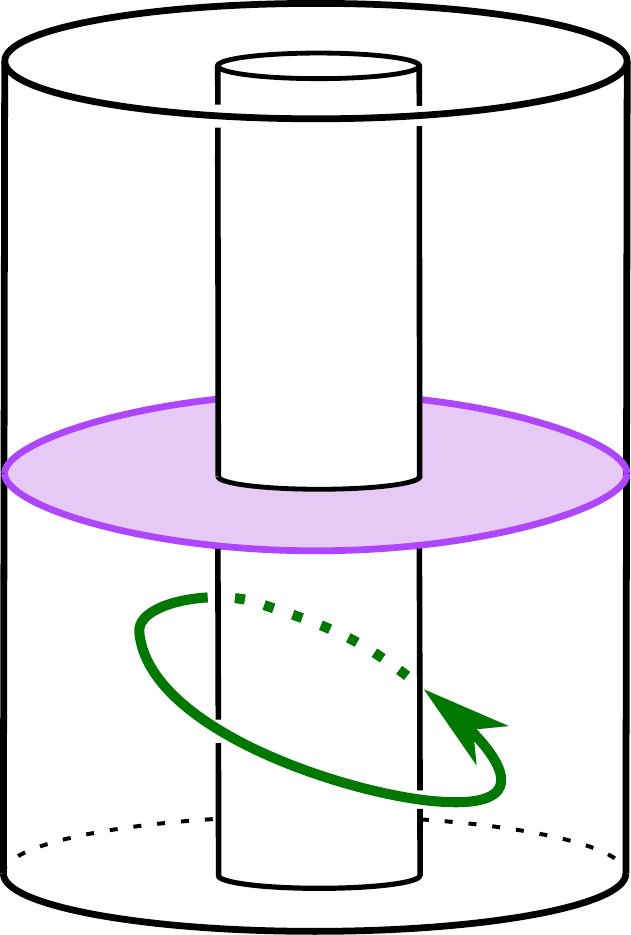}\;+\;\includegraphics[scale=0.15, valign = c]{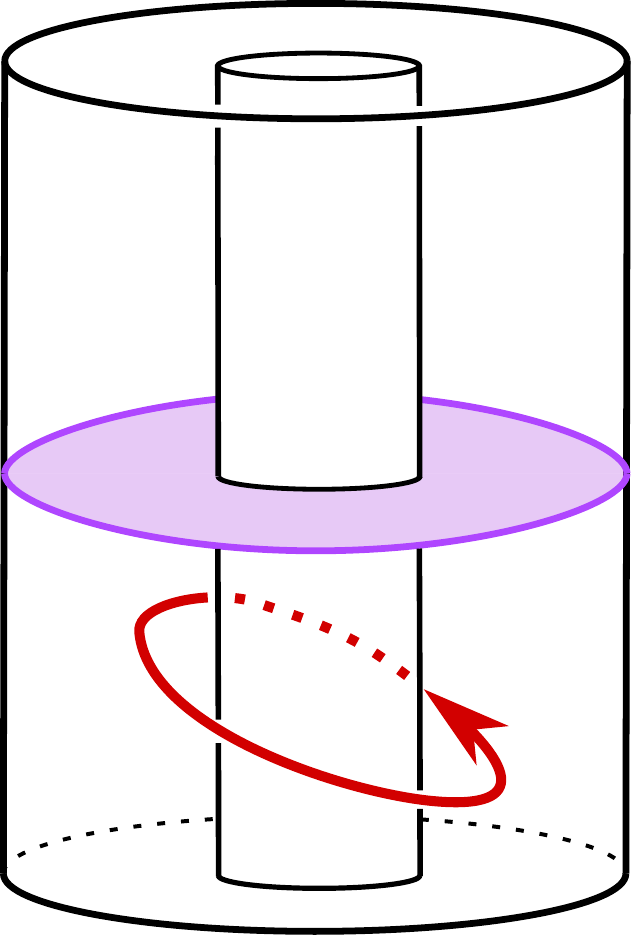}\;.
\end{split}
\end{equation*}
On the other hand, if $f$ is the framing induced by $\R^2$, the local relations projecting the link into the defect make use of the pivotal structure of $\RepqG$:
\begin{equation*}
\begin{split}
\includegraphics[scale=0.15, valign = c]{drawings/exlinka.pdf}\;=\;\MyFigure{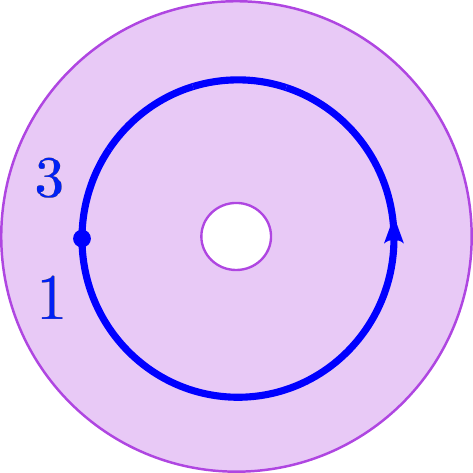}\;=\;q^{t_2}\;\includegraphics[scale=0.15, valign = c]{drawings/exlinke.pdf}\;+\;q^{-t_1}\;\includegraphics[scale=0.15, valign = c]{drawings/exlinkf.pdf}\;,
\end{split}
\end{equation*}
where the last equality follows from the definition of the pivotal structure of $\RepqG$ by doing the same manipulations as in the previous paragraph.
\end{remark}

\subsection{The counit} Let $\overline{\text{Rep}_q}\GL_n$ be the $\mathbb{C}$-linear category obtained by specialising $$q^t\;\rightsquigarrow\; q^n,\qquad\delta\;\rightsquigarrow\;\frac{q^n-q^{-n}}{q-q^{-1}}$$ in Definition \ref{RepqGdef}, with $q\in\mathbb{C}^\times$ and $q^2\neq 1$. This is the category $\dot{\mathcal{OS}}(q-q^{-1},q^n)$ from \cite{brundan}. Let $\varphi_{n}\colon\Bbbk\to\mathbb{C}(q)$ be the evaluation morphism in Remark \ref{universal1}. Then, the evaluation functor $\text{ev}^\text{G}_{q,n}$ factors as $$\RepqG\to\overline{\text{Rep}_q}\GL_n\to\Rep_q\GL_n,$$ where the first arrow is a $\varphi_{n}$-linear functor and the second one is $\mathbb{C}$-linear.

\begin{definition}\textup{(\cite{deligne})} 
Let $\C$ be a ribbon category. A morphism $f:X\to Y$ is \emph{negligible} if $\text{Tr}_\C(f\circ g)=0$ for all $g:Y\to X$ in $\C$.
\end{definition}

\begin{theorem}\textup{(\cite{deligne,brundan})} 
The functor $\overline{\Rep_q}\GL_n\to\Rep_q\GL_n$ above induces a monoidal equivalence $$\Phi\colon\bigslant{\overline{\Rep_q}\GL_n}{\mathcal{N}}\xrightarrow{\approx}\Rep_q\GL_n,$$ where $\mathcal{N}$ is the tensor ideal of $\overline{\Rep_q}\GL_n$ consisting of negligible morphisms.
\end{theorem}

In particular, $\GL_0$ is the trivial group, so that representations are just vector spaces and we have an equivalence $$\Rep_q\GL_0\simeq\Vect.$$ Moreover, if $n=0$, then $\delta=0$ and the dimension relation $$\MyFigure{\bknot}=0$$ in $\overline{\Rep_q}\GL_0$ implies that the identity $\includegraphics[scale=0.15, valign=c]{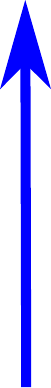}$ is a negligible morphism. Therefore, all nonempty diagrams are zero in $\bigslant{\overline{\Rep_q}\GL_n}{\mathcal{N}}$ and we obtain a $\mathbb{C}(q)$-linear functor $$\mathcal{E}\colon\RepqG\to\bigslant{\overline{\Rep_q}\GL_0}{\mathcal{N}}\simeq\Vect$$ annihilating all non-empty diagrams. If $S$ is any oriented surface, this clearly extends to a functor $$\mathcal{E}^S\colon\Skg\to\Vect.$$

\begin{proposition}
The linear map $$\varepsilon\colon\skg\to\mathbb{C}(q)$$ obtained by restricting $\mathcal{E}^S$ to $\skg$ provides a counit for the Turaev coproduct \eqref{turaev_coprod}.
\end{proposition}

\begin{proof}
Recall the local relations \eqref{box1}--\eqref{box4}. Since $\varepsilon$ annihilates non-empty diagrams and maps $q^t$ to $1$, the map $(1\otimes\varepsilon)\circ\Delta_f$ just switch $\bsquare$ with $\gsquare$. The same happens when $\varepsilon$ is applied to the second factor, so it is a counit for $\Delta_f$.
\end{proof}

\small
\newcommand{\etalchar}[1]{$^{#1}$}

\end{document}